\documentclass[reqno,11pt]{amsart}
\usepackage{amssymb, amsmath,latexsym,amsfonts,amsbsy, amsthm}
\usepackage{array,booktabs,mathrsfs}
\usepackage{xcolor}

\allowdisplaybreaks[4]

\newcommand{\ud}{\mathrm{d}}

\theoremstyle{plain}
\newtheorem{theorem}{Theorem}[section]
\newtheorem{lemma}[theorem]{Lemma}
\newtheorem{proposition}[theorem]{Proposition}

\theoremstyle{definition}

\theoremstyle{remark}
\newtheorem{remark}[theorem]{Remark}

\makeatletter
\renewcommand\subsubsection{\@startsection{subsubsection}{3}{\z@}%
	{0.7\linespacing\@plus0.3\linespacing}%
	{0.3\linespacing}%
	{\normalfont\itshape}}
\makeatother

\begin{document}
\title[Prandtl--Batchelor and flux-expulsion selection]
{Prandtl--Batchelor and flux-expulsion selection for steady MHD flows in a disk}

\author{Chen Gao}
\address{School of Mathematical Sciences, University of Science and Technology of China, Hefei, 230026, China}
\email{gaochen@amss.ac.cn}

\author{Zhiwu Lin}
\address{School of Mathematical Sciences, Fudan University,
Shanghai 200433, China}
\email{zwlin@fudan.edu.cn}

\author{Jianfeng Zhao}
\address{School of Mathematics and Information Science \& Center for Applied Mathematics of Guangxi (Guangxi University) \& Post-doctoral Research Station of the First-level Discipline in Mathematics, Guangxi University, Nanning, 530004, China}
\email{zhaojianfeng@amss.ac.cn}

\date{}
%%%%%%%%%%%%%%%%%%%%%%%%%%%%%%%%%%%%%%%%%%%%%%
%%%%%%%%%%%%%%%%%%%%%%%%%%%%%%%%%%%%%%%%%%
\begin{abstract}
We study the simultaneous vanishing-viscosity and vanishing-resistivity limit
of steady incompressible MHD flows in a disk.  The boundary velocity is a
small nonaxisymmetric perturbation of a rigid rotation with mean angular
speed \(\alpha\), while the prescribed tangential magnetic trace has mean
\(\beta\).  Assuming \(\alpha\neq0\) and the non-Alfv\'enic condition
\(|\alpha|\neq|\beta|\), we construct solutions converging on compact
interior subdisks to a rigidly rotating ideal MHD core with constant
vorticity and out-of-plane current density.  A new MHD--Wood law determines
the two core rotations from the boundary data: the velocity core is selected
by a coupled kinetic--magnetic balance, whereas the magnetic core is fixed by
the imposed mean circulation.  Consequently, zero circulation gives complete
interior magnetic expulsion, while nonzero circulation leaves a uniform
magnetic rotation after the nonaxisymmetric modes are confined to a thin
boundary layer.  This provides a fully coupled realization of
Prandtl--Batchelor selection and flux expulsion; unlike classical kinematic
models, the magnetic field actively changes the flow and need not be weak.
The proof combines a non-Alfv\'enic coercive theory for a periodic MHD
boundary layer, global matching of the two fields, and a coupled stability
estimate adapted to the magnetic boundary condition.  A separate conditional
rigidity argument, using exact viscous identities and local convergence but no
interior asymptotic expansion, explains the same core structure for a broader
single-eddy family.
\end{abstract}
\maketitle

\numberwithin{equation}{section}

\section{Introduction}

Two classical selection mechanisms arise in weakly dissipative flows with
closed streamlines.  The Prandtl--Batchelor principle predicts that weak
viscosity homogenizes vorticity in a closed eddy and selects a
constant-vorticity Euler core \cite{prandtl,B,W}.  Magnetic flux expulsion
predicts that rapid circulation and weak resistive diffusion remove the
nonconstant magnetic flux from such an eddy and concentrate it in thin
boundary or separatrix layers \cite{weiss66,Parker66,moffatt78,MoffattKamkar83}.
Both mechanisms select an ideal interior state through a singular
transport--diffusion balance, but they have usually been studied separately.

The separation between the two classical mechanisms is natural in kinematic
MHD, where the velocity is prescribed.  It fails in the fully coupled
equations: the Lorentz force changes the advecting eddy, while the eddy
simultaneously transports and stretches the magnetic field.  Vorticity
homogenization and flux expulsion must therefore be resolved together.  This
leads to a basic question: does a closed MHD eddy possess a joint inviscid
selection mechanism, and which parts of the two limiting fields are selected
by the boundary layer and which are fixed by the boundary data?

We answer this question for steady viscous--resistive incompressible MHD in a
disk.  Let \(\mathbf t\) be the positively oriented unit tangent.  We prescribe
\[
\mathbf u^\varepsilon\big|_{\partial B_1}
=(\alpha+\eta f(\theta))\mathbf t,
\qquad
\mathbf B^\varepsilon\cdot\mathbf t\big|_{\partial B_1}
=\beta+\eta\widehat f(\theta),
\]
where \(f,\widehat f\) are smooth zero-mean \(2\pi\)-periodic functions and
\(0<\eta\ll1\).  As viscosity and magnetic diffusivity vanish with fixed ratio
\(\kappa>0\), we
{\color{black}construct a coupled MHD--Prandtl boundary layer
expansion, prove its nonlinear stability, and obtain solutions converging on
every compact interior subdisk to}
\begin{equation}\label{intro:selected-core}
        \mathbf u_e=ar\,\mathbf e_\theta,\qquad
        \mathbf B_e=br\,\mathbf e_\theta .
\end{equation}
\begin{samepage}
Thus the limiting vorticity and the out-of-plane current
density are both constant.  We derive the following selection law, which generalizes the classical
Wood formula in Prandtl--Batchelor theory to the MHD setting:
\begin{equation}\label{intro:wood-formula}
b=\beta,\qquad
a^2=\alpha^2+\frac{\eta^2}{2\pi}\int_0^{2\pi}
\left(f^2-\kappa^{-1}\widehat f^{\,2}\right)\,\ud\theta.
\end{equation}
\end{samepage}

The identity \(b=\beta\) gives a precise form of flux expulsion.  When the
imposed circulation vanishes, the magnetic field converges to zero throughout
the interior core.  When it does not vanish, the nonaxisymmetric magnetic
modes are still confined to the boundary layer, but the globally constrained
circulation mode survives as the uniform magnetic rotation in
\eqref{intro:selected-core}.  Thus expulsion is selective rather than
absolute.  The disk provides a canonical single-eddy setting in which the local
closed-streamline mixing mechanism and the global boundary-circulation
constraint can be separated cleanly.

The classical flux-expulsion configurations also have a common global
feature.  Weiss \cite{weiss66} and Parker \cite{Parker66} start from a
uniform magnetic field; the circular-eddy calculation of Moffatt and Kamkar
uses the potential \(A=B_0r\sin\theta\), again representing a uniform field,
and their model shear problem uses a unidirectional periodic field
\cite{MoffattKamkar83}.  {The point-vortex analysis \cite{Bajer98} and recent
dynamical vortex studies \cite{GMT,TessierPH} likewise use an initially uniform
background field.}  Such a current-free field has zero circulation around the boundary
$\Gamma$ of a simply connected eddy. Indeed, letting $D$ denote the
simply connected region enclosed by $\Gamma$, Stokes' theorem yields
\[
\oint_\Gamma \mathbf{B}_0\cdot\mathbf{t}\,\ud s
=\int_D \operatorname{curl}\mathbf{B}_0\,\ud x
=0.
\]
For a circular boundary this is exactly the zero mean of the tangential
trace.  Consequently, the steady zero-circulation case of the theorem gives
a rigorous coupled counterpart of the classical prediction: the magnetic
field is expelled from every compact subset of the eddy.  We are not aware
of an earlier study of flux expulsion in the same steady, fully coupled,
closed-eddy boundary-value setting with a prescribed nonzero magnetic
circulation.  Nonzero circulation is qualitatively different: it is sustained
by the boundary data (equivalently, by a net current through the disk) and
cannot disappear from the selected steady core.

To our knowledge, this gives the first rigorous construction in the precise
steady, fully coupled, closed-streamline regime considered here.  It is neither
a passive-field limit nor a small-magnetic-field result: the mean magnetic
field may be comparable with the mean velocity and may also vanish. 

\medskip
\noindent\emph{Relation with Prandtl--Batchelor theory.}
The classical principle \cite{prandtl,B,W} was developed through
high-Reynolds-number asymptotics for finite domains and closed-streamline
flows \cite{feymann-lagerstrom,lagetstrom75,riley81}, followed by asymptotic
and existence analyses of cylindrical eddies
\cite{K1998,K2000,van-wijngaarden}.
Fei, Gao, Lin, and Tao rigorously constructed the disk flow
\cite{FGLTdisk}, where a scalar Batchelor--Wood formula fixes the rigidly
rotating Euler core; {the annular problem and its generalized
two-boundary selection rule were treated in \cite{FGLTannulus}.}

The hydrodynamic Wood formula in the disk is
\[
 a_{\rm hyd}^2=\frac1{2\pi}\int_0^{2\pi}
 (\alpha+\eta f)^2\,\ud\theta
 =\alpha^2+\frac{\eta^2}{2\pi}\int_0^{2\pi}f^2\,\ud\theta.
\]
It comes from one scalar averaged identity.  The MHD layer instead preserves
two independent averaged quantities across the layer: one fixes the mean
magnetic field, while the other balances kinetic and magnetic contributions.
Solving the two relations jointly fixes \(b\) and \(a^2\), introduces the
diffusivity ratio \(\kappa\), and shows
that the oscillatory magnetic trace lowers the selected \(a^2\) by
\((2\pi\kappa)^{-1}\eta^2\int\widehat f^{\,2}\), while the velocity
oscillation raises it.  The mean \(\beta\) survives as \(b=\beta\) and affects
the non-Alfv\'enic solvability condition, although it cancels from the final
formula for \(a^2\).  This two-field balance and the simultaneous selection
of a dynamically adjusted velocity core and a circulation-constrained
magnetic core are the main novelties of the MHD--Wood formula.

Rigorous steady Prandtl expansions for prescribed outer flows were
established for moving-plate and general steady boundary-layer problems
\cite{GN,GI1,GZ}; later works treat non-shear Euler flows, convergent channels,
and global stationary inviscid limits \cite{GZ2,GX,IMPi}.  Here the closed
geometry has no inflow boundary and, more importantly, the outer ideal state
is itself an unknown selected by the layer.

\medskip
\noindent\emph{Relation with flux expulsion and magnetic transport.}
Weiss's prescribed-eddy calculation \cite{weiss66} and the time-scale analysis
of Moffatt and Kamkar \cite{MoffattKamkar83} form the classical kinematic
picture; see also \cite{moffatt78,Bajer98,dormy-moffatt-19}.  Once the Lorentz
force is retained, the concentrated field can reorganize or suppress the eddy.
{Such nonlinear back reaction appears in the self-consistent
magnetoconvection calculation of Galloway, Proctor, and Weiss
\cite{GallowayProctorWeiss78}, the flux-rope study of Proctor and Galloway
\cite{ProctorGalloway79}, and the reduced or computational vortex models
\cite{GMT,TessierPH}.  A genuine runaway feedback occurs in the channel model
of Kamkar and Moffatt \cite{KamkarMoffatt82}.}  Nonlinear magnetoconvection likewise segregates
magnetic field into strong intercell structures and weak-field interiors,
with feedback on the circulation and, in stellar computations, flux
separation \cite{ProctorWeiss82,Weiss81a,Weiss81b,
TaoWeissBrownjohnProctor98,WeissProctorBrownjohn02}.  {These studies are
evolutionary, asymptotic, reduced, or numerical.  In the uniform or current-free
closed-eddy configurations above, the imposed field has zero circulation around
each simply connected cell.  The channel and magnetoconvection models have
different global field constraints and do not address the prescribed nonzero
boundary-circulation mode considered here.}  Our disk problem removes
buoyancy, stratification, and three-dimensional geometry to isolate a steady
coupled selection mechanism and to identify the nonzero circulation mode that
diffusion does not remove.

{The related transport analogy captures potential-vorticity
homogenization, passive-scalar mixing, and enhanced diffusion}
\cite{rhines-young82,rhines-young83,FannjiangPapanicolaou94,
ConstantinKiselevRyzhikZlatos08}.  It does not capture the present leading-order
Lorentz feedback, because both the advecting velocity and the transported
magnetic flux are unknown.

\medskip
\noindent\emph{Relation with MHD boundary layers.}
Wall-bounded MHD layers go back to Hartmann's theory
\cite{Hartmann37} and the classical accounts
\cite{Shercliff65,Moreau90,Davidson01}.  Effective conditions for thin
Hartmann layers and systematic derivations of Hartmann, Shercliff, and
magnetic-Prandtl models were given in
\cite{PotheratSommeriaMoreau02,GVPrestipino}.  The layer considered here is
not a standard Hartmann layer: the imposed field is tangential, the normal
magnetic component is not prescribed, and the resulting boundary layer is
characteristic.

For unsteady MHD boundary layers, Liu, Xie, and Yang
\cite{LXY1,LXY2} established Sobolev well-posedness and justified the
Prandtl ansatz under a nondegenerate tangential magnetic field.  Related
theories without resistivity or without viscosity appear in
\cite{LWXY,LiXu}, and strong layers for inhomogeneous incompressible MHD were
treated in \cite{LiXie}.  {\color{black}Liu, Yang, and Zhang \cite{LYZ} study
the Sobolev stability of prescribed Prandtl-type shear flows for a
periodically forced steady MHD system, under a nondegenerate tangential
magnetic field and a magnetic-dominance condition;} see also the
arbitrary-order rectangle construction \cite{BianSi24}.  These works prescribe the outer ideal flow and use a plate
or rectangle with a propagation direction.  By contrast, our angular variable
is periodic, the outer ideal state is selected globally, the mean magnetic
field may vanish, and coercivity follows from non-Alfv\'enic coupling rather
than magnetic dominance.

We consider the steady incompressible MHD system in the unit disk \(B_1(0)\):
\begin{equation}\label{ns}
\left\{
\begin{aligned}
	&\mathbf{u}^\varepsilon\cdot\nabla\mathbf{u}^\varepsilon-\mathbf{B}^\varepsilon\cdot\nabla\mathbf{B}^\varepsilon+\nabla p^\varepsilon-\kappa\varepsilon^2\Delta\mathbf{u}^\varepsilon=\mathbf{0},\\
	&\mathbf{u}^\varepsilon\cdot\nabla \mathbf{B}^\varepsilon-\mathbf{B}^\varepsilon\cdot\nabla \mathbf{u}^\varepsilon-\varepsilon^2\Delta \mathbf{B}^\varepsilon=0,\\
	&\nabla\cdot \mathbf{u}^\varepsilon=0,\qquad \nabla\cdot \mathbf{B}^\varepsilon=0.
\end{aligned}
\right.
\end{equation}
We impose the perturbed rotating velocity condition
\begin{equation}\label{velocity-boundary-condition}
\mathbf{u}^\varepsilon\big|_{\partial B_1}
=(\alpha+\eta f(\theta))\mathbf{t},
\end{equation}
and the prescribed tangential magnetic trace
\begin{align}\label{magnetic-boundary-condition}
\mathbf{B}^\varepsilon \cdot \mathbf{t}\big|_{\partial B_1}
=\beta+\eta \widehat f(\theta).
\end{align}
Here \(\kappa>0\) is fixed.  The inverse kinetic and magnetic Reynolds numbers
are \(\kappa\varepsilon^2\) and \(\varepsilon^2\), respectively.  The vector
\(\mathbf t\) is the positively oriented unit tangent, \(\alpha\neq0\) and
\(\beta\) are the mean boundary fields, and \(\eta\) measures the size of the
smooth \(2\pi\)-periodic perturbations \(f,\widehat f\).  With the convention
\(\mathbf B^\varepsilon=\nabla^\perp\Pi^\varepsilon\), prescribing the
tangential magnetic trace is exactly the scalar Neumann condition
\[
 -\Pi^\varepsilon_r(1,\theta)=\beta+\eta\widehat f(\theta).
\]
The additive constant of \(\Pi^\varepsilon\) is fixed separately.  Thus no
Dirichlet condition on \(\Pi^\varepsilon\), and no independent boundary value
of the normal component \(h^\varepsilon\), is needed for the scalar induction
problem.  This externally maintained tangential trace is different from the
usual perfectly conducting-wall condition and from a complete insulating-wall
matching problem; it may be realized by coupling to a compatible exterior
vacuum field and external currents.  Its mean part is the harmonic
circulation component \(\beta r^{-1}\mathbf e_\theta\), whereas a compatible zero-mean part is
an exterior harmonic multipole decaying at infinity.  For the zero-mean
perturbations in the theorem,
\[
 \frac{1}{2\pi}\int_{\partial B_1}\mathbf B^\varepsilon\cdot\mathbf t\,\ud s
 =\beta .
\]
The case \(\beta\neq0\) is understood as an externally maintained field, not
as an isolated finite-energy magnetic field in the whole plane.  In magnetic
potential variables, integration of the induction equation identifies a
constant \(2\beta\varepsilon^2\); Section~2 gives the derivation and the full
polar formulation.  This global identity independently confirms the exact
selection \(b=\beta\) obtained from the MHD--Wood formula.

Write
\[
\mathbf u^\varepsilon=u^\varepsilon\mathbf e_\theta
 +v^\varepsilon\mathbf e_r,\qquad
\mathbf B^\varepsilon=g^\varepsilon\mathbf e_\theta
 +h^\varepsilon\mathbf e_r,\qquad
\Omega=\mathbb T\times(0,1).
\]
Motivated by the Prandtl--Batchelor selection principle, we seek inviscid MHD
cores with constant vorticity and constant out-of-plane current
density.  In the boundary-tangent single-eddy class relevant to the disk,
these are precisely the Couette-type ideal MHD fields
\begin{align}\label{taylor-couette-flow}
(u_e(\theta,r),v_e(\theta,r),g_e(\theta,r),h_e(\theta,r))=(ar,0,br,0),
\end{align}
where \(a\) and \(b\) are constants.
Appendix~\ref{app:mhd-pb-selection} gives a conditional MHD
Prandtl--Batchelor argument based directly on the magnetic-potential equation.
For a single ideal eddy, the ideal induction equation makes the magnetic field
parallel to the velocity.  A non-Alfv\'enic change of stream function reduces
the ideal momentum equation to a semilinear elliptic equation.  Radial
symmetry in the disk, together with identities obtained by testing the
finite-\(\varepsilon\) magnetic-potential and vorticity equations by functions
of the velocity stream function, then gives constant current and constant
vorticity.
The theorem itself constructs a family with the selected core without
assuming this streamline structure.

The constants $a$ and $b$ are not imposed as outer data.  Nevertheless, the
MHD--Wood identities force \(b\) and \(a^2\) to have the explicit values in
\eqref{intro:wood-formula}; \(a\) is the root that converges to \(\alpha\) as
\(\eta\to0\).  Thus the boundary layer selects the core through two averaged
conservation laws coupling the velocity and magnetic traces.

Informally, the theorem says that every sufficiently small mean-zero
nonaxisymmetric perturbation of the two boundary traces generates only an
\(O(\varepsilon)\)-thin MHD boundary layer.  Away from that layer, the
vorticity and out-of-plane current density converge uniformly to
constants.  The velocity constant is the explicit quadratic balance in
\eqref{intro:wood-formula}, while the magnetic constant is the imposed
circulation.  The conditional
selection argument in Appendix~\ref{app:mhd-pb-selection} explains the same
constant-vorticity/constant-current structure under a single-eddy hypothesis,
non-Alfv\'enicity on the closed eddy, and a uniformly bounded family converging
locally in \(C^1\).

We next introduce the steady MHD--Prandtl equations near $r=1$:
\begin{equation}\label{prandtl-problem1}
\left\{
\begin{aligned}
&u_p\partial_\theta u_p+v_p\partial_Yu_p-g_p\partial_\theta g_p-h_p\partial_Yg_p-\kappa\partial_{YY}u_p=0,\\
&\partial_\theta g_p+\partial_Y\left(\frac{v_p g_p}{u_p}\right)-\partial_Y\left(\frac{\partial_Yg_p}{u_p}\right)=0,\\
&\partial_\theta u_p+\partial_Yv_p=0,\qquad \partial_\theta g_p+\partial_Yh_p=0,\\
&(u_p,v_p,g_p,h_p)(\theta,Y)=(u_p,v_p,g_p,h_p)(\theta+2\pi,Y),\\
&u_p|_{Y=0}=\alpha+\eta f(\theta),\qquad v_p|_{Y=0}=0,\qquad g_p|_{Y=0}=\beta+\eta \widehat f(\theta),\\
&\lim_{Y\to -\infty}(\partial_Yu_p,\partial_Yg_p)=(0,0).
\end{aligned}
\right.
\end{equation}
The constants \((a,b)\) are the far-field limits of \((u_p,g_p)\).  Every
sufficiently regular periodic solution of \eqref{prandtl-problem1} in the
weighted class constructed below, with \(u_p\geq c>0\) and constant far-field
limits, satisfies the MHD--Wood identities and hence
\eqref{intro:wood-formula}.  The sign is fixed by \(a\alpha>0\).  The
fixed-point argument constructs all Fourier modes of the solution; the
identities are applied only afterward to identify the two limits, so there is
no zero-mode solvability assumption hidden in the formula.

Our main theorem is stated as follows.
\begin{theorem}\label{main-theorem}
{
	Assume that \(\alpha\neq0\), \(|\alpha|\neq|\beta|\), and that
	\(f\) and \(\widehat f\) are smooth \(2\pi\)-periodic functions satisfying
	\(\int_0^{2\pi}f(\theta)\,\ud\theta
	=\int_0^{2\pi}\widehat f(\theta)\,\ud\theta=0\).
	Then there exist \(\varepsilon_0,\eta_0>0\) such that for
	\(0<\varepsilon<\varepsilon_0\) and \(0<\eta<\eta_0\), system
	\eqref{ns} with boundary conditions \eqref{velocity-boundary-condition}
	and \eqref{magnetic-boundary-condition} has a solution whose
	polar components \((u^\varepsilon,v^\varepsilon,g^\varepsilon,h^\varepsilon)\)
	satisfy
	\begin{align*}
		\left\|u^\varepsilon-u_e(r)-u_p\left(\theta,\frac{r-1}{\varepsilon}\right)+u_e(1)\right\|_{L^\infty(\Omega)}&\leq C\varepsilon,\\
		\|v^\varepsilon\|_{L^\infty(\Omega)}&\leq C\varepsilon,\\
		\left\|g^\varepsilon-g_e(r)-g_p\left(\theta,\frac{r-1}{\varepsilon}\right)+g_e(1)\right\|_{L^\infty(\Omega)}&\leq C\varepsilon,\\
		\|h^\varepsilon\|_{L^\infty(\Omega)}&\leq C\varepsilon.
	\end{align*}
	Here \(C\) is independent of \(\varepsilon\) and \(\eta\).
	The Couette-type ideal MHD field in \eqref{taylor-couette-flow} is given by
	$u_e(r)=ar$ and $g_e(r)=br$, with zero normal components.  The profiles
	$(u_p,v_p,g_p,h_p)$ solve \eqref{prandtl-problem1}.  The constants $a$ and
	$b$ are given by the following MHD--Wood formula:
	\[
	b=\beta,\qquad
	a^2=\alpha^2+\frac{\eta^2}{2\pi}
	\int_0^{2\pi}\left(f^2-\frac1\kappa\widehat f^{\,2}\right)\,\ud\theta.
	\]
	The value of \(a\) is determined by requiring \(a\) and \(\alpha\) to
	have the same sign.

	Moreover, for every \(0<r_0<1\),
	\[
		\lim_{\varepsilon\to0}\left(
		\|\omega^\varepsilon-2a\|_{L^\infty(B_{r_0})}
		+\|j^\varepsilon-2b\|_{L^\infty(B_{r_0})}\right)=0,
	\]
	where $\omega^\varepsilon$ and $j^\varepsilon$ denote the vorticity and
	current density, respectively.
}
\end{theorem}

\begin{remark}
	The zero-mean assumption is only a normalization.  It separates the mean
	boundary traces from the oscillatory data, but this separation does not by
	itself produce a selection law.  Wood-type identities are generally
	nontrivial, and coupled multi-field models usually do not have an explicit
	Wood formula.
\end{remark}

Diffusion removes the oscillatory magnetic component but
leaves the uniform rotational field \(b r\,\mathbf e_\theta\).  Thus the theorem
combines Prandtl--Batchelor selection with flux expulsion.

The condition \(|\alpha|\neq|\beta|\) is neither a small-field nor
a magnetic-dominance assumption.  In the Els\"asser variables
\(\mathbf z^\pm=\mathbf u\pm\mathbf B\), introduced by Els\"asser
\cite{Elsasser50}, the selected core satisfies
\[
        \mathbf z_e^\pm=(a\pm b)r\,\mathbf e_\theta.
\]
If \(a=\pm b\), one of these fields vanishes.  These are the nonlinear
Alfv\'enic states associated with the Wal\'en relation
\(\mathbf u=\pm\mathbf B\) (in Alfv\'en units)
\cite{Alfven42,Elsasser50,Chandrasekhar61,moffatt78}.  Physically, ideal MHD
nonlinearity is an interaction between counter-propagating Els\"asser fields;
when one field vanishes, the characteristic structure changes and the usual
two-field interaction degenerates.  The same degeneration is visible here at
the steady boundary-layer level: the choice \(a=\pm b\) makes the factor
\(a^2-b^2\) vanish in the periodic MHD--Prandtl estimate, the Euler matching,
and the final stability energy.  Since \(a=\alpha+O(\eta^2)\) and \(b=\beta\),
the assumption \(|\alpha|\neq|\beta|\) gives
\(|a^2-b^2|\geq c_0>0\) when \(\eta\) is sufficiently small.  This is the
non-Alfv\'enic condition used below.  The Alfv\'enic case is therefore a
distinct characteristic regime, not an endpoint reached by the present
estimates.

\medskip
\noindent\emph{Three main points of the proof.}
The proof is not obtained by appending a passive induction equation to the
hydrodynamic disk construction.  {In the disk and annulus constructions
\cite{FGLTdisk,FGLTannulus}, one velocity stream function leads, after a von
Mises change of variables, to scalar leading Prandtl equations and scalar Wood
relations.}  In the moving-boundary theories of Guo--Nguyen and
Guo--Iyer \cite{GN,GI1}, the outer Euler flow is prescribed and the
tangential variable supplies a propagation direction.  None of these
features survives unchanged here.  The velocity and magnetic layers act at
the same order, their two normal components are nonlocal, the Lorentz force
enters the leading momentum balance, and the induction equation supplies
only a Neumann condition for the magnetic potential.  Moreover, the angular
variable is periodic and the two outer constants are outputs of the global
matching.  Thus the coupled layer must be solved as a two-field elliptic-type
problem; its magnetic far-field mode must satisfy an additional compatibility
condition; and the final error estimate must recover magnetic tangential
control that the boundary conditions do not provide.  The following three
steps address these new difficulties.

\smallskip
\noindent\emph{1. Coercivity for a periodic steady MHD--Prandtl problem.}
In an ordinary steady Prandtl problem the tangential variable plays the role of
time, and data are propagated from an inflow boundary.  On the circle there is
no inflow direction; the forward-parabolic interpretation is unavailable, and
steady periodic Prandtl equations can even exhibit nonexistence phenomena
\cite{renardy}.  In addition, both normal components in
\eqref{prandtl-problem1} are nonlocal functions of the tangential variables.
A direct elimination either obscures the boundary conditions or loses a
tangential derivative.

We combine the magnetic-flux identity with a von Mises transformation and use
\(w_p=u_p^2\) together with \(g_p\) as the tangential unknowns.  At the
constant state \((\alpha^2,\beta)\), the principal tangential combinations are
\[
 \partial_\theta w-2\beta\,\partial_\theta g,
 \qquad
 \partial_\theta g-\frac{\beta}{2\alpha^2}\partial_\theta w .
\]
The corresponding matrix
\[
 \mathcal M_{\alpha,\beta}
 =\begin{pmatrix}
 1&-2\beta\\[2pt]
 -\dfrac{\beta}{2\alpha^2}&1
 \end{pmatrix},
 \qquad
 \det\mathcal M_{\alpha,\beta}
 =1-\frac{\beta^2}{\alpha^2},
\]
is invertible precisely away from the Alfv\'enic regime.  Multipliers aligned
with these two combinations cancel the mixed diffusion terms after integration
by parts and give coercive control of the periodic derivatives.  Weighted
estimates on the half-strip then control the decaying parts of the profiles
without forcing their constant limits at \(Y=-\infty\) to vanish.  This is the
main point: MHD coupling creates an elliptic-type coercive structure
for a periodic steady layer, rather than acting as a lower-order perturbation
of a hydrodynamic evolution problem.

\smallskip
\noindent\emph{2. MHD--Wood selection and global two-field matching.}
In the moving-plate Euler--Prandtl schemes \cite{GN,GI1}, the outer Euler
flow is prescribed and the boundary layer repairs its tangential trace.  In
the hydrodynamic disk \cite{FGLTdisk}, the outer rotation is unknown and is
fixed by one scalar Batchelor--Wood condition.  In the annulus
\cite{FGLTannulus}, the selected rotating shear solves an Euler ordinary
differential equation with scalar Wood conditions at the two boundary
components.  In both cases only the velocity field is involved.  Here the leading von Mises system yields two averaged identities whose
far-field consequences give the selection law
\eqref{intro:wood-formula}, thereby determining both outer constants.
The nonlinear fixed point first constructs the periodic boundary layer and its
constant far-field limits; only afterward are the identities evaluated to
identify \(a\) and \(b\).  This avoids a circular solvability argument.

The higher-order matching is also coupled.  The
two divergence equations determine the normal profiles, and already at the
first correction order the far-field limit of the induction relation gives
\[
        -a\,h_\infty+b\,v_\infty=0.
\]
Thus the Prandtl and Euler corrections must be constructed alternately rather
than by solving a velocity problem first and then adding a passive magnetic
field.  At higher orders one obtains analogous coupled solvability conditions
with lower-order source terms.  After eliminating the ideal magnetic
variables, the Euler correction is governed by the same non-Alfv\'enic factor
\(a^2-b^2\) that appears in the leading Prandtl analysis.  Section~2 verifies
the first-order compatibility, proves one reusable higher-order solvability
statement, and uses the exact
scalar identity
\[
 \mathbf u^\varepsilon\!\cdot\nabla\Pi^\varepsilon
 -\varepsilon^2\Delta\Pi^\varepsilon=2\beta\varepsilon^2
\]
to recover \(b=\beta\) and the magnetic compatibility needed by the recursion.

\smallskip
\noindent\emph{3. A coupled stability norm that replaces missing magnetic
boundary control.}
Let \(\Phi\) and \(\Pi\) be the velocity and magnetic error potentials.  Their
boundary conditions are
\[
        \Phi=\Phi_r=0,\qquad \Pi_r=0\qquad\text{at }r=1.
\]
The magnetic potential has no Dirichlet condition: only its additive constant
is normalized.  Consequently the clamped velocity estimate controls neither
the tangential magnetic trace \(\Pi_\theta/r\) nor its interaction with the
radial zero mode.  The magnetic diffusion identity alone gives
\(\varepsilon^2\nabla^2\Pi\), which is two powers of \(\varepsilon\) too weak
for the leading Lorentz terms in the momentum equation.  It also produces a
boundary term that cannot be removed by imposing \(\Pi=0\), since that would
change the physical Neumann problem.  Estimating the velocity and magnetic
equations separately therefore either loses a power of \(\varepsilon\), loses
a tangential derivative, or silently adds a boundary condition not present in
the model.

We therefore first estimate the coupled quantity
\[
 G=a\Pi_\theta-b\Phi_\theta
   -\frac{\Pi^a_\theta}{r}\Phi_r^0
   +\frac{\Phi^a_\theta}{r}\Pi_r^0 ,
\]
where the last two terms remove the interaction with the radial zero modes.
The leading part of the induction equation is precisely
\(rG-\varepsilon^2r\Delta\Pi\).  Testing this identity at its natural strong
scale controls
\[
        \mathbb L=\varepsilon^{-2}\int_\Omega rG^2,
\]
which is two powers of \(\varepsilon\) stronger than the basic diffusive
estimate.  Thus \(G\), rather than \(\Pi_\theta\) alone, is the characteristic
magnetic quantity supplied by induction.  Substituting this information into
the momentum identity gives the
coercive factor
\[
        \left|a-\frac{b^2}{a}\right|
        =\frac{|a^2-b^2|}{|a|},
\]
which recovers the missing tangential control of \(\Pi\) and explains exactly
why non-Alfv\'enicity is the correct hypothesis.  The stability norm is built
from three distinct scales:
\[
\begin{aligned}
\mathbb E&\sim\varepsilon^2
 \int_\Omega r\big(|\nabla^2\Phi|^2+|\nabla^2\Pi|^2
 +|\nabla^2\Phi_\theta|^2+|\nabla^2\Pi_\theta|^2\big),\\
\mathbb P&\sim\int_\Omega
 \left(r\Phi_{r\theta}^2+\frac{\Phi_{\theta\theta}^2}{r}
 +r\Pi_{r\theta}^2+\frac{\Pi_{\theta\theta}^2}{r}\right),\qquad
\mathbb L=\varepsilon^{-2}\int_\Omega rG^2.
\end{aligned}
\]
Here \(\mathbb E\) controls the viscous and resistive Hessians, including the
derivatives needed to interpret the polar center; \(\mathbb P\) supplies
unweighted tangential positivity for both fields; and \(\mathbb L\) gives the
strong characteristic control needed to transfer information from the
Neumann magnetic equation to the velocity equation.  None of the three can be
dropped: \(\mathbb E\) alone is too weak in \(\varepsilon\), \(\mathbb P\)
does not see the magnetic zero mode, and \(\mathbb L\) controls only one
coupled combination.  The algebraic inversion with determinant
\(a^2-b^2\) closes the triangle among them.

{
The refined estimate has a different difficulty.  The largest terms contain
the radial derivatives of the Prandtl vorticity and current, for example
\[
 \int \Phi_\theta(\Delta\Phi_p^a)_r
 \frac{r\Phi_{\theta\theta}}{u_e+u_p^a},\qquad
 \int \Pi_\theta(\Delta\Pi_p^a)_r
 \frac{r\Phi_{\theta\theta}}{u_e+u_p^a}.
\]
The first term can be estimated by Hardy's inequality because both
\(\Phi_\theta\) and \(\Phi_{\theta\theta}\) vanish on the boundary.  This
argument does not apply to the second
term, since \(\Pi_\theta\) has no boundary condition.  Using only one factor
of \(1-r\) loses a power of \(\varepsilon\), while assigning both factors to
\(\Phi_{\theta\theta}\) requires one more radial derivative of \(\Phi\).

For this term, we use the second magnetic-potential error equation directly.
We keep \(u_e+u_p^a\) as the leading coefficient of \(\Pi_\theta\) and move
all higher-order coefficients to the right-hand side.  The resulting
\(\varepsilon^2r\Delta\Pi\) term supplies the missing power of
\(\varepsilon\), and integration by parts in \(\theta\) avoids the extra
radial derivative.  The remaining terms are bounded by
\(C\eta(\mathbb E+\mathbb P+\mathbb L)\).  This is the step that prevents
derivative loss in the refined estimate.
}

The remainder of the paper is organized as follows.  Section~2 derives the
magnetic-potential and polar formulations, constructs the leading and
higher-order MHD--Prandtl profiles, matches them with global Euler corrections,
and estimates the residual.  Section~3 proves the combination, energy,
positivity, and higher-order linear stability estimates.  Section~4 solves
the nonlinear error equations.  Section~5 completes the proof of
Theorem~\ref{main-theorem}, including the interior convergence of the
vorticity and current density.  {Appendix~A constructs the velocity
divergence corrector, and Appendix~B gives the conditional MHD
Prandtl--Batchelor selection argument.}

\section{Construction of approximate solutions}

\subsection{Magnetic-potential and polar formulations}

We first record the scalar magnetic-flux equation and the polar system used in
the construction.  We fix the convention
\[
\nabla^\perp \Pi=\frac{1}{r}\Pi_\theta\,\mathbf e_r-\Pi_r\,\mathbf e_\theta,
\qquad
\mathbf B=h\,\mathbf e_r+g\,\mathbf e_\theta=\nabla^\perp\Pi,
\]
so that \(g=-\Pi_r\) and \(h=\Pi_\theta/r\).  For scalar functions, we use
the polar Laplacian
\[
\Delta=\partial_{rr}+\frac1r\partial_r+\frac1{r^2}\partial_{\theta\theta}.
\]
For \(\mathbf B^\varepsilon\), define
\[
\Pi^\varepsilon(\theta,r)=\int_r^1 g^\varepsilon(\theta,\rho)\,\ud\rho
+\mathcal I\big[h^\varepsilon(\cdot,1)\big](\theta),
\]
where \(\mathcal I\) is the zero-mean periodic primitive in \(\theta\).
This is well-defined because the divergence constraint gives
\(\int_0^{2\pi}h^\varepsilon(\theta,1)\,\ud\theta=0\).  With this convention,
\(\Pi_r^\varepsilon=-g^\varepsilon\) and
\(\Pi_\theta^\varepsilon=rh^\varepsilon\), and the induction equation becomes
\[
\mathbf{u}^\varepsilon\cdot\nabla \mathbf{B}^\varepsilon
-\mathbf{B}^\varepsilon\cdot\nabla \mathbf{u}^\varepsilon
-\varepsilon^2\Delta \mathbf{B}^\varepsilon
=\nabla^\perp\!\left(\mathbf{u}^\varepsilon\cdot\nabla\Pi^\varepsilon
-\varepsilon^2\Delta\Pi^\varepsilon\right)=0.
\]
Thus,
\begin{equation}\label{eq:magnetic-flux-scalar}
\mathbf{u}^\varepsilon\cdot\nabla\Pi^\varepsilon
-\varepsilon^2\Delta\Pi^\varepsilon=C_0
\end{equation}
for a constant \(C_0\).  Since
\(\mathbf u^\varepsilon\cdot\mathbf n=0\) on \(\partial B_1\), integration
over the disk gives
\begin{equation*}
-\varepsilon^2\int_0^{2\pi}\Pi_r^\varepsilon(\theta,1)\,\ud\theta
=\int_{B_1}C_0\,\ud S.
\end{equation*}
Using
\(-\Pi_r^\varepsilon(\theta,1)=g^\varepsilon(\theta,1)
=\beta+\eta\widehat f(\theta)\) and the zero mean of \(\widehat f\), we find
\begin{equation}\label{eq:magnetic-flux-constant}
C_0=2\beta\varepsilon^2.
\end{equation}
This calculation gives at the PDE level the circulation constraint and
independently agrees with the MHD--Wood identity \(b=\beta\).

Writing \((u^\varepsilon,v^\varepsilon)\) and
\((g^\varepsilon,h^\varepsilon)\) for the tangential and radial components of
the velocity and magnetic fields, respectively, the momentum and divergence
equations, together with \(-r\) times
\eqref{eq:magnetic-flux-scalar}, give
\begin{equation}\label{mhd-polar}
\left \{
\begin {array}{lll}
u^\varepsilon u^\varepsilon_\theta+rv^\varepsilon u^\varepsilon_r+u^\varepsilon v^\varepsilon-g^\varepsilon g^\varepsilon_\theta-rh^\varepsilon g^\varepsilon_r-g^\varepsilon h^\varepsilon+p^\varepsilon_\theta-\kappa\varepsilon^2\big(ru^\varepsilon_{rr}
+ u^\varepsilon_{r}+\frac{u^\varepsilon_{\theta\theta}}{r}+\frac{2}{r} v^\varepsilon_\theta-\frac{ u^\varepsilon}{r}\big)=0,\\[5pt]
u^\varepsilon v^\varepsilon_\theta+rv^\varepsilon v^\varepsilon_r-(u^\varepsilon)^2-g^\varepsilon h^\varepsilon_\theta-rh^\varepsilon h^\varepsilon_r+(g^\varepsilon)^2 +rp^\varepsilon_r-\kappa\varepsilon^2\big(rv^\varepsilon_{rr}
+v^\varepsilon_{r}+\frac{ v^\varepsilon_{\theta\theta}}{r}-\frac{2}{r} u^\varepsilon_\theta-\frac{ v^\varepsilon}{r}\big)=0,\\[5pt]
-u^\varepsilon rh^\varepsilon+rv^\varepsilon g^\varepsilon-\varepsilon^2 ( rg^\varepsilon_{r}+g^\varepsilon-h^\varepsilon_\theta)=-2\beta\varepsilon^2r,\\[5pt]
u^\varepsilon_\theta+rv^\varepsilon_r+ v^\varepsilon=0,\\[5pt]
g^\varepsilon_\theta+rh^\varepsilon_r+ h^\varepsilon=0,\\[5pt]
u^\varepsilon(\theta,1)=\alpha+\eta f(\theta),\qquad v^\varepsilon(\theta,1)=0,\\[5pt]
g^\varepsilon(\theta,1)=\beta+\eta \widehat f(\theta),
\end{array}
\right.
\end{equation}
for \((\theta,r)\in\mathbb T\times(0,1)\).

We now construct an approximate solution of \eqref{mhd-polar} by matched
asymptotic expansion.  In addition to \(\Omega\), defined in the introduction,
we use the open half-strip
\[
	\Omega_p=\mathbb{T}\times(-\infty,0).
\]
The closures of these domains are used when boundary traces are discussed.
Unless otherwise specified, integrals over $\Omega$ and $\Omega_p$ are taken with respect to $\ud\theta\,\ud r$ and $\ud\theta\,\ud Y$, respectively. If a weighted measure such as $r\,\ud\theta\,\ud r$ or $r^{-1}\,\ud\theta\,\ud r$ is used, the weight will be displayed explicitly. Integrals over the disk $B_1$ are taken with respect to the standard area measure $\ud S$.
Throughout the paper, $C$ denotes a positive constant independent of $\varepsilon$ and $\eta$, and it may change from line to line. Constants depending on fixed indices or on $\delta$ are denoted by $C(m,l,k)$, $C_\delta$, and similar notation.
We write $\langle z\rangle=(1+z^2)^{1/2}$ for the standard polynomial weight.
Without loss of generality, we assume $\alpha>0$ in the proof below; the case $\alpha<0$ is reduced to this one by reversing the tangential orientation.

\subsection{Solvability of the Euler and Prandtl equations}

Away from the boundary, we make the following formal Euler expansions
\begin{align}\label{euler-expansions}
	\begin{aligned}
	u^{\varepsilon}(\theta,r)&=u_e^{(0)}(\theta,r)+\varepsilon u_e^{(1)}(\theta,r)+\cdots,\\
	v^{\varepsilon}(\theta,r)&=v_e^{(0)}(\theta,r)+\varepsilon v_e^{(1)}(\theta,r)+\cdots,\\
	g^{\varepsilon}(\theta,r)&=g_e^{(0)}(\theta,r)+\varepsilon g_e^{(1)}(\theta,r)+\cdots,\\
	h^{\varepsilon}(\theta,r)&=h_e^{(0)}(\theta,r)+\varepsilon h_e^{(1)}(\theta,r)+\cdots,\\
	p^{\varepsilon}(\theta,r)&=p_e^{(0)}(\theta,r)+\varepsilon p_e^{(1)}(\theta,r)+\cdots.
	\end{aligned}
\end{align}
We introduce the scaled variable $Y=\frac{r-1}{\varepsilon}\in (-\infty,0]$ and make the following MHD--Prandtl expansions near $r=1$:
\begin{align}\label{first-order-expansion}
	\begin{aligned}
		 & u^\varepsilon=u_e(1+\varepsilon Y)+u_p^{(0)}(\theta,Y)+\varepsilon\big[u_e^{(1)}(\theta,1+\varepsilon Y)+u_p^{(1)}(\theta,Y)\big]
		+\cdots,                                                                                                             \\[5pt]
		 & v^\varepsilon=v_p^{(0)}(\theta,Y)+\varepsilon\big[v_e^{(1)}(\theta,1+\varepsilon Y)+v_p^{(1)}(\theta,Y)\big]
		+\varepsilon^2\big[v_e^{(2)}(\theta,1+\varepsilon Y)+v_p^{(2)}(\theta,Y)\big]+\cdots,                                        \\[5pt]
		 & g^\varepsilon=g_e(1+\varepsilon Y)+g_p^{(0)}(\theta,Y)+\varepsilon\big[g_e^{(1)}(\theta,1+\varepsilon Y)+g_p^{(1)}(\theta,Y)\big]
		+\cdots,                                                                                                             \\[5pt]
		 & h^\varepsilon=h_p^{(0)}(\theta,Y)+\varepsilon\big[h_e^{(1)}(\theta,1+\varepsilon Y)+h_p^{(1)}(\theta,Y)\big]
		+\varepsilon^2\big[h_e^{(2)}(\theta,1+\varepsilon Y)+h_p^{(2)}(\theta,Y)\big]+\cdots,                                        \\[5pt]
		 & p^\varepsilon=p_e(r)+p_p^{(0)}(\theta,Y)+\varepsilon\big[p_e^{(1)}(\theta,r)+p_p^{(1)}(\theta,Y)\big]
		+\varepsilon^2\big[p_e^{(2)}(\theta,r)+p_p^{(2)}(\theta,Y)\big]+\cdots,
	\end{aligned}
\end{align}
where, as $Y\rightarrow -\infty$ and except for {the top-order
closure traces and affine pressure specified below},
\begin{align}
	\partial_\theta^l\partial_Y^mv_p^{(i)}(\theta,Y)\rightarrow 0,\quad
	\partial_\theta^l\partial_Y^mh_p^{(i)}(\theta,Y)\rightarrow 0,\quad
	\partial_\theta^l\partial_Y^mp_p^{(i)}(\theta,Y)\rightarrow 0,\label{matching-condition}
\end{align}
for $l,m\geq 0$ and $i=0,1,\cdots$. These expansions satisfy the following boundary conditions:
\begin{align*}
	 & u_e^{(0)}(\theta,1)+u_p^{(0)}(\theta,0)=\alpha+\eta f(\theta), \quad u_e^{(i)}(\theta,1)+u_p^{(i)}(\theta,0)=0,\ i\geq 1,     \\
	 & v_e^{(i)}(\theta,1)+v_p^{(i)}(\theta,0)=0, i\geq 0,                                                                           \\
	 & g_e^{(0)}(\theta,1)+g_p^{(0)}(\theta,0)=\beta+\eta \widehat f(\theta),\quad g_e^{(i)}(\theta,1)+g_p^{(i)}(\theta,0)=0,\ i\geq 1.
\end{align*}
The far-field conditions for $u_p^{(i)}(\theta,Y)$ and $g_p^{(i)}(\theta,Y)$ as $Y\rightarrow-\infty$ will be specified below. The profiles are solved in the following order:
\begin{align*}
	(u_e(r),0,g_e(r),0)\rightarrow (u_p^{(0)},v_p^{(1)},g_p^{(0)},h_p^{(1)})\rightarrow (u_e^{(1)},v_e^{(1)},g_e^{(1)},h_e^{(1)})\rightarrow(u_p^{(1)},v_p^{(2)},g_p^{(1)},h_p^{(2)})\rightarrow\cdot\cdot\cdot.
\end{align*}

\subsubsection{Zeroth-order Euler equations}

Substituting the above expansions into \eqref{mhd-polar} and collecting the terms of order one, we find that $(u_e^{(0)},v_e^{(0)},g_e^{(0)},h_e^{(0)},p_e^{(0)})$ satisfies the steady nonlinear Euler system
\begin{equation}
\left \{
\begin {array}{ll}
u_e^{(0)} \partial_\theta u_e^{(0)}+rv_e^{(0)} \partial_ru_e^{(0)}+u_e^{(0)} v_e^{(0)}-g_e^{(0)} \partial_\theta g_e^{(0)}-rh_e^{(0)} \partial_rg_e^{(0)}-g_e^{(0)} h_e^{(0)}+\partial_\theta p_e^{(0)}=0,\\[7pt]
u_e^{(0)} \partial_\theta v_e^{(0)}+rv_e^{(0)} \partial_rv_e^{(0)}-(u_e^{(0)})^2- g_e^{(0)} \partial_\theta h_e^{(0)}-rh_e^{(0)} \partial_rh_e^{(0)}+(g_e^{(0)})^2+r\partial_rp_e^{(0)}=0,\\[7pt]
-u_e^{(0)}  h_e^{(0)}+ g_e^{(0)} v_e^{(0)}=0,\\[7pt]
\partial_\theta u_e^{(0)}+r\partial_rv_e^{(0)}+v_e^{(0)}=0,\\[7pt]
\partial_\theta g_e^{(0)}+r\partial_rh_e^{(0)}+h_e^{(0)}=0.\label{outer-leading-order-equation}
\end{array}
\right.
\end{equation}

Motivated by the conditional selection argument in
Appendix~\ref{app:mhd-pb-selection}, we take the leading-order Euler core to be
the constant-vorticity/constant-current Couette field
\begin{align*}
 u_e^{(0)}(\theta,r)=u_e(r)=:ar,\qquad v_e^{(0)}(\theta,r)=0,\qquad
 g_e^{(0)}(\theta,r)=g_e(r):=br,\qquad h_e^{(0)}(\theta,r)=0.
\end{align*}
Here \(a\) and \(b\) are given explicitly by the MHD--Wood formula
\eqref{mhd-wood-formula}; in particular, \(b=\beta\).  The construction below
does not use the conditional hypotheses of Appendix~\ref{app:mhd-pb-selection}.

Then \eqref{outer-leading-order-equation} reduces to
\begin{align*}
	\partial_\theta p_e^{(0)}(\theta,r)=0,\quad  \partial_rp_e^{(0)}(\theta,r)=\frac{1}{r}u_e^2-\frac{1}{r}g_e^2.
\end{align*}

\subsubsection{Prandtl equations and solvability}

Substituting the expansions \eqref{first-order-expansion} into \eqref{mhd-polar} and collecting the order $\varepsilon^{-1}$ terms gives
\begin{align*}
	\partial_Yv_p^{(0)}(\theta,Y)=0,\ \ \partial_Yh_p^{(0)}(\theta,Y)=0,
	\ \ \partial_Yp_p^{(0)}(\theta,Y)=0,
\end{align*}
which together with \eqref{matching-condition} implies
\begin{align*}
	v_p^{(0)}=0,\ \ h_p^{(0)}=0,\ \ p_p^{(0)}=0.
\end{align*}
Collecting the order-one boundary layer terms gives the following steady Prandtl equations for $(u_p^{(0)},v_p^{(1)},g_p^{(0)},h_p^{(1)})$:
\begin{align}
\left \{
\begin {array}{ll}
\big(u_e(1)+u_p^{(0)}\big)\partial_\theta u_p^{(0)}+\big(v_p^{(1)}-v_p^{(1)}(\theta,0)\big)\partial_Yu_p^{(0)}-\big(g_e(1)+g_p^{(0)}\big)\partial_\theta g_p^{(0)}\\[5pt]
-\big(h_p^{(1)}+h_e^{(1)}(\theta,1)\big)\partial_Yg_p^{(0)}-\kappa\partial_{YY}u_p^{(0)}=0,\\[5pt]
\partial_\theta g_p^{(0)}+\partial_Y(\frac{\big(g_e(1)+g_p^{(0)}\big)\big(v_p^{(1)}-v_p^{(1)}(\theta,0)\big)}{u_e(1)+u_p^{(0)}})-\partial_Y(\frac{ \partial_Yg_p^{(0)}}{u_e(1)+u_p^{(0)}})=0,\\[5pt]
\partial_\theta u_p^{(0)}+\partial_Yv_p^{(1)}=0,\\[5pt]
\partial_\theta g_p^{(0)}+\partial_Yh_p^{(1)}=0,\\[5pt]
u_p^{(0)}(\theta,Y)=u_p^{(0)}(\theta+2\pi,Y),\quad v_p^{(1)}(\theta,Y)=v_p^{(1)}(\theta+2\pi,Y),\\[5pt]
u_p^{(0)}\big|_{Y=0}=\alpha+\eta f(\theta)-u_e(1),\\[5pt]
g_p^{(0)}(\theta,Y)=g_p^{(0)}(\theta+2\pi,Y),\quad h_p^{(1)}(\theta,Y)=h_p^{(1)}(\theta+2\pi,Y),\\[5pt]
g_p^{(0)}\big|_{Y=0}=\beta+\eta \widehat f(\theta)-g_e(1),\\[5pt]
\lim\limits_{Y\rightarrow -\infty}(\partial_Yu_p^{(0)},v_p^{(1)},\partial_Yg_p^{(0)},h_p^{(1)})=(0,0,0,0).
\label{prandtl-problem-near-1}
\end{array}
\right.
\end{align}
The pressure $p_p^{(1)}$ satisfies
\begin{align}\label{equation-of-first-pressure}
	 & \partial_Yp_p^{(1)}(\theta, Y)=(u_p^{(0)})^2(\theta,Y)+2u_e(1)u_p^{(0)}(\theta,Y)-(g_p^{(0)})^2(\theta,Y)-2g_e(1)g_p^{(0)}(\theta,Y).
\end{align}

\begin{lemma}[MHD--Wood identities and formula]
\label{lem:mhd-wood-identities}
Let \((u_p,v_p,g_p,h_p)\) be a sufficiently regular
\(2\pi\)-periodic solution of \eqref{prandtl-problem1}.  Assume, after a
possible reversal of orientation, that
\[
 u_p\geq c_*>0,
 \qquad
 (u_p,g_p)\longrightarrow(a,b)
 \quad\text{as }Y\to-\infty,
\]
where \(a,b\) are constants, and assume the polynomially weighted bounds used
in \eqref{prandtl-estimate} for two derivatives and some weight \(l>1\).
Then, in the von Mises coordinate
\(\psi=\int_0^Y u_p(\theta,z)\,\ud z\),
\begin{align}\label{eq:wood-identities}
 \frac{\ud^2}{\ud\psi^2}\int_0^{2\pi}g_p\,\ud\theta&=0,\nonumber\\
 \frac{\ud^2}{\ud\psi^2}\int_0^{2\pi}
 (\kappa u_p^2-g_p^2)\,\ud\theta&=0.
\end{align}
Both integrals are constant in \(\psi\), and hence
\begin{align}\label{eq:wood-formula-lemma}
 b&=\beta,\nonumber\\
 \kappa a^2-b^2
 &=\frac1{2\pi}\int_0^{2\pi}
 \big[\kappa(\alpha+\eta f)^2
 -(\beta+\eta\widehat f)^2\big]\,\ud\theta.
\end{align}
Equivalently,
{
\[
 a^2=\alpha^2+\frac{\eta^2}{2\pi}\int_0^{2\pi}
 \left(f^2-\kappa^{-1}\widehat f^{\,2}\right)\,\ud\theta,
\]
and \(a\) is determined by requiring it to have the same sign as \(\alpha\).}
\end{lemma}
\begin{proof}
The lower bound on \(u_p\) makes \((\theta,Y)\mapsto(\theta,\psi)\) a
global diffeomorphism of the half-strip.  In these variables the two
tangential equations are
\begin{align*}
 \frac12\partial_\theta(u_p^2)-g_p\partial_\theta g_p
 +u_p(\partial_\psi g_p)^2
 -\frac{\kappa u_p}{2}\partial_{\psi\psi}(u_p^2)&=0,\\
 u_p\partial_\theta g_p-g_p\partial_\theta u_p
 -u_p^2\partial_{\psi\psi}g_p&=0.
\end{align*}
Dividing the second identity by \(u_p^2\) gives
\[
 \partial_\theta\left(\frac{g_p}{u_p}\right)
 =\partial_{\psi\psi}g_p.
\]
Dividing the first identity by \(u_p\) and using the preceding relation gives
\[
 \frac12\partial_{\psi\psi}(\kappa u_p^2-g_p^2)
 =\partial_\theta\left(u_p-\frac{g_p^2}{u_p}\right).
\]
Periodic integration proves \eqref{eq:wood-identities}.  The weighted bounds
imply uniform convergence to the constant far-field limits after the inverse
von Mises change of variables; in particular, the two affine functions of
\(\psi\) in \eqref{eq:wood-identities} are bounded on
\((-\infty,0]\) and hence have zero slope.  Evaluating at the wall and at
\(\psi=-\infty\), and using the zero means of \(f,\widehat f\), proves
\eqref{eq:wood-formula-lemma}.  {The right-hand side of the
formula for \(a^2\) is positive for sufficiently small \(\eta\).  After
reversing the orientation back if needed, the branch with the same sign as
\(\alpha\) is the one continuous at \(\eta=0\).}
\end{proof}

We solve \eqref{prandtl-problem-near-1} by an energy method based on the
von Mises transform.  The algebraic relation for the normal magnetic component
reduces the MHD--Prandtl system to a coupled system for the tangential
variables, which is solved by a weighted fixed-point argument.
\begin{proposition}
	There exists $\eta_0>0$ such that, for any $\eta\in (0,\eta_0)$,
	\eqref{prandtl-problem-near-1} has a solution
	\[
	(u_p^{(0)},v_p^{(1)},g_p^{(0)},h_p^{(1)}).
	\]
	It is unique among profiles in the small weighted class constructed below
	and satisfies
	\begin{align}\label{prandtl-estimate}
		\sum_{j+k\leq m}\int_{-\infty}^0\int_0^{2\pi}\Big|\partial_\theta^j\partial_Y^k (u_p^{(0)},v_p^{(1)},g_p^{(0)},h_p^{(1)})\Big|^2\langle Y\rangle^{2l}\,\ud\theta\,\ud Y\leq C(m,l)\eta^2,\qquad m,l \geq 0,
	\end{align}
	and
	\begin{align*}
		\int_{0}^{2\pi}v^{(1)}_p(\theta, Y)\,\ud\theta=\int_{0}^{2\pi}h^{(1)}_p(\theta, Y)\,\ud\theta=0,\qquad \forall\ Y\leq 0.
	\end{align*}
	Furthermore, $(u_p^{(0)},v_p^{(1)},g_p^{(0)},h_p^{(1)},v_e^{(1)}(\theta,1),h_e^{(1)}(\theta,1))$ satisfies
	\begin{align}\label{original-prandtl}
		-\big(u_e(1)+u_p^{(0)}\big)\big(h_e^{(1)}(\theta,1)+h_p^{(1)}\big)+\big(g_e(1)+g_p^{(0)}\big)\big(v_e^{(1)}(\theta,1)+v_p^{(1)}\big)-\partial_Yg_p^{(0)}=0.
	\end{align}
	The far-field constants satisfy
	\begin{equation}\label{mhd-wood-formula}
		b=\beta,
		\qquad
		{a^2=\alpha^2+\frac{\eta^2}{2\pi}
 \int_0^{2\pi}\left(f^2-\frac1\kappa\widehat f^{\,2}\right)\ud\theta.}
	\end{equation}
	{The requirement that \(a\) and \(\alpha\) have the same sign
	determines \(a\).}
\end{proposition}
\begin{proof}

We prove the proposition by a contraction argument. The proof is divided into eight steps.

\smallskip
\noindent\textbf{Step 1.} \emph{Von Mises reduction and the MHD--Wood formula.}

We first rewrite \eqref{prandtl-problem-near-1} in terms of the total boundary layer variables:
\begin{align}
\left \{
\begin {array}{ll}\label{prandtl-1-1}
u_p\partial_\theta u_p+v_p\partial_Yu_p-g_p\partial_\theta g_p-h_p\partial_Yg_p -\kappa\partial_{YY}u_p=0,\\[5pt]
\partial_\theta g_p+ \partial_Y(\frac{v_p g_p}{u_p})-\partial_Y(\frac{\partial_{Y}g_p}{u_p})=0,\\[5pt]
\partial_\theta u_p+\partial_Yv_p=0,\\[5pt]
\partial_\theta g_p+\partial_Yh_p=0,\\[5pt]
u_p(\theta,Y)=u_p(\theta+2\pi,Y),\quad v_p(\theta,Y)=v_p(\theta+2\pi,Y),\\[5pt]
g_p(\theta,Y)=g_p(\theta+2\pi,Y),\quad h_p(\theta,Y)=h_p(\theta+2\pi,Y),\\[5pt]
u_p\big|_{Y=0}=\alpha+\eta f(\theta),\\[5pt]
v_p\big|_{Y=0}=0,\\[5pt]
g_p\big|_{Y=0}=\beta+\eta \widehat f(\theta),\\[5pt]
\lim\limits_{Y\to -\infty}(\partial_Yu_p,\partial_Yg_p)=(0,0).
\end{array}
\right.
\end{align}
Here $u_p=u_e(1)+u_p^{(0)}$ and $g_p=g_e(1)+g_p^{(0)}$. We set
\[
	v_p(\theta,Y)=-\int_0^Y \partial_\theta u_p(\theta,Y')\,\ud Y',
\]
and define $h_p$ through the algebraic identity
\[
	-u_p h_p+v_p g_p-\partial_Y g_p=0.
\]
{Assume temporarily that \(u_p>0\), and set
\(\xi=\theta\) and
\(\psi=\int_0^Y u_p(\theta,z)\,\ud z\).  The calculation in
Lemma~\ref{lem:mhd-wood-identities} shows that
\eqref{prandtl-1-1} is equivalent to the system below.  The
fixed-point construction below gives \(u_p\geq\alpha/2\), so this change of
variables is valid on the whole half-strip.}
\begin{align}\label{auxiliary-prandtl}
\left \{
\begin {array}{ll}
\frac{1}{2}\partial_\xi(u_p^2)-g_p\partial_\xi g_p+u_p(\partial_\psi g_p)^2-\kappa\frac{u_p}{2}\partial_{\psi\psi}(u_p^2)=0,\\[5pt]
u_p \partial_\xi g_p-\partial_\xi u_p g_p-u_p^2\partial_{\psi\psi}g_p=0,\\[5pt]
u_p(\theta,0)=\alpha+\eta f(\theta),\\[5pt]
g_p(\theta,0)=\beta+\eta\widehat f(\theta),
\end{array}
\right.
\end{align}
{After the fixed point constructs the constant far-field limits,
Lemma~\ref{lem:mhd-wood-identities} gives \eqref{mhd-wood-formula}.}

We now return to the original fixed-point argument and write directly
\[
u_p^2=\alpha^2+w^b+w,\qquad g_p=\beta+g^b+g.
\]
The pair $(w^b,g^b)$ solves the coupled constant-coefficient problem
\begin{equation}\label{leading-boundary-lift}
\left\{
\begin{aligned}
&w^b_\xi-2\beta g^b_\xi-\kappa\alpha w^b_{\psi\psi}=0,\\
&g^b_\xi-\frac{\beta}{2\alpha^2}w^b_\xi-\alpha g^b_{\psi\psi}=0,\\
&w^b(\xi,0)=2\alpha\eta f(\xi)+\eta^2f^2(\xi),
 \qquad g^b(\xi,0)=\eta\widehat f(\xi).
\end{aligned}
\right.
\end{equation}
Its zero mode is the constant
$\left(\dfrac{\eta^2}{2\pi}\int_0^{2\pi}f^2\,\ud\theta,0\right)$,
and its nonzero modes decay as
$\psi\to-\infty$.  {A Fourier-series expansion gives a unique
smooth lift; after subtracting the zero mode, all its polynomially weighted
Sobolev norms are bounded by \(C(m,l)\eta\).}  In the
fixed-point ball, the weighted Sobolev inequality gives
\[
	\|w^b+\bar w\|_{L^\infty}\leq \frac{3\alpha^2}{4},
	\qquad
	u_p=\sqrt{\alpha^2+w^b+\bar w}\geq\frac{\alpha}{2},
\]
which verifies the positivity used in the von Mises transformation.  For an
input $(\bar w,\bar g)$, the symbols $u_p$ and $g_p$ in the right-hand sides
below mean
\[
u_p=(\alpha^2+w^b+\bar w)^{1/2},\qquad
g_p=\beta+g^b+\bar g.
\]
We define $\mathcal T(\bar w,\bar g)=(w,g)$ by the linear system
\begin{align}\label{linear-prandtl}
\left \{
\begin {array}{ll}
w_\xi-2\beta g_\xi-\kappa\alpha w_{\psi\psi}=F_p^1,\\[5pt]
g_\xi-\frac{\beta}{2\alpha^2}w_\xi-\alpha g_{\psi\psi}=F_p^2,\\[5pt]
w(\xi,0)=g(\xi,0)=0.
\end{array}
\right.
\end{align}
Here
\begin{align}\label{leading-fixed-point-sources}
F_p^1&=\left(1-\frac{\alpha}{u_p}\right)
\partial_\xi(w^b+\bar w)
+2\left(\frac{\alpha g_p}{u_p}-\beta\right)
\partial_\xi(g^b+\bar g)
-2\alpha\big|\partial_\psi(g^b+\bar g)\big|^2,\\
F_p^2&=\left(1-\frac{\alpha}{u_p}\right)
\partial_\xi(g^b+\bar g)
+\left(\frac{\alpha g_p}{2u_p^3}
-\frac{\beta}{2\alpha^2}\right)
\partial_\xi(w^b+\bar w).
\end{align}
Indeed, substituting
$u_p^2=\alpha^2+w^b+\bar w$ and
$g_p=\beta+g^b+\bar g$ into \eqref{auxiliary-prandtl}, and multiplying the first and second
equations by $2\alpha/u_p$ and $\alpha/u_p^2$, respectively, gives
\[
\frac{\alpha}{u_p}\partial_\xi(w^b+\bar w)
-\frac{2\alpha g_p}{u_p}\partial_\xi(g^b+\bar g)
+2\alpha\big|\partial_\psi(g^b+\bar g)\big|^2
-\kappa\alpha\partial_{\psi\psi}(w^b+\bar w)=0
\]
and
\[
\frac{\alpha}{u_p}\partial_\xi(g^b+\bar g)
-\frac{\alpha g_p}{2u_p^3}\partial_\xi(w^b+\bar w)
-\alpha\partial_{\psi\psi}(g^b+\bar g)=0.
\]
Moving only the lower-order coefficient perturbations to the right gives
\eqref{linear-prandtl} and \eqref{leading-fixed-point-sources}.  In
particular, both second normal derivatives retain constant coefficients on
the left-hand side.  Since $(w^b,g^b)$ solves
\eqref{leading-boundary-lift}, a fixed point of $\mathcal T$ is exactly a
solution of \eqref{auxiliary-prandtl}.  The right-hand sides contain at most
first derivatives of the input $(\bar w,\bar g)$.
\par\smallskip
\noindent\textbf{Step 2.} \emph{Linear energy estimate.}
In Steps 2--5, $\Omega_p$ denotes the transformed strip $\mathbb{T}\times(-\infty,0]$ in the variables $(\xi,\psi)$. Multiplying \eqref{linear-prandtl} by $(w,4\alpha^2g)$, integrating over $\Omega_p$, and adding the two equations, we obtain
\begin{align*}
	&\int_{\Omega_p} (w_\xi-2\beta g_\xi-\kappa\alpha w_{\psi\psi})w
	+(g_\xi-\frac{\beta}{2\alpha^2}w_\xi-\alpha g_{\psi\psi})4\alpha^2 g \,\ud\xi\,\ud\psi\\
	&\qquad=\int_{\Omega_p} F_p^1 w\,\ud\xi\,\ud\psi
	+4\alpha^2\int_{\Omega_p} F_p^2 g\,\ud\xi\,\ud\psi .
\end{align*}
A direct integration by parts gives
\begin{equation*}
	\begin{split}
		&\int_{\Omega_p} (w_\xi-2\beta g_\xi-\kappa\alpha w_{\psi\psi})w+(g_\xi-\frac{\beta}{2\alpha^2}w_\xi-\alpha g_{\psi\psi})4\alpha^2 g \,\ud\xi\,\ud\psi\\
		=&{\int_{\Omega_p} -\kappa\alpha w_{\psi\psi}w
		-4\alpha^3g_{\psi\psi}g-2\beta g_\xi w-2\beta w_\xi g
		\,\ud\xi\,\ud\psi}\\[3pt]
		=&\int_{\Omega_p} \kappa\alpha w_{\psi}^2 +4\alpha^3 g_{\psi}^2-2\beta (g w)_\xi \,\ud\xi\,\ud\psi\\[3pt]
		=&\int_{\Omega_p} \kappa\alpha w_{\psi}^2 +4\alpha^3 g_{\psi}^2\,\ud\xi\,\ud\psi,
	\end{split}
\end{equation*}
Hence
\begin{equation}
	\label{0est}
	\begin{split}
		\int_{\Omega_p}\kappa\alpha w_\psi^2+4\alpha^3g_\psi^2
		\,\ud\xi\,\ud\psi
		={}&\int_{\Omega_p}\psi F_p^1\frac{w}{\psi}\,\ud\xi\,\ud\psi
		+4\alpha^2\int_{\Omega_p}\psi F_p^2\frac{g}{\psi}
		\,\ud\xi\,\ud\psi\\
		\leq{}&\delta\int_{\Omega_p}(w_\psi^2+g_\psi^2)
		\,\ud\xi\,\ud\psi
		+C_\delta\int_{\Omega_p}\big[(\psi F_p^1)^2+(\psi F_p^2)^2\big]
		\,\ud\xi\,\ud\psi.
	\end{split}
\end{equation}
\par\smallskip
\noindent\textbf{Step 3.} \emph{Linear positivity estimate.}
Multiplying \eqref{linear-prandtl} by
\[
	\Big(w_{\xi}-2\beta g_\xi,\ 4\kappa\alpha^2\big(g_{\xi}-\frac{\beta}{2\alpha^2}w_\xi\big)\Big),
\]
integrating over $(\xi,\psi)\in [0,2\pi)\times(-\infty,0]$, and adding the two equations, we obtain

\begin{align*}
	  & \int_{\Omega_p} (w_\xi-\kappa\alpha w_{\psi\psi}-2\beta g_\xi)(w_{\xi}-2\beta g_\xi)+(g_\xi-\alpha g_{\psi\psi}-\frac{\beta}{2\alpha^2}w_\xi)4\kappa\alpha^2 (g_{\xi}-\frac{\beta}{2\alpha^2}w_\xi)\,\ud\xi\,\ud\psi \\[3pt]
	= & \int_{\Omega_p} F_p^1 (w_{\xi}-2\beta g_\xi)\,\ud\xi\,\ud\psi+\int_{\Omega_p} F_p^2 4\kappa\alpha^2 \Big(g_{\xi}-\frac{\beta}{2\alpha^2}w_\xi\Big)\,\ud\xi\,\ud\psi,
\end{align*}
and another integration by parts yields
\begin{align*}
	  & \int_{\Omega_p} (w_\xi-\kappa\alpha w_{\psi\psi}-2\beta g_\xi)(w_{\xi}-2\beta g_\xi)+(g_\xi-\alpha g_{\psi\psi}-\frac{\beta}{2\alpha^2}w_\xi)4\kappa\alpha^2 (g_{\xi}-\frac{\beta}{2\alpha^2}w_\xi)\,\ud\xi\,\ud\psi \\
	= & \int_{\Omega_p}  (w_{\xi}-2\beta g_\xi)^2+ 4\kappa\alpha^2 (g_{\xi}-\frac{\beta}{2\alpha^2}w_\xi)^2+2\kappa\alpha \beta w_{\psi\psi} g_\xi+2\kappa\alpha \beta g_{\psi\psi} w_\xi \,\ud\xi\,\ud\psi               \\
	= & \int_{\Omega_p}  (w_{\xi}-2\beta g_\xi)^2+ 4\kappa\alpha^2 (g_{\xi}-\frac{\beta}{2\alpha^2}w_\xi)^2-2\kappa\alpha \beta w_{\psi} g_{\xi\psi}-2\kappa\alpha \beta g_{\psi} w_{\xi\psi}\,\ud\xi\,\ud\psi            \\
	= & \int_{\Omega_p}  (w_{\xi}-2\beta g_\xi)^2+ 4\kappa\alpha^2 (g_{\xi}-\frac{\beta}{2\alpha^2}w_\xi)^2 \,\ud\xi\,\ud\psi
\end{align*}
Therefore
\begin{align}\label{p-est}
	\int_{\Omega_p}  (w_{\xi}^2+g_\xi^2)\,\ud\xi\,\ud\psi\leq & C\int_{\Omega_p}  \Big(w_{\xi}-2\beta g_\xi\Big)^2+ 4\kappa\alpha^2 \Big(g_{\xi}-\frac{\beta}{2\alpha^2}w_\xi\Big)^2\,\ud\xi\,\ud\psi    \\
	\leq                                   & \delta C \int_{\Omega_p} (w_\xi^2+g_\xi^2)\,\ud\xi\,\ud\psi +C_\delta\int_{\Omega_p} \big[(F_p^1)^2+(F_p^2)^2\big]\,\ud\xi\,\ud\psi.\nonumber
\end{align}
\par\smallskip
\noindent\textbf{Step 4.} \emph{Higher-order derivative estimate.}
In Steps 4--5, \(\|\cdot\|_2\) denotes the \(L^2(\Omega_p)\)-norm unless
another domain is displayed.  For any $i\geq 0$ and $j\geq 2$,
\eqref{linear-prandtl} gives
\begin{align}\label{prandtl3}
	\|\partial_\xi^i\partial_\psi^j w\|_2+\|\partial_\xi^i\partial_\psi^j g\|_2 \leq C(\|\partial_\xi^{i+1}\partial_\psi^{j-2} w\|_2+\|\partial_\xi^{i+1}\partial_\psi^{j-2} g\|_2+\|\partial_\xi^i\partial_\psi^{j-2} F_p^1\|_2+\|\partial_\xi^i\partial_\psi^{j-2} F_p^2\|_2)
\end{align}
For any $i\geq 1$, we apply $\partial_\xi^i$ to \eqref{linear-prandtl}, multiply by $(-\partial_\xi^i w,-4\alpha^2\partial_\xi^i g)$, integrate over $(\xi,\psi)\in [0,2\pi)\times(-\infty,0]$, and add the two equations. We obtain
\begin{align}\label{prandtl4}
	\|\partial_\xi^i\partial_\psi w\|_2+\|\partial_\xi^i\partial_\psi g\|_2 \leq C(\|\partial_\xi^{i}F_p^1\|_2+\|\partial_\xi^{i}F_p^2\|_2+\delta \|\partial_\xi^iw\|_2+\delta \|\partial_\xi^ig\|_2)
\end{align}
For any $i\geq 1$, multiplying by
\[
	\Big(\partial_\xi^{2i-1} w-2\beta \partial_\xi^{2i-1} g,\ 4\alpha^2 \kappa\big(\partial_\xi^{2i-1} g-\frac{\beta}{2\alpha^2}\partial_\xi^{2i-1}w\big)\Big),
\]
integrating over $(\xi,\psi)\in [0,2\pi)\times(-\infty,0]$, and adding the two equations, we obtain
\begin{align}\label{prandtl5}
	\|\partial_\xi^i  w\|_2+\|\partial_\xi^ig\|_2 \leq C(\|\partial_\xi^{i-1}F_p^1\|_2+\|\partial_\xi^{i-1}F_p^2\|_2)
\end{align}
Combining \eqref{0est}, \eqref{p-est}, \eqref{prandtl3}, \eqref{prandtl4}, and \eqref{prandtl5}, we obtain
\begin{align*}
	     & \sum_{1\leq i+j\leq m}\Big(
	     \|\partial_\xi^i\partial_\psi^j w\|_2
	     +\|\partial_\xi^i\partial_\psi^j g\|_2\Big)\\
	\leq & C\Big(\|\psi F_p^1\|_2+\|\psi F_p^2\|_2
	+\sum_{i+j\leq m-1}\Big(
	\|\partial_\xi^i\partial_\psi^j F_p^1\|_2
	+\|\partial_\xi^i\partial_\psi^j F_p^2\|_2\Big)\Big).
\end{align*}
\par\smallskip
\noindent\textbf{Step 5.} \emph{Weighted estimate.}
Decompose $w$ and $g$ as $w=w_0+w_{\neq}$ and $g=g_0+g_{\neq}$.
Multiplying by $(-w_{\neq}\psi^{2l},-4\alpha^2 g_{\neq}\psi^{2l})$, integrating over
$(\xi,\psi)\in [0,2\pi)\times(-\infty,0]$, and adding the two equations, we obtain
\begin{align*}
	     & \|\partial_\psi w_{\neq}\psi^l\|_2+\|\partial_\psi g_{\neq}\psi^l\|_2                                                                                     \\[3pt]
	\leq & C(\|w_{\neq}\psi^{l-1}\|_2+\|g_{\neq}\psi^{l-1}\|_2+\| F_p^1\psi^l\|_2+\| F_p^2\psi^l\|_2+\|w_{\neq}\psi^l\|_2+\|g_{\neq}\psi^l\|_2)                    \\[3pt]
	\leq & C(\|w_{\xi}\psi^{l-1}\|_2+\|g_{\xi}\psi^{l-1}\|_2+\| F_p^1\psi^l\|_2+\| F_p^2\psi^l\|_2+\|\partial_\xi w\psi^l\|_2+\|\partial_\xi g\psi^l\|_2)
\end{align*}
Multiplying by $((w_{\xi}-2\beta g_\xi)\psi^{2l},4\kappa\alpha^2 (g_{\xi}-\frac{\beta}{2\alpha^2}w_\xi)\psi^{2l})$, integrating over $(\xi,\psi)\in [0,2\pi)\times(-\infty,0]$, and adding the two equations, we obtain
\begin{align*}
	  & \int_{\Omega_p} (w_\xi-\kappa\alpha w_{\psi\psi}-2\beta g_\xi)(w_{\xi}-2\beta g_\xi)\psi^{2l}+(g_\xi-\alpha g_{\psi\psi}-\frac{\beta}{2\alpha^2}w_\xi)4\kappa \alpha^2 (g_{\xi}-\frac{\beta}{2\alpha^2}w_\xi)\psi^{2l}\,\ud\xi\,\ud\psi \\
	= & \int_{\Omega_p} F_p^1 (w_{\xi}-2\beta g_\xi)\psi^{2l}\,\ud\xi\,\ud\psi+\int_{\Omega_p} F_p^2 4\kappa\alpha^2 (g_{\xi}-\frac{\beta}{2\alpha^2}w_\xi)\psi^{2l}\,\ud\xi\,\ud\psi
\end{align*}
where
\begin{align*}
	     & \int_{\Omega_p} (w_\xi-\kappa\alpha w_{\psi\psi}-2\beta g_\xi)(w_{\xi}-2\beta g_\xi)\psi^{2l}+(g_\xi-\alpha g_{\psi\psi}-\frac{\beta}{2\alpha^2}w_\xi)4\kappa \alpha^2 (g_{\xi}-\frac{\beta}{2\alpha^2}w_\xi)\psi^{2l}\,\ud\xi\,\ud\psi \\
	=    & \int_{\Omega_p} (w_{\xi}-2\beta g_\xi)^2\psi^{2l}+4\kappa\alpha^2 (g_{\xi}-\frac{\beta}{2\alpha^2}w_\xi)^2\psi^{2l}-2 \kappa\alpha w_{\psi\psi} w_\xi\psi^{2l}+2\kappa\alpha \beta w_{\psi\psi} g_\xi\psi^{2l}                \\
	     & -4\kappa\alpha^3  g_{\psi\psi} g_\xi\psi^{2l}+2\kappa\alpha \beta g_{\psi\psi} w_\xi\psi^{2l}\,\ud\xi\,\ud\psi                                                                                                                    \\
	=    & \int_{\Omega_p} (w_{\xi}-2\beta g_\xi)^2\psi^{2l}+4\kappa\alpha^2 (g_{\xi}-\frac{\beta}{2\alpha^2}w_\xi)^2\psi^{2l}+{4l\kappa\alpha w_{\neq,\psi} w_\xi\psi^{2l-1}}                                                       \\
	     & -4l\kappa\alpha \beta w_{\neq,\psi} g_\xi\psi^{2l-1}+8l\kappa\alpha^3  g_{\neq,\psi} g_\xi\psi^{2l-1}-4l\kappa\alpha \beta g_{\neq,\psi} w_\xi\psi^{2l-1}-2\kappa\alpha\beta ( g_\psi w_\psi)_\xi \psi^{2l}\,\ud\xi\,\ud\psi      \\
	\geq & C\int_{\Omega_p} w_{\xi}^2\psi^{2l} +   g_{\xi}^2\psi^{2l}\,\ud\xi\,\ud\psi-C'\int_{\Omega_p}  (w_{\neq,\psi})^2\psi^{2l-2}+ (g_{\neq,\psi})^2\psi^{2l-2}\,\ud\xi\,\ud\psi
\end{align*}
Thus,
\begin{align*}
	     & \int_{\Omega_p} w_{\xi}^2\psi^{2l} +   g_{\xi}^2\psi^{2l}\,\ud\xi\,\ud\psi                                                                                          \\
	\leq & \int_{\Omega_p} C (w_{\neq,\psi})^2\psi^{2l-2}+C (g_{\neq,\psi})^2\psi^{2l-2}\,\ud\xi\,\ud\psi+\int_{\Omega_p} (F_p^1 \psi^l)^2+(F_p^2 \psi^l)^2\,\ud\xi\,\ud\psi
\end{align*}
Applying $\partial_\xi^i$ to \eqref{linear-prandtl}, multiplying by $(-\partial_\xi^i w\psi^{2l},-4\alpha^2\partial_\xi^i g\psi^{2l})$, integrating over $(\xi,\psi)\in [0,2\pi)\times(-\infty,0]$, and adding the two equations, we obtain
\begin{align*}
	     & \|\partial_\xi^i\partial_\psi w\psi^{l}\|_2+\|\partial_\xi^i\partial_\psi g\psi^{l}\|_2                                                                                                      \\
	\leq & \|\partial_\xi^i w\psi^{l-1}\|_2+\|\partial_\xi^i g\psi^{l-1}\|_2+\delta \|\partial_\xi^i w\psi^{l}\|_2+\delta\|\partial_\xi^i g\psi^{l}\|_2+\|\partial_\xi^i F_p^1\psi^{l}\|_2+\|\partial_\xi^i F_p^2\psi^{l}\|_2
\end{align*}
Multiplying by $[(\partial_\xi^{2i-1} w-2\beta \partial_\xi^{2i-1} g)\psi^{2l},4\kappa\alpha^2 (\partial_\xi^{2i-1} g-\frac{\beta}{2\alpha^2}\partial_\xi^{2i-1}w)\psi^{2l}]$, integrating over $(\xi,\psi)\in [0,2\pi)\times(-\infty,0]$, and adding the two equations, we obtain
\begin{align*}
	     & \|\partial_\xi^i w\psi^{l}\|_2+\|\partial_\xi^i g\psi^{l}\|_2                                                                                                                                                             \\
	\leq & \|\partial_\xi^i w\psi^{l-1}\|_2+\|\partial_\xi^i g\psi^{l-1}\|_2+\delta \|\partial_\xi^{i-1}\partial_\psi  w\psi^{l}\|_2+\delta\|\partial_\xi^{i-1}\partial_\psi g\psi^{l}\|_2+\|\partial_\xi^{i-1} F_p^1\psi^{l}\|_2+\|\partial_\xi^{i-1} F_p^2\psi^{l}\|_2
\end{align*}
Combining the two weighted multiplier estimates above with the lower-order weighted bounds, we have
\begin{align*}
	\sum_{ i\geq 1,\,1\leq i+j\leq m}\Big(\|\partial_\xi^i\partial_\psi^j w\psi^{l}\|_2+\|\partial_\xi^i\partial_\psi^j g\psi^{l}\|_2\Big) \leq C\sum_{i+j\leq m-1}\Big(\|\partial_\xi^i\partial_\psi^j F_p^1\psi^{l}\|_2+\|\partial_\xi^i\partial_\psi^j F_p^2\psi^{l}\|_2\Big).
\end{align*}
{Using \eqref{linear-prandtl} once more, we also have}
\begin{align*}
	\sum_{ 2\leq k\leq m}\Big(\|\partial_\psi^k w\psi^{l}\|_2+\|\partial_\psi^k g\psi^{l}\|_2\Big) \leq C\sum_{i+j\leq m-1}\Big(\|\partial_\xi^i\partial_\psi^j F_p^1\psi^{l}\|_2+\|\partial_\xi^i\partial_\psi^j F_p^2\psi^{l}\|_2\Big).
\end{align*}

{Existence for this constant-coefficient system follows by Fourier
series: the zero mode is obtained by integrating twice in \(\psi\), and the
nonzero modes are invertible because \(\alpha^2-\beta^2\ne0\).  }  To record the zero-mode weights, write
\[
F=(F^1,F^2),\qquad F_0=\frac1{2\pi}\int_0^{2\pi}F\,\ud\xi,
\qquad \widetilde F=F-F_0,
\]
and set
\begin{align*}
N_{m,l}(F):={}&
\sum_{i+j\leq m-1}
\|\langle\psi\rangle^l
\partial_\xi^i\partial_\psi^j\widetilde F\|_2\\
&+\sum_{j\leq m-2}
\|\langle\psi\rangle^{l+2}\partial_\psi^jF_0\|_2.
\end{align*}
{The two extra powers in the zero mode account for the two integrations
needed to recover that mode.}  For fixed integers $m\geq4$ and $l\geq2$, define
\[
X_{m,l}=\{(w,g):w,g\text{ are periodic},\quad
w(\xi,0)=g(\xi,0)=0,\quad \|(w,g)\|_{X_{m,l}}<\infty\},
\]
where
\begin{align*}
\|(w,g)\|_{X_{m,l}}^2
={}&|w_\infty|^2+|g_\infty|^2\\
&+\sum_{j+k\leq m}
\left\|\langle\psi\rangle^l
\partial_\xi^j\partial_\psi^k
(w-w_\infty,g-g_\infty)\right\|_2^2.
\end{align*}
For functions with nonzero boundary data, the same notation means the
displayed norm without the homogeneous trace requirement.  Weighted Hardy
inequalities for the zero mode and the mode estimates for the nonzero part
give
\begin{equation}\label{leading-linear-solver}
\|(w,g)\|_{X_{m,l}}
\leq C(m,l)N_{m,l}\big((F_p^1,F_p^2)\big).
\end{equation}
{This proves existence and \eqref{leading-linear-solver}, records the two
far-field constants, and gives uniqueness.}

\begin{lemma}[Weighted products and zero modes]\label{lem:weighted-moser}
Let \(m\geq4\) and \(l\geq2\).  Suppose
\((q_1,q_2)\in X_{m,l}\), \(q_1\) stays in a fixed compact subset of
\((-\alpha^2,\infty)\), and \(\mathcal{Q}(q_1,q_2)\) is any finite sum of the
coefficient products occurring in \(F_p^1,F_p^2\).  Each summand is assumed
to vanish to second order at the constant far-field state, except that a
far-field constant may multiply a tangential derivative.  Write
\(q=(q_1,q_2)\) and
\(\widetilde q=(\widetilde q_1,\widetilde q_2)\).  Then
\begin{align}
N_{m,l}(\mathcal{Q}(q_1,q_2))
&\leq C_{m,l}\|(q_1,q_2)\|_{X_{m,l}}^2,\label{eq:weighted-moser}\\
N_{m,l}(\mathcal{Q}(q_1,q_2)-\mathcal{Q}(\widetilde q_1,\widetilde q_2))
&\leq C_{m,l}\big(\|q\|_{X_{m,l}}+\|\widetilde q\|_{X_{m,l}}\big)
\|q-\widetilde q\|_{X_{m,l}}.\label{eq:weighted-moser-difference}
\end{align}
The same bounds hold with one factor replaced by the boundary lift, with
its \(X_{m,l}\)-norm on the right-hand side.
\end{lemma}
\begin{proof}
For a nonzero Fourier mode, apply Leibniz' rule, put the factor carrying the
largest number of derivatives in weighted \(L^2\), and put all other factors
in \(L^\infty\) using the two-dimensional Sobolev inequality.  The rational
coefficients in \(u_p=(\alpha^2+q_1)^{1/2}\) are harmless because of the
stated lower bound.

For the zero mode, decompose each factor into its far-field constant and a
decaying part.  The average of a constant times
\(\partial_\xi q_i\) is zero.  Every remaining averaged product therefore
contains two decaying factors.  If \(p,q\) are such factors, then \(l\geq2\)
gives
\[
\langle\psi\rangle^{l+2}|pq|
\leq
\big(\langle\psi\rangle^l|p|\big)
\big(\langle\psi\rangle^l|q|\big).
\]
Leibniz' rule, weighted \(L^2\)-\(L^\infty\) estimates, and the same argument
after taking a difference prove
\eqref{eq:weighted-moser}--\eqref{eq:weighted-moser-difference}.  This is
precisely the two-weight gain required by the zero-mode part of \(N_{m,l}\).
\end{proof}

\par\smallskip
\noindent\textbf{Step 6.} \emph{$\mathcal{T}$ maps $B_{m,l}(r_0)$ into $B_{m,l}(r_0)$.}
We define the ball $B_{m,l}(r_0)$ in $X_{m,l}$ by
\begin{align*}
	B_{m,l}(r_0)=\big\{(w,g)\in X_{m,l}:
	\|(w,g)\|_{X_{m,l}}\leq r_0\big\},
\end{align*}
where $r_0$ is a small number to be chosen later.

We use the Sobolev inequality
\begin{equation*}
	\|u\|_\infty\leq C(\|u\|_2+\|u_\xi\|_2+\|u_\psi\|_2+\|u_{\xi\psi}\|_2),
\end{equation*}
which implies, for $i+j\leq m-2$,
\begin{equation*}
	\|\langle\psi\rangle^{l}\partial_\xi^{i} \partial_\psi^{j} w\|_\infty\leq C\Big(\|\langle\psi\rangle^{l}\partial_\xi^{i+1} \partial_\psi^{j+1} w\|_2+\|\langle\psi\rangle^{l-1}\partial_\xi^{i+1} \partial_\psi^{j} w\|_2+\|\langle\psi\rangle^{l}\partial_\xi^{i} \partial_\psi^{j} w\|_2+\|\langle\psi\rangle^{l-1}\partial_\xi^{i} \partial_\psi^{j+1} w\|_2\Big).
\end{equation*}
Writing \(F_p=(F_p^1,F_p^2)\), we next prove
\begin{align}\label{leading-moser-estimate}
&N_{m,l}\big((F_p^1,F_p^2)\big)
\leq C(m,l)(\eta+r_0)^2,\\
&N_{m,l}\big(F_p(\bar w_1,\bar g_1)
-F_p(\bar w_2,\bar g_2)\big)
\leq C(m,l)(\eta+r_0)
\|(\bar w_1-\bar w_2,\bar g_1-\bar g_2)\|_{X_{m,l}}.
\end{align}
Indeed, the identities
\[
1-\frac{\alpha}{u_p}
=\frac{w^b+\bar w}{u_p(u_p+\alpha)},
\qquad
\frac{\alpha g_p}{u_p}-\beta
=\frac{\alpha(g^b+\bar g)}{u_p}
+\beta\left(\frac{\alpha}{u_p}-1\right),
\]
and
\[
\frac{\alpha g_p}{2u_p^3}-\frac{\beta}{2\alpha^2}
=\frac{\alpha^3(g^b+\bar g)+\beta(\alpha^3-u_p^3)}
{2\alpha^2u_p^3}.
\]
These coefficient perturbations are smooth for \(u_p\geq\alpha/2\), vanish
at the constant state, and multiply a second decaying factor or a tangential
derivative.  Lemma~\ref{lem:weighted-moser}, together with
{the lift bounds stated in Step~1}, therefore proves both estimates, including the
two additional weights in their zero modes.
Combining \eqref{leading-linear-solver} and
\eqref{leading-moser-estimate}, let \(C_{\mathrm{lin}}\) denote the resulting
estimate constant.  Choose $r_0=2C_{\mathrm{lin}}\eta^2$ and then decrease
$\eta_0$ so that $C(\eta+r_0)\leq1/2$.  It follows that $\mathcal T$ maps
$B_{m,l}(r_0)$ into itself.

\par\smallskip
\noindent\textbf{Step 7.} \emph{$\mathcal{T}$ is a contraction on $B_{m,l}(r_0)$.}
The second estimate in \eqref{leading-moser-estimate} shows that
$\mathcal T$ has Lipschitz constant at most $1/2$ on this ball.  Hence it
has a unique fixed point.  The weighted norm shows that the total profiles
$u_p$ and $g_p$ have constant limits, denoted by $a$ and $b$, respectively.
The MHD--Wood identities derived at the beginning of the proof identify these
limits with the explicit values in \eqref{mhd-wood-formula}.
The inverse von Mises map is well defined because $u_p\geq\alpha/2$ and
preserves all polynomially weighted estimates.  Applying the same estimates
successively to the differentiated equations gives the stated bounds for
every $m,l$.
We then define
\begin{align*}
u_p^{(0)}=u_p-a,\quad
v_p^{(1)}=v_p-v_\infty,\quad
g_p^{(0)}=g_p-b,\quad
h_p^{(1)}=h_p-h_\infty,
\end{align*}
where $v_\infty=\lim_{Y\to-\infty}v_p$ and
$h_\infty=\lim_{Y\to-\infty}h_p$.
\par\smallskip
\noindent\textbf{Step 8.} \emph{The choice of $h_p$.}
By the definition
\[
	h_p=\frac{v_pg_p-\partial_Yg_p}{u_p},
\]
the second equation in \eqref{prandtl-1-1} gives directly
\[
	\partial_\theta g_p+\partial_Yh_p
	=\partial_\theta g_p+\partial_Y\Big(\frac{v_pg_p}{u_p}\Big)-\partial_Y\Big(\frac{\partial_Yg_p}{u_p}\Big)=0.
\]
Moreover, taking $Y\to-\infty$ in the identity $-u_ph_p+v_pg_p-\partial_Yg_p=0$ gives
\[
	\lim\limits_{Y\to-\infty}h_p=\frac{\lim\limits_{Y\to-\infty}g_p}{\lim\limits_{Y\to-\infty}u_p}\lim\limits_{Y\to-\infty}v_p.
\]
Set
\[
	v_e^{(1)}(\theta,1)=\lim\limits_{Y\to-\infty}v_p,\qquad
	h_e^{(1)}(\theta,1)= \frac{\lim\limits_{Y\to -\infty} g_p}{\lim\limits_{Y\to -\infty}u_p} v_e^{(1)}(\theta,1).
\]
Then $(u_p,v_p,g_p,h_p)$ solves the MHD--Prandtl equations, and \eqref{original-prandtl} holds.

The weighted estimate for $\partial_\theta u_p^{(0)}$ implies that
$v_\infty$ exists.  Moreover,
\[
	v^{(1)}_p(\theta,Y)=\int_{-\infty}^Y\partial_Yv^{(1)}_{p}(\theta,Y')\,\ud Y'=-\int_{-\infty}^Y\partial_\theta u^{(0)}_{p}(\theta,Y')\,\ud Y',
\]
so the weighted Hardy inequality gives the asserted bounds for
$v_p^{(1)}$.  The algebraic formula for $h_p$ gives the same bounds for
$h_p^{(1)}$.  Thus \eqref{prandtl-estimate} follows.  Periodicity also gives
\begin{align*}
	\int_{0}^{2\pi}v^{(1)}_p(\theta, Y)\,\ud\theta=0,\qquad \forall\ Y\leq 0.
\end{align*}
Similarly, $\partial_\theta g_p^{(0)}+\partial_Yh_p^{(1)}=0$ and $h_p^{(1)}\to0$ as $Y\to-\infty$ imply
\[
	\int_{0}^{2\pi}h^{(1)}_p(\theta,Y)\,\ud\theta=0,\qquad \forall\ Y\leq0.
\]
Finally, solving \eqref{equation-of-first-pressure}, we obtain a pressure profile $p_p^{(1)}(\theta,Y)$ which decays rapidly as $Y\rightarrow -\infty$.
\end{proof}

\subsubsection{First-order linearized Euler equations}
Substituting the Euler expansion \eqref{euler-expansions} into \eqref{mhd-polar} and collecting the order-$\varepsilon$ terms, we find that $(u_e^{(1)},v_e^{(1)},g_e^{(1)},h_e^{(1)},p_e^{(1)})$ satisfies the following linearized Euler equations in $\Omega$:
\begin{equation}\label{euler-1}
\left \{
\begin {array}{ll}
ar \partial_\theta u_e^{(1)}+2arv_e^{(1)}-br \partial_\theta g_e^{(1)}-2brh_e^{(1)} +\partial_\theta p_e^{(1)}=0,\\[5pt]
ar \partial_\theta v_e^{(1)}-2aru_e^{(1)}-br \partial_\theta h_e^{(1)}+2brg_e^{(1)}+r\partial_rp_e^{(1)}=0,\\[5pt]
-u_e^{(0)}  h_e^{(1)}+g_e^{(0)}  v_e^{(1)}=0,\\[5pt]
\partial_\theta u_e^{(1)}+r\partial_rv_e^{(1)}+ v_e^{(1)}=0,\\[5pt]
\partial_\theta g_e^{(1)}+r\partial_rh_e^{(1)}+ h_e^{(1)}=0,\\[5pt]
v_e^{(1)}|_{r=1}=-v_p^{(1)}|_{Y=0},\quad v_e^{(1)}(\theta,r)=v_e^{(1)}(\theta+2\pi,r) ,
\end{array}
\right.
\end{equation}

\begin{proposition}
	The linearized MHD--Euler equations \eqref{euler-1} have a solution
	\[
		(u_e^{(1)}, v_e^{(1)},g_e^{(1)}, h_e^{(1)}, p_e^{(1)})
	\]
	satisfying
	\begin{align}
		 & |\partial_\theta u_e^{(1)}+v_e^{(1)}|(\theta,r)+|\partial_\theta g_e^{(1)}+h_e^{(1)}|(\theta,r) \leq C\eta r, \ \forall (\theta, r)\in \Omega,\nonumber                                                                \\
		 & {\color{black}|\partial_\theta v_e^{(1)}-u_e^{(1)}|(\theta,r)+|\partial_\theta h_e^{(1)}-g_e^{(1)}|(\theta,r) \leq C\eta r,  \ \forall (\theta, r)\in \Omega,} \label{estimate-of-first-combined-linearized-euler-equation1} \\[5pt]
		 & \|\partial^k_\theta\partial^j_r(u_e^{(1)},v_e^{(1)},g_e^{(1)},h_e^{(1)})\|_2\leq C(k,j)\eta, \quad \forall j,k\geq 0,\nonumber                                                                                                  \\[5pt]
		 & r^2\Delta u_e^{(1)}-u_e^{(1)}+2\partial_\theta v_e^{(1)}=0,\quad \int_{0}^{2\pi}v_e^{(1)}(\theta,r)\,\ud\theta=0,\nonumber                                                                                            \\[5pt]
		 & r^2\Delta g_e^{(1)}-g_e^{(1)}+2\partial_\theta h_e^{(1)}=0,\quad \int_{0}^{2\pi}h_e^{(1)}(\theta,r)\,\ud\theta=0.\nonumber
	\end{align}
\end{proposition}
\begin{proof}
Eliminating the pressure $p_e^{(1)}$ in \eqref{euler-1}, we obtain the following equation for $rv_e^{(1)}$ in $\Omega$:
\begin{equation}\label{equation-for-first-euler-normal}
\left \{
\begin {array}{ll}
-ar\Delta(rv_e^{(1)})+br\Delta(rh_e^{(1)})=0,\\[5pt]
-ah_e^{(1)}+bv_e^{(1)}=0  ,\\[5pt]
rv_e^{(1)}|_{r=1}=-v_p^{(1)}|_{Y=0}.
\end{array}
\right.
\end{equation}
From the first two equations,
\begin{align*}
	-\frac{a^2-b^2}{a}r\Delta(rv_e^{(1)})=0.
\end{align*}
The identities $a=\alpha+O(\eta^2)$ and $b=\beta$, together with
$|\alpha|\neq|\beta|$, give $|a|\neq|b|$ after decreasing $\eta_0$ if
necessary.  Hence $\Delta(rv_e^{(1)})=0$.  Write the boundary value as the
Fourier series
\begin{align*}
	v_p^{(1)}(\theta,0)=\sum_{n=1}^{+\infty}\big(a_{n1} \cos(n\theta)+b_{n1}\sin(n\theta)\big).
\end{align*}
Then
\begin{align*}
	u_e^{(1)}(\theta, r) & =\sum_{n=1}^{+\infty}\big(a_{n1} r^{n-1} \sin(n\theta)-b_{n1}r^{n-1}\cos(n\theta)\big),  \\
	v_e^{(1)}(\theta,r)  & =-\sum_{n=1}^{+\infty}\big(a_{n1} r^{n-1} \cos(n\theta)+b_{n1}r^{n-1}\sin(n\theta)\big).
\end{align*}
This pair solves \eqref{equation-for-first-euler-normal}. We define
\[
	g_e^{(1)}=\frac{b}{a}u_e^{(1)},\qquad h_e^{(1)}=\frac{b}{a}v_e^{(1)}.
\]
The divergence equations and the algebraic relation in \eqref{euler-1} are then satisfied. The first two equations determine $p_e^{(1)}$ up to an additive constant, and the stated estimates follow directly from the Fourier representation.
\end{proof}

\subsubsection{First-order linearized Prandtl equations}

Substituting the Prandtl expansions \eqref{first-order-expansion} into the first and third equations in \eqref{mhd-polar} and collecting the order-$\varepsilon$ terms, we obtain the following linearized steady Prandtl equations for
\[
	(u_p^{(1)},v_p^{(2)},g_p^{(1)},h_p^{(2)}).
\]
{\small
\begin{equation}
\left \{
\begin {array}{ll}
\big(u_e(1)+u_p^{(0)}\big)\partial_\theta u_p^{(1)}+\big(v_e^{(1)}(\theta,1)+ v_p^{(1)}\big)\partial_Yu_p^{(1)}+(v_e^{(2)}(\theta,1)+v_p^{(2)})\partial_{Y}u_p^{(0)}\\[5pt]
+(u_p^{(1)}+u_e^{(1)}(\theta,1))\partial_\theta u_p^{(0)}-\big(g_e(1)+g_p^{(0)}\big)\partial_\theta g_p^{(1)}-\big(h_e^{(1)}(\theta,1)+ h_p^{(1)}\big)\partial_Yg_p^{(1)}\\[5pt]
\quad\quad-(g_e^{(1)}(\theta,1)+g_p^{(1)})\partial_{\theta}g_p^{(0)}
-(h_e^{(2)}(\theta,1)+h_p^{(2)})\partial_{Y}g_p^{(0)}
-\kappa\partial_{YY}u_p^{(1)}=f_p^{(1)}(\theta,Y)\\[5pt]
-\big(u_e(1)+u_p^{(0)}\big)\big(h_e^{(2)}(\theta,1)+ \partial_rh_e^{(1)}(\theta,1)Y+ h_p^{(2)}\big)+\big(g_e(1)+g_p^{(0)}\big)\big(v_e^{(2)}(\theta,1)+\partial_rv_e^{(1)}(\theta,1)Y+ v_p^{(2)}\big)\\[5pt]
-(u_e^{(1)}(\theta,1)+u_p^{(1)})(h_e^{(1)}(\theta,1)+h_p^{(1)})+(g_e^{(1)}(\theta,1)+g_p^{(1)})(v_e^{(1)}(\theta,1)+v_p^{(1)})-\partial_{Y}g_p^{(1)}=\partial_Y(Yg_p^{(0)}),\\[5pt]
\partial_\theta u_p^{(1)}+\partial_Yv_p^{(2)}+\partial_Y(Yv_p^{(1)})=0,\quad \partial_\theta g_p^{(1)}+\partial_Yh_p^{(2)}+\partial_Y(Yh_p^{(1)})=0,\\[5pt]
u_p^{(1)}(\theta,Y)=u_p^{(1)}(\theta+2\pi,Y),\quad v_p^{(2)}(\theta,Y)=v_p^{(2)}(\theta+2\pi,Y),\\[5pt]
g_p^{(1)}(\theta,Y)=g_p^{(1)}(\theta+2\pi,Y),\quad h_p^{(2)}(\theta,Y)=h_p^{(2)}(\theta+2\pi,Y),\\[5pt]
u_p^{(1)}\big|_{Y=0}=-u_e^{(1)}\big|_{r=1},\quad g_p^{(1)}\big|_{Y=0}=-g_e^{(1)}\big|_{r=1},\\[5pt]
\lim\limits_{Y\rightarrow -\infty}(\partial_Yu_p^{(1)},v_p^{(2)},\partial_Yg_p^{(1)},h_p^{(2)})=(0,0,0,0),
\label{first-linearized-prandtl-problem-near-1}
\end{array}
\right.
\end{equation}}
where
\begin{align*}
	f_p^{(1)}(\theta,Y)= & -\partial_\theta p_p^{(1)}+\kappa Y\partial_{YY}u_p^{(0)}+\kappa \partial_Yu_p^{(0)}-u_p^{(0)}\big(\partial_\theta u_e^{(1)}(\theta,1)+v_e^{(1)}(\theta,1)+v_p^{(1)}\big)-u'_e(1)Y\partial_\theta u_p^{(0)} \\[5pt]
	                     & -(\partial_rv_e^{(1)}(\theta,1)+v_e^{(1)}(\theta,1))Y\partial_Y u_p^{(0)}-(u_e'(1)+Y\partial_Yu_p^{(0)}+u_e(1))v_p^{(1)}                                                                                    \\[5pt]
	                     & +g_p^{(0)}\big(\partial_\theta g_e^{(1)}(\theta,1)+h_e^{(1)}(\theta,1)+h_p^{(1)}\big)+g'_e(1)Y\partial_\theta g_p^{(0)}                                                                                     \\[5pt]
	                     & +(\partial_rh_e^{(1)}(\theta,1)+h_e^{(1)}(\theta,1))Y\partial_Y g_p^{(0)}+(g_e'(1)+Y\partial_Yg_p^{(0)}+g_e(1))h_p^{(1)}.
\end{align*}
We now prove the solvability of \eqref{first-linearized-prandtl-problem-near-1}.
\begin{proposition}
	There exists $\eta_0>0$ such that, for any $\eta\in(0,\eta_0)$,
	\eqref{first-linearized-prandtl-problem-near-1} has a solution
	\[
	(u_p^{(1)},v_p^{(2)},g_p^{(1)},h_p^{(2)}).
	\]
	It is unique in the weighted class described in the proof and satisfies
	\begin{align}
		 & \sum_{j+k\leq m}\int_{-\infty}^0\int_0^{2\pi}\big|\partial_\theta^j\partial_Y^k \big(u_p^{(1)}-A_{1\infty},v_p^{(2)},g_p^{(1)},h_p^{(2)}\big)\big|^2\langle Y\rangle^{2l}\,\ud\theta\,\ud Y\leq C(m,l)\eta^2,\quad m,l\geq 0;\label{decay-behavior-prandtl-1} \\
		 & \int_0^{2\pi}v_p^{(2)}(\theta,Y)\,\ud\theta =\int_0^{2\pi}h_p^{(2)}(\theta,Y)\,\ud\theta =0,\qquad \forall \ Y\leq 0, \nonumber
	\end{align}
	where
	$A_{1\infty}:=\lim\limits_{Y\rightarrow -\infty}u_p^{(1)}(\theta,Y)$ is a constant which satisfies $|A_{1\infty}|\leq C\eta$.
\end{proposition}
\begin{proof}
Let $\chi_0\in C_c^\infty ((-\infty,0])$ satisfy
\begin{align*}
	\chi_0(0)=1,\quad \int^0_{-\infty}\chi_0(y)\,\ud y=0.
\end{align*}
For convenience, we set
\begin{align*}
	\bar u&:=u_e(1)+u_p^{(0)}, \quad \bar v:=v_e^{(1)}(\theta,1)+v_p^{(1)},                                                                                \\
	u&:=u_p^{(1)}+u_e^{(1)}(\theta,1)\chi_0(Y), \quad v:=v_p^{(2)}-v_p^{(2)}(\theta,0)+Y v_p^{(1)}-\partial_\theta u_e^{(1)}(\theta,1)\int_0^Y\chi_0(z)dz. \\
	\bar g&:=g_e(1)+g_p^{(0)}, \quad \bar h:=h_e^{(1)}(\theta,1)+h_p^{(1)},                                                                                \\
	g&:=g_p^{(1)}+g_e^{(1)}(\theta,1)\chi_0(Y), \quad h:=h_p^{(2)}+h_e^{(2)}(\theta,1)+Y h_p^{(1)}-\partial_\theta g_e^{(1)}(\theta,1)\int_0^Y\chi_0(z)dz.
\end{align*}
We regard $v_e^{(2)}(\theta,1)$ and $h_e^{(2)}(\theta,1)$ as boundary traces to be determined after solving $(u,v,g,h)$. The term $v_e^{(2)}(\theta,1)$ is absorbed by the boundary relation
\[
	v_p^{(2)}(\theta,0)=-v_e^{(2)}(\theta,1),
\]
while $h_e^{(2)}(\theta,1)$ is included in the definition of $h$; the coercive estimates below are independent of this later choice. Then the equations \eqref{first-linearized-prandtl-problem-near-1} reduce to
\begin{equation}\label{new-linearized-prandtl-equation}
\left \{
\begin {array}{ll}
\bar{u}\partial_\theta u+\bar{v}\partial_Yu+u\partial_\theta \bar{u}+v\partial_Y\bar{u}-\bar{g}\partial_\theta g-\bar{h}\partial_Yg-g\partial_\theta \bar{g}-h\partial_Y\bar{g} -\kappa\partial_{YY}u=F^*,\\[7pt]
-\bar{u}h-u\bar{h}+\bar{g}v+g\bar{v}-\partial_Y g=G^*,\\[7pt]
\partial_\theta u+\partial_Yv=0,\\[5pt]
\partial_\theta g+\partial_Yh=0,\\[5pt]
u(\theta,Y)=u(\theta+2\pi,Y),\quad v(\theta,Y)=v(\theta+2\pi,Y),\\[5pt]
g(\theta,Y)=g(\theta+2\pi,Y),\quad h(\theta,Y)=h(\theta+2\pi,Y)\\[5pt]
u|_{Y=0}=v|_{Y=0}=g|_{Y=0}=0,\\ [5pt]
\lim\limits_{Y\rightarrow -\infty}\partial_Yu=\lim\limits_{Y\rightarrow -\infty}\partial_Yg=0,
\end{array}
\right.
\end{equation}
where $F^*(\theta,Y),G^*(\theta,Y)$ are known $2\pi$-periodic functions
determined by the lower-order profiles and the cutoff $\chi_0$.  Moreover, for
any $m,l\geq0$,
\begin{align*}
	\sum_{j+k\leq m}\Big(&\|\partial_\theta^j\partial_Y^kF^*\langle Y\rangle^l\|_2+\|\partial_\theta^j\partial_Y^kG^*\langle Y\rangle^l\|_2 \\
	&+\|\partial_\theta^j\partial_Y^k\partial_YG^*\langle Y\rangle^l\|_2\Big)\leq C(m,l)\eta.
\end{align*}
In particular, they decay rapidly as $Y\rightarrow -\infty$.
Multiplying the second equation of \eqref{new-linearized-prandtl-equation} by $\bar{u}^{-1}$ and then applying $\bar{u}\partial_Y$, we obtain the auxiliary problem
\begin{equation}\label{re-new-linearized-prandtl-equation}
\left \{
\begin {array}{ll}
\bar{u}\partial_\theta u+\bar{v}\partial_Yu+u\partial_\theta \bar{u}+v\partial_Y\bar{u}-\bar{g}\partial_\theta g-\bar{h}\partial_Yg-g\partial_\theta \bar{g}-h\partial_Y\bar{g} -\kappa\partial_{YY}u=F^*,\\[7pt]
[\partial_\theta g -\partial_Y(\frac{\bar{h}}{\bar{u} })u-\frac{\bar{h}}{\bar{u} }\partial_Y u +\partial_Y(\frac{\bar{g}}{\bar{u} })v-\frac{\bar{g}}{\bar{u}} \partial_\theta u  +\frac{\bar{v}}{\bar{u} }\partial_Yg + \partial_Y(\frac{\bar{v}}{\bar{u}})g-\partial_Y(\frac{\partial_Y g}{ \bar{u} }) ]\bar{u}=\bar{u}\partial_Y(\frac{G^*}{ \bar{u} })  ,\\[5pt]
u(\theta,Y)=u(\theta+2\pi,Y),\quad g(\theta,Y)=g(\theta+2\pi,Y)\\[5pt]
u|_{Y=0}=g|_{Y=0}=0,\ \ \lim\limits_{Y\rightarrow -\infty}\partial_Yu=\lim\limits_{Y\rightarrow -\infty}\partial_Yg=0,
\end{array}
\right.
\end{equation}
In this equation, $v=-\int_0^{Y}\partial_\theta u\,\ud Y'$ and
\[
	h=\frac{-G^*-u\bar{h}+\bar{g}v+g\bar{v}-\partial_Y g}{\bar{u}},
\]
so \eqref{re-new-linearized-prandtl-equation} is an equation for $(u,g)$. We
seek \((u,g)\) in the periodic homogeneous product space
\[
\dot H^1_{0,\mathrm{per}}(\Omega_p)^2
=\left\{(u,g):
\begin{array}{l}
u,g\ \hbox{are \(2\pi\)-periodic in \(\theta\)},\
u|_{Y=0}=g|_{Y=0}=0,\\
\partial_\theta u,\partial_Yu,\partial_\theta g,\partial_Yg
\in L^2(\Omega_p)
\end{array}\right\},
\]
where the possible far-field constants are recorded separately.  This is a
homogeneous, rather than an \(H^1_0\), formulation.  We now establish an a
priori estimate for \eqref{re-new-linearized-prandtl-equation}.
Multiplying the two equations in \eqref{re-new-linearized-prandtl-equation} by $(u,g)$, integrating over $(\theta,Y)\in(0,2\pi)\times(-\infty,0)$, and adding the results, we obtain
\begin{align*}
	  & \int_{\Omega_p}\Big[\bar{u}\partial_\theta u+\bar{v}\partial_Yu+u\partial_\theta \bar{u}+v\partial_Y\bar{u}-\bar{g}\partial_\theta g-\bar{h}\partial_Yg-g\partial_\theta \bar{g}-h\partial_Y\bar{g} -\kappa\partial_{YY}u\Big]u                                                  \\
	  & +\int_{\Omega_p}\big(\partial_\theta g -\partial_Y(\frac{\bar{h}}{\bar{u} })u-\frac{\bar{h}}{\bar{u} }\partial_Y u +\partial_Y(\frac{\bar{g}}{\bar{u} })v-\frac{\bar{g}}{\bar{u}} \partial_\theta u  +\frac{\bar{v}}{\bar{u} }\partial_Yg + \partial_Y(\frac{\bar{v}}{\bar{u}})g-\partial_Y(\frac{\partial_Y g}{ \bar{u} })  \big)\bar{u}g\,\ud\theta\,\ud Y \\
	= & \underbrace{\int_{\Omega_p} -\kappa\partial_{YY}u u-\partial_Y(\frac{\partial_Y g}{ \bar{u} }) \bar{u}g \,\ud\theta\,\ud Y}_{I_1}- \underbrace{\int_{\Omega_p} bg_\theta u+b u_\theta g\,\ud\theta\,\ud Y}_{I_2}                                                                                 \\
	+ & \underbrace{\int_{\Omega_p}\Big[\bar{u}\partial_\theta u+\bar{v}\partial_Yu +u\partial_\theta \bar{u}+v\partial_Y\bar{u}-g_p^{(0)}\partial_\theta g-g\partial_\theta \bar{g}\Big]u \,\ud\theta\,\ud Y}_{I_3}                                                                       \\
	+ & \underbrace{\int_{\Omega_p}     \Big[\partial_\theta g +\frac{\partial_\theta\bar{g}}{\bar{u} }u +\partial_Y(\frac{\bar{g}}{\bar{u} })v-\frac{g_p^{(0)}}{\bar{u}} \partial_\theta u  +\frac{\bar{v}}{\bar{u} }\partial_Yg + \partial_Y(\frac{\bar{v}}{\bar{u}})g\Big] \bar{u}g \,\ud\theta\,\ud Y}_{I_4}                                 \\
	+ & \underbrace{\int_{\Omega_p}\Big(-\bar{h}\partial_Ygu-\bar{h}\partial_Yug-\partial_Y\bar{h}\,ug+\frac{\bar{h} \partial_Y \bar{u}}{\bar{u}}ug\Big) \,\ud\theta\,\ud Y}_{I_5}     + \underbrace{\int_{\Omega_p} -uh\partial_Y\bar{g}\,\ud\theta\,\ud Y}_{I_6}                                                                            \\
	= & \int_{\Omega_p}F^*u +\partial_Y(\frac{G^*}{ \bar{u} })\bar{u} g\,\ud\theta\,\ud Y.
\end{align*}

By \eqref{prandtl-estimate} and \eqref{estimate-of-first-combined-linearized-euler-equation1}, we have
\begin{align*}
	 & \big|\partial^j_{\theta}\partial^k_{Y}(\bar{u}-a)\langle Y\rangle^l\big|+\big|\partial^j_{\theta}\partial^k_{Y}(\bar{g}-b)\langle Y\rangle^l\big| \leq C(j,k,l)\eta, \\[5pt] &\big|\partial^j_{\theta}\partial^k_{Y}v_p^{(1)}\langle Y\rangle^l\big|+\big|\partial^j_{\theta}\partial^k_{Y}h_p^{(1)}\langle Y\rangle^l\big| \leq C(j,k,l)\eta,\\[5pt]
	 & \big|\partial^j_\theta v^{(1)}_e(\theta,1)\big|+\big|\partial^j_\theta h^{(1)}_e(\theta,1)\big|\leq C(j)\eta,
\end{align*}
for all admissible $j,k$. Hence
\begin{align*}
	I_1= & \int_{\Omega_p}(-\kappa\partial_{YY}u u- \partial_Y(\frac{\partial_Y g}{ \bar{u} }) \bar{u}g) \,\ud\theta\,\ud Y               \\
	=    & \kappa\|\partial_Y u\|_2^2 +\|\partial_Yg \|_2^2+ \int \frac{Y\partial_Y\bar{u}}{ \bar{u}} \frac{g}{Y}\partial_Yg \,\ud\theta\,\ud Y \\
	\geq & \kappa\|\partial_Y u\|_2^2 +\|\partial_Yg \|_2^2-C\eta \|\partial_Y g\|_2^2
\end{align*}
and we obtain
\begin{align*}
	I_2=\int_{\Omega_p}(-b\partial_\theta g\, u-b\partial_\theta u\, g)\,\ud\theta\,\ud Y=0
\end{align*}
and
\begin{align*}
	|I_3|= & \big|\int_{\Omega_p}\Big[\bar{u}\partial_\theta u+\bar{v}\partial_Yu +v\partial_Y\bar{u}+u\partial_\theta \bar{u}-g_p^{(0)}\partial_\theta g-g\partial_\theta \bar{g}\Big]u \,\ud\theta\,\ud Y\big|           \\
	\leq   & \big|\int_{\Omega_p}\frac{1}{2}\big[\partial_\theta\bar{u}+\partial_Y\bar{v}\big]u^2 \,\ud\theta\,\ud Y\big|+\|Y^2\bar{u}_Y\|_{\infty}\Big\|\frac{\int_Y^0\partial_{\theta}u(\theta,z)\,\ud z}{Y}\Big\|_2\Big\|\frac{u}{Y}\Big\|_2  \\
	       & +\|Y^2\bar{u}_\theta\|_{\infty}\Big\|\frac{u}{Y}\Big\|_2^2+\|Yg_p^{(0)}\|_\infty \|\partial_\theta g\|_2\|\frac{u}{Y}\|_2+\| Y^2 \partial_\theta \bar{g}\|_\infty \Big\|\frac{u}{Y}\Big\|_2\Big\|\frac{g}{Y}\Big\|_2 \\
	\leq   & C\eta \big(\|\partial_{\theta}u\|_2^2+\|\partial_Yu \|_2^2+\|\partial_{\theta}g\|_2^2+\|\partial_Yg \|_2^2 \big),
\end{align*}
where we used $\bar{u}_\theta+\bar{v}_Y=0$, $\bar{g}-b=g_p^{(0)}$, and Hardy's inequality. Similarly,
\begin{align*}
	\big|I_4\big|= & \big|\int_{\Omega_p}     \Big[\partial_\theta g +\frac{\partial_\theta\bar{g}}{\bar{u} }u +\partial_Y(\frac{\bar{g}}{\bar{u} })v-\frac{g_p^{(0)}}{\bar{u}} \partial_\theta u  +\frac{\bar{v}}{\bar{u} }\partial_Yg + \partial_Y(\frac{\bar{v}}{\bar{u}})g \Big]\bar{u}g \,\ud\theta\,\ud Y\big| \\[5pt]
	\leq           & \big \|Yg_p^{(0)}\|_\infty \|\partial_\theta u\|_2\|\frac{g}{Y}\|_2+\| Y^2 \partial_\theta \bar{g}\|_\infty \Big\|\frac{u}{Y}\Big\|_2\Big\|\frac{g}{Y}\Big\|_2                                                                                            \\[5pt]
	               & + \|Y^2 \bar{u}\partial_Y  (\frac{\bar{v}}{\bar{u}})\|_\infty\|\frac{g}{Y}\|_2^2+\|\frac{\big|\int_Y^0 \partial_\theta u\big|}{Y}\|_2\|Y^2\bar{u} \partial_Y(\frac{\bar{g}}{\bar{u}})\|_\infty\|\frac{g}{Y}\|_2                                                             \\[5pt]
	\leq           & C\eta(\|\partial_\theta u\|_2^2+\|\partial_\theta g\|_2^2+\|\partial_Y u\|_2^2+\|\partial_Y g\|_2^2)
\end{align*}
The cancellation in $I_5$ is
\begin{align*}
	\big|I_5\big| & =\big|\int_{\Omega_p}-\partial_Y(\bar{h}ug)+\frac{\bar{h}\partial_Y\bar{u}}{\bar{u}}ug\,\ud\theta\,\ud Y \big|            \\
	              & =\big|\int_{\Omega_p}\frac{Y^2\bar{h}\partial_Y\bar{u}}{\bar{u}}\frac{u}{Y}\frac{g}{Y}\,\ud\theta\,\ud Y \big| \\
	              & \leq C\eta(\|\partial_\theta u\|_2^2+\|\partial_\theta g\|_2^2+\|\partial_Y u\|_2^2+\|\partial_Y g\|_2^2)
\end{align*}
where the boundary term of $\partial_Y(\bar{h}ug)$ vanishes by the boundary condition at $Y=0$ and by the zero mean of the far-field trace of $\bar{h}$.

We handle $I_6$ by using
\[
	h=\frac{-G^*-u\bar{h}+\bar{g}v+g\bar{v}-\partial_Y g}{\bar{u}},\qquad \bar{u}\geq \frac{a}{2}.
\]
\begin{align*}
	\big|I_6\big|= & \big|\int_{\Omega_p} -uh\partial_Y\bar{g}\,\ud\theta\,\ud Y \big|                                                                                                \\
	=              & \big|\int_{\Omega_p} -u \frac{-G^*-u\bar{h}+\bar{g}v+g\bar{v}-\partial_Y g}{\bar{u}} \partial_Y\bar{g}\,\ud\theta\,\ud Y \big|                                                         \\
	               & \leq C \int \big|\frac{u}{Y} YG^*\partial_Y\bar{g}\big|+\big|\frac{u}{Y}\frac{u}{Y}Y^2\bar{h}\partial_Y\bar{g}\big|+\big|\frac{u}{Y}\frac{v}{Y}Y^2\bar{g}\partial_Y\bar{g}\big|\,\ud\theta\,\ud Y \\
		       & \quad +C\int \big|\frac{u}{Y}\frac{g}{Y}Y^2\bar{v}\partial_Y\bar{g}\big|+\big|\frac{u}{Y}\partial_Y g\, Y\partial_Y\bar{g}\big|\,\ud\theta\,\ud Y \\
	               & \leq C\eta(\|\partial_\theta u\|_2^2+\|\partial_\theta g\|_2^2+\|\partial_Y u\|_2^2+\|\partial_Y g\|_2^2)+C\|YG^*\|_2^2
\end{align*}

Finally,
\begin{align*}
	     & \big|\int_{\Omega_p}F^*u +\partial_Y(\frac{G^*}{ \bar{u} })\bar{u} g\,\ud\theta\,\ud Y\big|                          \\
	\leq & \|YF^*\|_2^2+ \|Y\partial_Y(\frac{G^*}{ \bar{u} })\bar{u}\|_2^2+ \frac{\kappa}{2}\|\partial_Yu\|_2^2+   \frac{1}{2}\|\partial_Yg\|_2^2 \\[5pt]
\end{align*}
Collecting the estimates above and taking $\eta$ sufficiently small, we obtain
\begin{align}\label{energy-estimate-prandtl}
	 & \|\partial_Yu\|_2^2 +\|\partial_Yg\|_2^2                                                          \\
	 & \nonumber\leq C\eta [\|\partial_{\theta}u\|_2^2+\|\partial_{\theta}g\|_2^2] +C\|YF^*\|_2^2+C\|Y\partial_YG^*\|_2^2+C\|YG^*\|_2^2,
\end{align}
where $C$ is independent of $\eta$.
Next, we multiply the two equations in \eqref{re-new-linearized-prandtl-equation} by
\[
	\big(a\partial_{\theta}u-b\partial_{\theta}g,\ \kappa(a\partial_{\theta}g-b\partial_{\theta}u)\big),
\]
add them together, and integrate over ${\Omega_p}$. We obtain
\begin{align*}
	  & \int_{\Omega_p}\Big(\bar{u}\partial_\theta u+\bar{v}\partial_Yu+v\partial_Y\bar{u}+u\partial_\theta \bar{u}-\kappa\partial_{YY}u-\bar{g}\partial_\theta g-\bar{h}\partial_Yg-g\partial_\theta \bar{g}-h\partial_Y\bar{g}\Big)(a\partial_{\theta}u-b\partial_{\theta}g) \,\ud\theta\,\ud Y \\
	  & +\int_{\Omega_p}\big(\partial_\theta g -\partial_Y(\frac{\bar{h}}{\bar{u} })u-\frac{\bar{h}}{\bar{u} }\partial_Y u +\partial_Y(\frac{\bar{g}}{\bar{u} })v-\frac{\bar{g}}{\bar{u}} \partial_\theta u  +\frac{\bar{v}}{\bar{u} }\partial_Yg                                                                                   \\
	  & \quad\quad + \partial_Y(\frac{\bar{v}}{\bar{u}})g-\partial_Y(\frac{\partial_Y g}{ \bar{u} })  \big)\bar{u}\kappa(a\partial_{\theta}g-b\partial_{\theta}u)\,\ud\theta\,\ud Y                                                                                                                                                \\
	  & =\underbrace{\int_{\Omega_p}(\bar{u}\partial_\theta u-\bar{g}\partial_\theta g)(a\partial_{\theta}u-b\partial_{\theta}g)+\kappa(\bar{u}\partial_\theta g-\bar{g}\partial_\theta u)(a\partial_{\theta}g-b\partial_{\theta}u) \,\ud\theta\,\ud Y}_{I_1}                                                          \\
	  & +\underbrace{\kappa\int_{\Omega_p} -\partial_{YY}u(a\partial_{\theta}u-b\partial_{\theta}g)-\partial_Y(\frac{\partial_Y g}{ \bar{u} })\bar{u}(a\partial_{\theta}g-b\partial_{\theta}u) \,\ud\theta\,\ud Y}_{I_2}                                                                                                              \\
	  & +\underbrace{\int_{\Omega_p}\Big(\bar{v}\partial_Yu+v\partial_Y\bar{u}+u\partial_\theta \bar{u}-\bar{h}\partial_Yg-g\partial_\theta \bar{g}\Big)(a\partial_{\theta}u-b\partial_{\theta}g)\,\ud\theta\,\ud Y}_{I_3}                                                                        \\
	  & +\underbrace{\kappa\int_{\Omega_p}\big( -\partial_Y(\frac{\bar{h}}{\bar{u} })u-\frac{\bar{h}}{\bar{u} }\partial_Y u +\partial_Y(\frac{\bar{g}}{\bar{u} })v  +\frac{\bar{v}}{\bar{u} }\partial_Yg + \partial_Y(\frac{\bar{v}}{\bar{u}})g  \big)   (a\partial_{\theta}g-b\partial_{\theta}u)\,\ud\theta\,\ud Y}_{I_4}                       \\
	  & +\underbrace{\int_{\Omega_p}-h\partial_Y\bar{g} (a\partial_{\theta}u-b\partial_{\theta}g)\,\ud\theta\,\ud Y}_{I_5}                                                                                                                                                                                           \\
	= & \int_{\Omega_p}F^*(a\partial_{\theta}u-b\partial_{\theta}g)+\kappa\partial_Y G^* (a\partial_{\theta}g-b\partial_{\theta}u) \,\ud\theta\,\ud Y.
\end{align*}
Using $|a|\neq|b|$, we obtain
\begin{align*}
	I_1= & \int_{\Omega_p}(\bar{u}\partial_\theta u-\bar{g}\partial_\theta g)(a\partial_{\theta}u-b\partial_{\theta}g)+\kappa(\bar{u}\partial_\theta g-\bar{g}\partial_\theta u)(a\partial_{\theta}g-b\partial_{\theta}u) \,\ud\theta\,\ud Y \\
	=    & \int_{\Omega_p}(a\partial_{\theta}u-b\partial_{\theta}g)^2+\kappa(a\partial_{\theta}g-b\partial_{\theta}u)^2+(u_p^{(0)}\partial_\theta u-g_p^{(0)}\partial_\theta g)(a\partial_{\theta}u-b\partial_{\theta}g)                                     \\
	     & +\kappa(u_p^{(0)}\partial_\theta g-g_p^{(0)}\partial_\theta u)(a\partial_{\theta}g-b\partial_{\theta}u)                                                                                                                   \\
	\geq & (1-C\eta)\int_{\Omega_p}(a\partial_{\theta}u-b\partial_{\theta}g)^2+(a\partial_{\theta}g-b\partial_{\theta}u)^2                                                                                                           \\
	\geq & C'\int_{\Omega_p}(\partial_{\theta}u)^2+(\partial_{\theta}g)^2
\end{align*}
The diffusion term can be computed as follows
\begin{align*}
	  & I_2=\kappa\int_{\Omega_p} -\partial_{YY}u(a\partial_{\theta}u-b\partial_{\theta}g)-\partial_Y\Big(\frac{\partial_Y g}{\bar{u}}\Big)\bar{u}(a\partial_{\theta}g-b\partial_{\theta}u) \,\ud\theta\,\ud Y                             \\
	= & \kappa\int_{\Omega_p} -\partial_{YY}u(a\partial_{\theta}u-b\partial_{\theta}g)-\partial_{YY}g(a\partial_{\theta}g-b\partial_{\theta}u)+\frac{\partial_Y\bar{u}}{\bar{u}}\partial_Yg(a\partial_{\theta}g-b\partial_{\theta}u)\,\ud\theta\,\ud Y \\
	= & \kappa\int_{\Omega_p}\frac{\partial_Y\bar{u}}{\bar{u}}\partial_Yg(a\partial_{\theta}g-b\partial_{\theta}u)\,\ud\theta\,\ud Y,
\end{align*}
and hence
\[
	|I_2|\leq C\eta\big[\|\partial_Yu\|_2^2+\|\partial_\theta u\|_2^2+\|\partial_Yg\|_2^2+\|\partial_\theta g\|_2^2\big].
\]
Moreover,
\begin{align*}
	\big|I_3\big|= & \big|\int_{\Omega_p}\Big[\bar{v}\partial_Yu+v\partial_Y\bar{u}+u\partial_\theta \bar{u}-\bar{h}\partial_Yg-g\partial_\theta \bar{g}\Big](a\partial_{\theta}u-b\partial_{\theta}g)\big| \\
	\leq           & C\eta\big[\|\partial_Yu\|^2_2+\|\partial_\theta u\|^2_2+\|\partial_Yg\|^2_2+\|\partial_\theta g\|^2_2\big ].
\end{align*}
Similarly, we have
\begin{align*}
	\big|I_4\big|= & \big|\kappa\int_{\Omega_p}\big( -\partial_Y(\frac{\bar{h}}{\bar{u} })u-\frac{\bar{h}}{\bar{u} }\partial_Y u +\partial_Y(\frac{\bar{g}}{\bar{u} })v  +\frac{\bar{v}}{\bar{u} }\partial_Yg + \partial_Y(\frac{\bar{v}}{\bar{u}})g  \big)   (a\partial_{\theta}g-b\partial_{\theta}u)\,\ud\theta\,\ud Y\big| \\
	\leq           & C\eta\big[\|\partial_Yu\|^2_2+\|\partial_\theta u\|^2_2+\|\partial_Yg\|^2_2+\|\partial_\theta g\|^2_2\big ].
\end{align*}
We handle $I_5$ as in the previous energy estimate, using the expression for $h$:
\begin{align*}
	\big|I_5\big|= & \big|\int_{\Omega_p}-h\partial_Y\bar{g} (a\partial_{\theta}u-b\partial_{\theta}g)\big|                                                                                      \\
	\leq           & C\int_{\Omega_p}|\partial_Y\bar{g}|\, |a\partial_{\theta}u-b\partial_{\theta}g|\,
	\Big|\frac{-G^*-u\bar{h}+\bar{g}v+g\bar{v}-\partial_Yg}{\bar{u}}\Big|\,\ud\theta\,\ud Y \\
	\leq           & C\eta\big(\|\partial_Yu\|^2_2+\|\partial_\theta u\|^2_2+\|\partial_Yg\|^2_2+\|\partial_\theta g\|^2_2\big)+C\|G^*\|_2^2,
\end{align*}
and
\begin{align*}
	     & \big|\int_{\Omega_p}F^*(a\partial_{\theta}u-b\partial_{\theta}g)+\kappa\partial_YG^*(a\partial_{\theta}g-b\partial_{\theta}u)  \,\ud\theta\,\ud Y\big| \\
	\leq & \|F^*\|_2\| (a\partial_{\theta}u-b\partial_{\theta}g)\|_2+\kappa\|\partial_YG^*\|_2\| (a\partial_{\theta}g-b\partial_{\theta}u)\|_2,
\end{align*}
Thus
\begin{align*}
	\|\partial_{\theta}u\|_2^2+\|\partial_{\theta}g\|_2^2\leq C \eta\big[\|\partial_Y u\|_2^2+\|\partial_Y g\|_2^2\big]+C(\|F^*\|_2^2+\|\partial_YG^*\|_2^2+\|G^*\|_2^2).
\end{align*}
Combining this with \eqref{energy-estimate-prandtl}, we have
\begin{align}\label{lax-mil-prandtl}
	&\|\partial_Y u\|_2^2+\|\partial_Y g\|_2^2+\|\partial_{\theta}u\|_2^2+\|\partial_{\theta}g\|_2^2 \nonumber\\
	&\quad \leq C\|YF^*\|_2^2+C\|Y\partial_YG^*\|_2^2+C\|YG^*\|_2^2+C\|F^*\|_2^2+C\|\partial_YG^*\|_2^2+C\|G^*\|_2^2
\end{align}

We now justify solvability.  Keep the exact
nonlocal maps
\[
v[u](\theta,Y)=-\int_0^Y u_\theta(\theta,z)\,\ud z,
\qquad
h_0[u,g]=\frac{-u\bar h+\bar g\,v[u]+g\bar v-g_Y}{\bar u},
\qquad h=h_0[u,g]-\frac{G^*}{\bar u}.
\]
After the terms containing $G^*$ are moved to the right-hand side, the two
equations for $(u,g)$ are the first equation of
\eqref{re-new-linearized-prandtl-equation} with $h_0[u,g]$ in its linear part
and $g_\theta+(h_0[u,g])_Y=(G^*/\bar u)_Y$.  Denote this full linear
operator, including its nonlocal part, by \(\mathcal L\).  Let $\mathcal L_0$ be its
constant far-field part,
\begin{equation*}
\mathcal L_0\binom{u}{g}
=\binom{a u_\theta-b g_\theta-\kappa u_{YY}}
{a g_\theta-b u_\theta-g_{YY}}.
\end{equation*}
{The zero tangential mode is obtained by integrating twice in
\(Y\).  For the nonzero modes, Fourier series and \(a^2-b^2\ne0\) give a
bounded inverse of \(\mathcal L_0\); equivalently, this follows from
\eqref{lax-mil-prandtl} by a standard Lax--Milgram approximation.  The full
operator is a small perturbation of \(\mathcal L_0\), since the profile
bounds and Hardy's inequality give
\(\|\mathcal L_0^{-1}(\mathcal L-\mathcal L_0)\|\leq C\eta\).  Thus it is
invertible for small \(\eta\).}  We obtain the weighted bounds afterward by
applying the same energy estimates to the differentiated equations with the
weight $\langle Y\rangle^{2l}$ and arguing by induction on $l$.  Derivatives
of the weight give lower weights.  The limits $v_\infty(\theta)$ and
$h_\infty(\theta)$ only multiply $u_Y$ and $g_Y$, so they are controlled at
the current weight; all other coefficient differences decay rapidly.  This
gives the stated estimates for every $m,l$.  The boundary term at
$Y=-\infty$ vanishes because the limiting tangential constants multiply
zero-mean normal traces, while the remaining parts decay.

As in the nonlinear Prandtl construction, the zero Fourier modes at $Y=-\infty$
are not required to vanish. {We now identify the far-field traces of the
auxiliary solution.} Recall that
$A_{1\infty}=\lim_{Y\to-\infty}u$, and denote the other possible far-field
traces by
\[
	v_\infty(\theta):=\lim\limits_{Y\to-\infty}v,\qquad
	g_\infty:=\lim\limits_{Y\to-\infty}g,\qquad
	h_\infty(\theta):=\lim\limits_{Y\to-\infty}h.
\]
The weighted estimates above are imposed on $u-A_{1\infty}$, $v-v_\infty$, $g-g_\infty$ and $h-h_\infty$. These far-field traces are absorbed into the next Euler boundary data by setting
\[
	v_e^{(2)}(\theta,1)=v_\infty(\theta),\qquad
	h_e^{(2)}(\theta,1)=h_\infty(\theta).
\]
Then the Prandtl profiles are defined by
\begin{align*}
	u_p^{(1)}&=u-u_e^{(1)}(\theta,1)\chi_0(Y),\\
	v_p^{(2)}&=v-v_e^{(2)}(\theta,1)-Yv_p^{(1)}+\partial_\theta u_e^{(1)}(\theta,1)\int_0^Y\chi_0(z)dz,\\
	g_p^{(1)}&=g-g_e^{(1)}(\theta,1)\chi_0(Y),\\
	h_p^{(2)}&=h-h_e^{(2)}(\theta,1)-Yh_p^{(1)}+\partial_\theta g_e^{(1)}(\theta,1)\int_0^Y\chi_0(z)dz.
\end{align*}
Since $\int_{-\infty}^0\chi_0(z)\,\ud z=0$ and the lower-order profiles decay
rapidly, the desired boundary conditions and estimates follow from these
definitions once $g_\infty=0$ is verified below.  The zero means of
$v_p^{(2)}$ and $h_p^{(2)}$ follow from the periodic divergence equations and
the chosen far-field traces.

\medskip
\noindent\textbf{First magnetic far-field compatibility.}
Although $\lim_{Y\to-\infty}u_p^{(1)}$ may be nonzero, the magnetic trace
satisfies $\lim_{Y\to-\infty}g_p^{(1)}=0$.  Set $r=1+tY$ and write
\[
	\Phi_1(t)=t\phi_0+t^2\phi_1+O(t^3),
\]
where
\[
	\phi_0=\int_0^Y\bar u(\theta,Y')\,\ud Y',\qquad
	\phi_1=\int_0^Y
	\big(aY'+u_e^{(1)}(\theta,1)+u_p^{(1)}(\theta,Y')\big)\,\ud Y'.
\]
For a formal power series $Q(t)$, let $[t^m]Q$ denote the coefficient of
$t^m$.
Let
\[
	g^{a,1}(t)=\bar{g}+t\big(bY+g_e^{(1)}(\theta,1)+g_p^{(1)}\big),\qquad
	h^{a,1}(t)=t\bar{h}+t^2H_2,
\]
where
\[
	H_2=h_e^{(2)}(\theta,1)
	+Y\partial_rh_e^{(1)}(\theta,1)+h_p^{(2)}.
\]
The magnetic equation at this order can be written as
\[
	[t^2]\big(-\partial_r\Phi_1(t)\, r h^{a,1}(t)
	-\partial_\theta\Phi_1(t)\, g^{a,1}(t)\big)
	-\partial_Y g_p^{(1)}=\partial_Y(Yg_p^{(0)}).
\]
The magnetic divergence equation gives, to the required order,
\[
	-\Phi_{1,r}rh^{a,1}-\Phi_{1,\theta}g^{a,1}
	=-\partial_r(\Phi_1rh^{a,1})
	-\partial_\theta(\Phi_1g^{a,1}).
\]
Since $\partial_r=t^{-1}\partial_Y$, integration over
$\mathbb T\times(-R,0)$ leaves only the boundary values of
\[
	B(Y):=[t^3](\Phi_1rh^{a,1})
	=\phi_0(H_2+Y\bar h)+\phi_1\bar h.
\]
The value at $Y=0$ is zero because $\phi_0(0)=\phi_1(0)=0$.  To treat
$Y=-R$, set
\[
	v_i=v_e^{(i)}(\theta,1),\qquad h_i=h_e^{(i)}(\theta,1),
	\qquad i=1,2.
\]
As $Y\to-\infty$, the profile estimates and the divergence equations give
\[
	\phi_0=aY+c_0+o(1),\qquad
	\phi_1=\frac a2Y^2+
	\big(u_e^{(1)}(\theta,1)+A_{1\infty}\big)Y+c_1+o(1),
\]
with $(c_0)_\theta=-v_1$ and $(c_1)_\theta=-v_2$.  Moreover,
\[
	h_1=\frac ba v_1,\qquad
	ah_2=bv_2-A_{1\infty}h_1+g_\infty v_1.
\]
The second identity is the far-field limit of the magnetic equation at this
order.  In the integral of $B(Y)$ over $\mathbb T$, the coefficient of $Y^2$
vanishes by the zero means of the Euler normal traces.  The coefficient of
$Y$ vanishes by
\[
	\partial_rv_e^{(1)}(\theta,1)+v_1
	=-\partial_\theta u_e^{(1)}(\theta,1),\qquad
	(c_0)_\theta=-v_1,
\]
and the two relations for $h_1,h_2$.  The constant coefficient is
\[
	\frac ba\int_{\mathbb T}(c_0v_2+c_1v_1)\,\ud\theta
	+
	\left(\frac{g_\infty}{a}
	-\frac{bA_{1\infty}}{a^2}\right)
	\int_{\mathbb T}c_0v_1\,\ud\theta=0,
\]
because
\[
	\int_{\mathbb T}c_0v_2\,\ud\theta
	=-\int_{\mathbb T}c_1v_1\,\ud\theta,\qquad
	\int_{\mathbb T}c_0v_1\,\ud\theta
	=-\int_{\mathbb T}c_0(c_0)_\theta\,\ud\theta=0.
\]
Thus the lower boundary contribution tends to zero as $R\to\infty$.
Also,
$\int_{\Omega_p}\partial_Y(Yg_p^{(0)})\,\ud\theta\,\ud Y=0$, and hence
\[
	\int_0^{2\pi}\big(g_p^{(1)}(\theta,0)-g_\infty\big)\,\ud\theta=0.
\]
Using $g_p^{(1)}(\theta,0)=-g_e^{(1)}(\theta,1)$ and the zero mean of
$g_e^{(1)}$, we conclude that $g_\infty=0$.  This completes the proof.
\end{proof}
\begin{remark}
	We cannot expect $\lim_{Y\rightarrow-\infty}u^{(1)}_p=0$ because of
	periodicity in the tangential direction.  This nonzero constant trace does
	not affect the outer velocity correction.  Indeed, if
	$(u^{(1)}_e,v^{(1)}_e)$ solves the linearized Euler system, then
	$(u^{(1)}_e+A_{1\infty}r,v^{(1)}_e)$ does as well.  We therefore absorb
	$A_{1\infty}$ into the Euler profile.
\end{remark}
Next, we construct the pressure $p_p^{(2)}(\theta,Y)$. Consider the equation
\begin{align}\label{equation-for-second-pressure}
	\partial_Yp_p^{(2)}(\theta, Y)=g_1(\theta,Y), \quad \lim_{Y\rightarrow -\infty}p_p^{(2)}(\theta,Y)=0,
\end{align}
where
\begin{align*}
	g_1(\theta,Y)= & -Y\partial_Yp_p^{(1)}+\kappa \partial_{YY}v_p^{(1)}-u_e(1)\partial_\theta v_p^{(1)}-u_p^{(0)} (\partial_\theta v_e^{(1)}(\theta,1)+\partial_\theta v_p^{(1)}) \\[5pt]
	               & -\partial_Yv_p^{(1)}(v_e^{(1)}(\theta,1)+v_p^{(1)})-2(Yu'_e(1)u_p^{(0)}+u_e(1)\bar{u}_p^{(1)}+[u_e^{(1)}(\theta,1)+A_{1\infty}]u_p^{(0)}+u_p^{(0)}\bar{u}_p^{(1)})           \\
	               & +g_e(1)\partial_\theta h_p^{(1)}+g_p^{(0)} (\partial_\theta h_e^{(1)}(\theta,1)+\partial_\theta h_p^{(1)})                                                           \\[5pt]
	               & +\partial_Yh_p^{(1)}(h_e^{(1)}(\theta,1)+h_p^{(1)})+2(Yg'_e(1)g_p^{(0)}+g_e(1)g_p^{(1)}+g_e^{(1)}(\theta,1)g_p^{(0)}+g_p^{(0)}g_p^{(1)}).
\end{align*}
Here and below,
\[
\bar{u}_p^{(1)}=u_p^{(1)}-A_{1\infty}.
\]
$g_1(\theta,Y)$ is obtained by replacing $u_p^{(1)}$ with
$\bar{u}_p^{(1)}$ in \eqref{first-order-expansion}, substituting the resulting
expansion into the second equation of \eqref{mhd-polar}, and collecting the
terms of order $\varepsilon$.
Since $g_1(\theta,Y)$ decays rapidly as $Y\rightarrow -\infty$, \eqref{equation-for-second-pressure} determines $p_p^{(2)}(\theta,Y)$, and $p_p^{(2)}$ also decays rapidly as $Y\rightarrow -\infty$.

\subsubsection{Higher-order linearized Euler equations}

Assume that the Euler and Prandtl profiles up to order $k-1$ have already been
constructed.  As in the modification
$\bar{u}_e^{(1)}=u_e^{(1)}+A_{1\infty}r$, set
\[
\bar{u}_e^{(i)}(\theta,r):=u_e^{(i)}(\theta,r)+A_{i\infty}r,\qquad
\widetilde u_e^{(i)}:=\bar u_e^{(i)}-A_{i\infty}r=u_e^{(i)}.
\]
Substitute
\begin{align*}
	 & u^{\varepsilon}(\theta,r)=u_e(r)+\sum_{i=1}^{k-1}\varepsilon^{i} \bar{u}_e^{(i)}(\theta,r)+\varepsilon^{k} u_e^{(k)}(\theta,r)+\cdots, \\[-3pt] &v^{\varepsilon}(\theta,r)=\sum_{i=1}^{k-1}\varepsilon^{i} v_e^{(i)}(\theta,r)+\varepsilon^{k} v_e^{(k)}(\theta,r)+\cdots, \\[-3pt]
	 & g^{\varepsilon}(\theta,r)=g_e(r)+\sum_{i=1}^{k-1}\varepsilon^{i} g_e^{(i)}(\theta,r)+\varepsilon^{k} g_e^{(k)}(\theta,r)+\cdots,       \\[-3pt]
	 & h^{\varepsilon}(\theta,r)=\sum_{i=1}^{k-1}\varepsilon^{i} h_e^{(i)}(\theta,r)+\varepsilon^{k} h_e^{(k)}(\theta,r)+\cdots,               \\[-3pt]
	 & p^{\varepsilon}(\theta,r)=p_e(r)+\sum_{i=1}^{k}\varepsilon^{i} p_e^{(i)}(\theta,r)+\cdots
\end{align*}
into the MHD equations \eqref{mhd-polar}.  The order-$\varepsilon^k$ terms
give the following linearized MHD--Euler equations for
$(u_e^{(k)},v_e^{(k)},g_e^{(k)},h_e^{(k)},p_e^{(k)})$ in $\Omega$:
\begin{equation}\label{outer-4-order-equation}
\left \{
\begin {array}{ll}
ar \partial_\theta u_e^{(k)}+2arv_e^{(k)}-br \partial_\theta g_e^{(k)}-2brh_e^{(k)} +\partial_\theta p_e^{(k)}=F_{e}^{(k)}(\theta,r),\\[5pt]
ar \partial_\theta v_e^{(k)}-2aru_e^{(k)}-br \partial_\theta h_e^{(k)}+2brg_e^{(k)}+r\partial_rp_e^{(k)}=G_{e}^{(k)}(\theta,r),\\[5pt]
-arh_e^{(k)}+brv_e^{(k)}=H_e^{(k)},\\[5pt]
\partial_\theta u_e^{(k)}+r\partial_rv_e^{(k)}+ v_e^{(k)}=0,\\[5pt]
\partial_\theta g_e^{(k)}+r\partial_rh_e^{(k)}+ h_e^{(k)}=0,\\[5pt]
v_e^{(k)}|_{r=1}=-v_p^{(k)}|_{Y=0},\quad v_e^{(k)}(\theta,r)=v_e^{(k)}(\theta+2\pi,r) ,
\end{array}
\right.
\end{equation}
where
\begin{align*}
	F_{e}^{(k)}(\theta,r)= & -\sum_{i=1}^{k-1}\Big(\bar{u}_e^{(i)}\partial_\theta \bar{u}_e^{(k-i)}+v_e^{(i)}r\partial_r\bar{u}_e^{(k-i)}+\bar{u}_e^{(i)}v_e^{(k-i)}\Big)                                                             \\
	                       & +\sum_{i=1}^{k-1}\Big(g_e^{(i)}\partial_\theta g_e^{(k-i)}+h_e^{(i)}r\partial_rg_e^{(k-i)}+g_e^{(i)}h_e^{(k-i)}\Big)                                                                                     \\
	                       & +\underbrace{\kappa\Big(\frac{\partial_{\theta\theta}\bar{u}_e^{(k-2)}}{r}+r\partial_{rr}\bar{u}_e^{(k-2)}+\partial_r\bar{u}_e^{(k-2)}
	+\frac{2}{r}\partial_\theta v_e^{(k-2)}-\frac{\bar{u}_e^{(k-2)}}{r}\Big)}_{=0},                                                                                                                                         \\
	=                      & -\sum_{i=1}^{k-1}\Big(\bar{u}_e^{(i)}\partial_\theta \bar{u}_e^{(k-i)}+v_e^{(i)}r\partial_r\bar{u}_e^{(k-i)}+\bar{u}_e^{(i)}v_e^{(k-i)}\Big)                                                             \\
	                       & +\sum_{i=1}^{k-1}\Big(g_e^{(i)}\partial_\theta g_e^{(k-i)}+h_e^{(i)}r\partial_rg_e^{(k-i)}+g_e^{(i)}h_e^{(k-i)}\Big),                                                                                     \\[3pt]
	G_{e}^{(k)}(\theta,r)= & -\sum_{i=1}^{k-1}\Big(\bar{u}_e^{(i)}\partial_\theta v_e^{(k-i)}+v_e^{(i)}r\partial_r v_e^{(k-i)}-\bar{u}_e^{(i)}\bar{u}_e^{(k-i)}\Big)                                                                 \\
	                       & +\sum_{i=1}^{k-1}\Big(g_e^{(i)}\partial_\theta h_e^{(k-i)}+h_e^{(i)}r\partial_r h_e^{(k-i)}-g_e^{(i)}g_e^{(k-i)}\Big)                                                                                   \\
	                       & +\underbrace{\kappa\Big(\frac{\partial_{\theta\theta}v_e^{(k-2)}}{r}+r\partial_{rr}v_e^{(k-2)}+\partial_rv_e^{(k-2)}
	-\frac{2}{r}\partial_\theta \bar{u}_e^{(k-2)}-\frac{v_e^{(k-2)}}{r}\Big)}_{=0}                                                                                                                                          \\
	=                      & -\sum_{i=1}^{k-1}\Big(\bar{u}_e^{(i)}\partial_\theta v_e^{(k-i)}+v_e^{(i)}r\partial_r v_e^{(k-i)}-\bar{u}_e^{(i)}\bar{u}_e^{(k-i)}\Big)                                                                 \\
	                       & +\sum_{i=1}^{k-1}\Big(g_e^{(i)}\partial_\theta h_e^{(k-i)}+h_e^{(i)}r\partial_r h_e^{(k-i)}-g_e^{(i)}g_e^{(k-i)}\Big),                                                                                   \\[3pt]
	H_{e}^{(k)}(\theta,r)= & \sum_{i=1}^{k-1} \bar{u}_e^{(i)}h_e^{(k-i)}-g_e^{(i)}v_e^{(k-i)}+\underbrace{\frac{1}{r}\Big(r\partial_{r}g_e^{(k-2)}+g_e^{(k-2)}-\partial_\theta h_e^{(k-2)} \Big)-2\beta 1_{\{k=2\}}}_{=0} \\
	=                      & \sum_{i=1}^{k-1} \bar{u}_e^{(i)}h_e^{(k-i)}-g_e^{(i)}v_e^{(k-i)},
\end{align*}
where $1_{\{k=2\}}$ denotes the Kronecker delta: $1_{\{k=2\}}=1$ if and only
if $k=2$.  This calculation also shows that the magnetic part of the
MHD--Euler flow must be $(\beta r,0)$.  Otherwise
$\int_0^{2\pi}H_{e}^{(2)}\neq 0$, which violates periodic compatibility.

\begin{proposition}
	The linearized MHD--Euler equations \eqref{outer-4-order-equation} have a solution
	for each $2\leq k\leq11$,
	\[
		(\bar{u}_e^{(k)}, v_e^{(k)},g_e^{(k)},h_e^{(k)}, p_e^{(k)})
	\]
	satisfying
	\begin{align}\label{estimate-of-fourth-linearized-euler-equation}
		 & |\partial_\theta \bar{u}_e^{(k)}+v_e^{(k)}|(\theta,r)+|\partial_\theta g_e^{(k)}+h_e^{(k)}|(\theta,r)\leq C\eta r, \ \forall (\theta, r)\in \Omega,\nonumber       \\
		 & \ |\partial_\theta v_e^{(k)}-\bar{u}_e^{(k)}|(\theta,r)+ \ |\partial_\theta h_e^{(k)}-g_e^{(k)}|(\theta,r)\leq C\eta r,  \ \forall (\theta, r)\in \Omega,\nonumber \\[5pt]
		 & \|\partial^i_\theta\partial^j_r(\bar{u}_e^{(k)},v_e^{(k)},g_e^{(k)},h_e^{(k)})\|_2\leq C(i,j,k)\eta, \quad \forall i,j\geq 0;                                      \\[5pt]
		 & r^2\Delta \bar{u}_e^{(k)} -\bar{u}_e^{(k)}+2\partial_\theta v_e^{(k)}= r^2\Delta g_e^{(k)} -g_e^{(k)}+2\partial_\theta h_e^{(k)}= 0,\nonumber                \\[5pt]
		 & \int_{0}^{2\pi} v_e^{(k)}\,\ud\theta=\int_{0}^{2\pi} h_e^{(k)}\,\ud\theta=\int_{0}^{2\pi} g_e^{(k)}\,\ud\theta=0.\nonumber
	\end{align}
	In addition,
	\begin{align}
		H_{e}^{(k)}=\sum_{i=1}^{k-1} A_{i\infty}rh_e^{(k-i)}\label{eq-of-h_e-k}.
	\end{align}
\end{proposition}
\begin{proof}
	We first prove \eqref{eq-of-h_e-k}. Assume inductively that
	\begin{align}\label{eq-of-h_e-j}
		H_{e}^{(j)}=\sum_{i=1}^{j-1} A_{i\infty}rh_e^{(j-i)},\quad \forall j\leq k-1.
	\end{align}
	Since $H_{e}^{(j)}=\sum_{i=1}^{j-1} \bar{u}_e^{(i)}h_e^{(j-i)}-g_e^{(i)}v_e^{(j-i)}$ for $j\leq k-1$, the inductive assumption \eqref{eq-of-h_e-j} is equivalent to
	\begin{align}\label{property-of-hek}
		\sum_{i=1}^{j-1} \widetilde u_e^{(i)}h_e^{(j-i)}-g_e^{(i)}v_e^{(j-i)}=0.
	\end{align}
	Combining $-arh_e^{(j)}+brv_e^{(j)}=H_e^{(j)}$ with the divergence-free conditions and the zero-mean condition for $g_e^{(j)}$, we have
	\begin{align*}
		h_e^{(j)}=\frac{b}{a}v_e^{(j)}-\sum_{i=1}^{j-1} \frac{A_{i\infty}}{a}h_e^{(j-i)}, \\
		g_e^{(j)}=\frac{b}{a}\widetilde u_e^{(j)}-\sum_{i=1}^{j-1} \frac{A_{i\infty}}{a}g_{e}^{(j-i)}.
	\end{align*}
	Then
	\begin{align*}
		 & \quad\sum_{i=1}^{k-1}\widetilde u_e^{(i)}h_e^{(k-i)}-g_{e}^{(i)}v_e^{(k-i)}                                                                                                                                                                \\
		 & = \sum_{i=1}^{k-1} \widetilde u_e^{(i)}h_e^{(k-i)}- \sum_{i=1}^{k-1} g_{e}^{(k-i)}v_e^{(i)}                                                                                                                                                \\
		 & =\sum_{i=1}^{k-1} \widetilde u_e^{(i)}   ( \frac{b}{a}v_e^{(k-i)}-\sum_{j=1}^{k-i-1} \frac{A_{j\infty}}{a}h_e^{(k-i-j)} )- \sum_{i=1}^{k-1}v_e^{(i)} ( \frac{b}{a}\widetilde u_e^{(k-i)}-\sum_{j=1}^{k-i-1} \frac{A_{j\infty}}{a}g_{e}^{(k-i-j)} ) \\
		 & =\sum_{i=1}^{k-1} \sum_{j=1}^{k-i-1}\frac{A_{j\infty}}{a}(-\widetilde u_e^{(i)} h_e^{(k-i-j)} + g_{e}^{(k-i-j)} v_e^{(i)})                                                                                                                   \\
		 & =\sum_{j=1}^{k-2}\sum_{i=1}^{k-j-1}\frac{A_{j\infty}}{a}(-\widetilde u_e^{(i)} h_e^{(k-i-j)} + g_{e}^{(k-i-j)} v_e^{(i)})                                                                                                                    \\
		 & =\sum_{j=1}^{k-2}\frac{A_{j\infty}}{a} \sum_{i=1}^{k-j-1}(-\widetilde u_e^{(i)} h_e^{(k-i-j)} + g_{e}^{(k-i-j)} v_e^{(i)})                                                                                                                   \\
		 & =\sum_{j=1}^{k-2}\frac{A_{j\infty}}{a} \sum_{p+q=k-j}(-\widetilde u_e^{(p)} h_e^{(q)} + g_{e}^{(p)} v_e^{(q)})                                                                                                                               \\
		 & =0.
	\end{align*}
	The last equality follows from \eqref{property-of-hek}.  Hence
	\begin{align*}
		-arh_e^{(k)}+brv_e^{(k)}=\sum_{i=1}^{k-1} A_{i\infty}rh_e^{(k-i)}.
	\end{align*}
	After eliminating the pressure, the curl of the source vanishes.  Indeed,
	grouping the ordered lower-order pairs and using the divergence-free
	conditions reduces the curl to derivatives of the lower-order vorticities
	and currents.  These vanish because the velocity corrections have constant
	vorticity and the magnetic corrections have zero current.
	This is the same as $\Delta(rv_e^{(i)})=\Delta(rh_e^{(i)})=0$ for all
	$i\leq k-1$.  We therefore obtain
	\begin{align*}
		-a\Delta(rv_e^{(k)})+b\Delta(rh_e^{(k)})= -\frac{a^2-b^2}{a}\Delta(rv_e^{(k)}) =0.
	\end{align*}
	Since $|a|\neq|b|$,
	\begin{align*}
		\Delta(rv_e^{(k)}) =0.
	\end{align*}
	The remaining properties in \eqref{estimate-of-fourth-linearized-euler-equation} then follow from the harmonic representation of $rv_e^{(k)}$, the algebraic relation above, and the divergence-free equations.
\end{proof}
\subsubsection{Higher-order linearized Prandtl equations}

For $i\geq1$, let
$\bar{u}_p^{(i)}(\theta,Y)=u_p^{(i)}(\theta,Y)-A_{i\infty}$.  We also set
\[
	A_{0\infty}=0,\qquad \bar u_p^{(0)}=u_p^{(0)},\qquad
	\bar u_e^{(0)}=u_e,
\]
and every profile with a negative superscript is understood to be zero.  We
make the following formal expansion
\begin{align*}
	u^\varepsilon(\theta,1+\varepsilon Y)= & u_e(1+\varepsilon Y)+u_p^{(0)}(\theta,Y)+\sum_{i=1}^{k-1}\varepsilon^i\big[ \bar{u}_e^{(i)}(\theta,1+\varepsilon Y)+\bar{u}_p^{(i)}(\theta,Y)\big] \\[-3pt]
	                                    & +\varepsilon^k\big[u_e^{(k)}(\theta,1+\varepsilon Y)+ u_p^{(k)}(\theta,Y)\big]+\cdots,                                                              \\[-3pt]
	v^\varepsilon(\theta,1+\varepsilon Y)= & \sum_{i=1}^{k}\varepsilon^i\big[v_e^{(i)}(\theta,1+\varepsilon Y)+v_p^{(i)}(\theta,Y)\big]
	+\varepsilon^{k+1}[v_e^{(k+1)}(\theta,1+\varepsilon Y) +v_p^{(k+1)}(\theta,Y)]+\cdots,                                                                                                  \\[-3pt]
	g^\varepsilon(\theta,1+\varepsilon Y)= & g_e(1+\varepsilon Y)+g_p^{(0)}(\theta,Y)+\sum_{i=1}^{k-1}\varepsilon^i\big[g_e^{(i)}(\theta,1+\varepsilon Y)+g_p^{(i)}(\theta,Y)\big]
	\\[-3pt]
	                                    & +\varepsilon^k\big[g_e^{(k)}(\theta,1+\varepsilon Y)+ g_p^{(k)}(\theta,Y)\big]+\cdots,                                                              \\[-3pt]
	h^\varepsilon(\theta,1+\varepsilon Y)= & \sum_{i=1}^{k}\varepsilon^i\big[h_e^{(i)}(\theta,1+\varepsilon Y)+h_p^{(i)}(\theta,Y)\big]
	+\varepsilon^{k+1}[h_e^{(k+1)}(\theta,1+\varepsilon Y) +h_p^{(k+1)}(\theta,Y)]+\cdots,                                                                                               \\[-3pt]
	p^\varepsilon(\theta,1+\varepsilon Y)= & p_e(1+\varepsilon Y)+\sum_{i=1}^{k}\varepsilon^i\big[p_e^{(i)}(\theta,1+\varepsilon Y)+p_p^{(i)}(\theta,Y)\big]+\cdots
\end{align*}
Substituting these expansions into \eqref{mhd-polar}, with the following boundary conditions for $2\leq k\leq 10$,
\begin{align*}
	&u_e^{(k)}(\theta,1)+u_p^{(k)}(\theta,0)=0,\quad g_e^{(k)}(\theta,1)+g_p^{(k)}(\theta,0)=0,\\
	&v_e^{(k+1)}(\theta,1)+v_p^{(k+1)}(\theta,0)=0,\quad
	\lim_{Y\rightarrow -\infty}(\partial_Yu_p^{(k)},v_p^{(k+1)},\partial_Yg_p^{(k)},h_p^{(k+1)})=(0,0,0,0).
\end{align*}
At the top order $k=11$, no twelfth Euler profile is
introduced.  We take $v_e^{(12)}(\theta,1)=h_e^{(12)}(\theta,1)=0$ and
impose $v_p^{(12)}(\theta,0)=0$.  We retain the zero-mean limits
\[
V_{12,\infty}(\theta)=\lim_{Y\to-\infty}v_p^{(12)}(\theta,Y),
\qquad
H_{12,\infty}(\theta)=\lim_{Y\to-\infty}h_p^{(12)}(\theta,Y)
\]
and also allow $g_p^{(11)}$ to have a constant limit.  After subtracting
these limits, the profiles decay rapidly.  Their cutoff terms are of order
$\varepsilon^{11}$ or higher and are covered by the residual estimate below.
For $2\leq k\leq11$, define the pressure profile before writing the
$k$th Prandtl system by
\[
	\partial_Yp_p^{(k)}=g_{k-1},\qquad
	p_p^{(k)}(\theta,-\infty)=0,
\]
where $g_{k-1}$ depends only on profiles constructed at lower orders.  Its
explicit expression is recorded below.
We obtain the following linearized Prandtl problem for $(u_p^{(k)},v_p^{(k+1)},g_p^{(k)},h_p^{(k+1)})$
\begin{equation}
\left \{
\begin {array}{ll}
\big(u_e(1)+u_p^{(0)}\big)\partial_\theta u_p^{(k)}+\big(v_e^{(1)}(\theta,1)+ v_p^{(1)}\big)\partial_Yu_p^{(k)}+u_p^{(k)}\partial_\theta u_p^{(0)} \\[5pt]
\quad \quad \quad + (v_p^{(k+1)}+v_e^{(k+1)}(\theta,1))\partial_{Y}u_p^{(0)}
-\big(g_e(1)+g_p^{(0)}\big)\partial_\theta g_p^{(k)}
-\big(h_e^{(1)}(\theta,1)+h_p^{(1)}\big)\partial_Yg_p^{(k)}\\[5pt]
\quad \quad \quad -g_p^{(k)}\partial_\theta g_p^{(0)}
-\big(h_e^{(k+1)}(\theta,1)+h_p^{(k+1)}\big)\partial_Yg_p^{(0)}
-\kappa\partial_{YY}u_p^{(k)}=f_p^{(k)},\\[5pt]
-(u_e(1)+u_p^{(0)})(h_e^{(k+1)}(\theta,1)+h_p^{(k+1)})-u_p^{(k)}(h_e^{(1)}(\theta,1)+h_p^{(1)})\\[5pt]
+(g_e(1)+g_p^{(0)})(v_e^{(k+1)}(\theta,1)+v_p^{(k+1)})+g_p^{(k)}(v_e^{(1)}(\theta,1)+v_p^{(1)})-\partial_Yg_p^{(k)} =\mathcal{G}_p^{(k)},\\[5pt]
\partial_\theta u_p^{(k)}+\partial_Yv_p^{(k+1)}+\partial_Y(Yv_p^{(k)})=0,\\[5pt]
\partial_\theta g_p^{(k)}+\partial_Yh_p^{(k+1)}+\partial_Y(Yh_p^{(k)})=0,\\[5pt]
u_p^{(k)}(\theta,Y)=u_p^{(k)}(\theta+2\pi,Y),\quad v_p^{(k+1)}(\theta,Y)=v_p^{(k+1)}(\theta+2\pi,Y),\\[5pt]
g_p^{(k)}(\theta,Y)=g_p^{(k)}(\theta+2\pi,Y),\quad h_p^{(k+1)}(\theta,Y)=h_p^{(k+1)}(\theta+2\pi,Y),\\[5pt]
u_p^{(k)}\big|_{Y=0}=-u_e^{(k)}\big|_{r=1},\quad g_p^{(k)}\big|_{Y=0}=-g_e^{(k)}\big|_{r=1},\\[5pt]
{
\begin{cases}
\displaystyle\lim_{Y\to-\infty}
(\partial_Yu_p^{(k)},v_p^{(k+1)},\partial_Yg_p^{(k)},h_p^{(k+1)})=0,
&2\leq k\leq10,\\[3pt]
\displaystyle\lim_{Y\to-\infty}
(\partial_Yu_p^{(11)},v_p^{(12)}-V_{12,\infty},
\partial_Yg_p^{(11)},h_p^{(12)}-H_{12,\infty})=0,&k=11,
\end{cases}}
\label{third-linearized-prandtl-problem-near-1}
\end{array}
\right.
\end{equation}
where the known source terms $f_p^{(k)}$ and $\mathcal{G}_p^{(k)}$ are given by
	{\small \begin{align*}
			f_p^{(k)}(\theta,Y)= & -\partial_\theta p_p^{(k)}+\kappa Y\partial_{YY}u_p^{(k-1)}
			+\kappa\partial_Yu_p^{(k-1)}
			+\sum_{l=0}^{k-2}(-1)^l(k-l-3)\frac{Y^l\partial_{\theta\theta}\bar{u}_p^{(k-l-2)}}{l!} \\
			               & +2\partial_\theta v_p^{(k-2)}-2Y\partial_\theta v_p^{(k-3)}
			-\sum_{l=0}^{k-2}(-1)^l\frac{Y^l\bar{u}_p^{(k-2-l)}}{l!} \\
			               & -\sum_{i+j=k, 1\leq i\leq k-1}\bar{u}_p^{(i)}\partial_\theta\bar{u}_p^{(j)}
			-\sum_{i+j=k }[v_p^{(i)}Y\partial_Y\bar{u}_p^{(j)}+\bar{u}_p^{(i)}v_p^{(j)}]                          \\
			               & -\sum_{l=0}^k\sum_{i+j=k-l, (l,j)\neq (0,k)}\Big(\frac{\partial_r^l\bar{u}_e^{(i)}(\theta,1)}{l!}Y^l\partial_\theta\bar{u}_p^{(j)}
			+\bar{u}_p^{(i)}\frac{\partial_r^l\partial_\theta\bar{u}_e^{(j)}(\theta,1)}{l!}Y^l\Big)                                                                                                                                                                   \\
			               & -\sum_{l=0}^k\sum_{i+j=k-l}\Big(\frac{\partial_r^lv_e^{(i)}(\theta,1)}{l!}Y^{l+1}\partial_Y\bar{u}_p^{(j)}
			+v_p^{(i)}\frac{\partial_r^l(r\partial_r\bar{u}_e^{(j)})(\theta,1)}{l!}Y^l\Big)                                                                                                                                                                           \\
			               & -\sum_{l=0}^k\sum_{i+j=k+1-l, (l,j)\neq (0,k),i\leq k}\frac{\partial_r^lv_e^{(i)}(\theta,1)}{l!}Y^l\partial_Y\bar{u}_p^{(j)}                                                                                                             \\
			               & -\sum_{\substack{i+j=k+1\\2\leq i\leq k,\ 1\leq j\leq k-1}} v_p^{(i)}\partial_Y\bar{u}_p^{(j)}
			               +\sum_{\substack{i+j=k+1\\2\leq i\leq k,\ 1\leq j\leq k-1}} h_p^{(i)}\partial_Yg_p^{(j)}                                                                                                                        \\
			               & -\sum_{l=0}^k\sum_{i+j=k+1-l}\Big(\frac{\partial_r^l\bar{u}_e^{(i)}(\theta,1)}{l!}Y^{l}v_p^{(j)}
			+\bar{u}_p^{(i)}\frac{\partial_r^lv_e^{(j)}(\theta,1)}{l!}Y^l\Big)                                                                                                                                                                                        \\
			               & +\sum_{i+j=k, 1\leq i\leq k-1}g_p^{(i)}\partial_\theta g_p^{(j)}
			+\sum_{i+j=k }[h_p^{(i)}Y\partial_Yg_p^{(j)}+g_p^{(i)}h_p^{(j)}]                                                                                                         \\
			               & +\sum_{l=0}^k\sum_{i+j=k-l, (l,j)\neq (0,k)}\Big(\frac{\partial_r^l g_e^{(i)}(\theta,1)}{l!}Y^l\partial_\theta g_p^{(j)}
			+g_p^{(i)}\frac{\partial_r^l\partial_\theta g_e^{(j)}(\theta,1)}{l!}Y^l\Big)                                                                                                                                                                              \\
			               & +\sum_{l=0}^k\sum_{i+j=k-l}\Big(\frac{\partial_r^lh_e^{(i)}(\theta,1)}{l!}Y^{l+1}\partial_Yg_p^{(j)}
			+h_p^{(i)}\frac{\partial_r^l(r\partial_rg_e^{(j)})(\theta,1)}{l!}Y^l\Big)                                                                                                                                                                         \\
			               & +\sum_{l=0}^k\sum_{i+j=k+1-l, (l,j)\neq (0,k),i\leq k}\frac{\partial_r^lh_e^{(i)}(\theta,1)}{l!}Y^l\partial_Yg_p^{(j)}                                                                                                                   \\
			               & +\sum_{l=0}^k\sum_{i+j=k+1-l}\Big(\frac{\partial_r^l g_e^{(i)}(\theta,1)}{l!}Y^{l}h_p^{(j)}
			+g_p^{(i)}\frac{\partial_r^lh_e^{(j)}(\theta,1)}{l!}Y^l\Big),
		\end{align*}}
and
	{\small \begin{align*}
			\mathcal{G}_p^{(k)}(\theta,Y)= & \sum_{i=1}^{k-1}[u_p^{(i)}+\sum_{l=0}^i\frac{1}{l!}\partial_r^lu_e^{(i-l)}(\theta,1)Y^l]h_p^{(k-i)}  \\
			               & +{\sum_{i=1}^{k-1}u_p^{(i)}
			               \sum_{l=0}^{k-i}\frac{1}{l!}\partial_r^lh_e^{(k-i-l)}(\theta,1)Y^l}    \\
			               & -\sum_{i=1}^{k-1}[g_p^{(i)}+\sum_{l=0}^i\frac{1}{l!}\partial_r^lg_e^{(i-l)}(\theta,1)Y^l]v_p^{(k-i)} \\
			               & -{\sum_{i=1}^{k-1}g_p^{(i)}
			               \sum_{l=0}^{k-i}\frac{1}{l!}\partial_r^lv_e^{(k-i-l)}(\theta,1)Y^l} \\
			               & +\sum_{l=0}^{k-2}Y^{k-l-2}g_p^{(l)}+\sum_{l=0}^{k-2}(k-l-3)Y^{k-l-2}h_p^{(l)}.
			\end{align*}}

\begin{proposition}[Higher-order MHD--Prandtl profiles]
	There exists $\eta_0>0$ such that, for any $\eta\in(0,\eta_0)$,
	\eqref{third-linearized-prandtl-problem-near-1} has a solution
	$(u_p^{(k)},v_p^{(k+1)},g_p^{(k)},h_p^{(k+1)})$ for $2\leq k\leq11$,
	unique in the weighted class below.
	Define
	\[
	\mathcal Q_k=(u_p^{(k)}-A_{k\infty},
	v_p^{(k+1)}-V_{k+1,\infty},
	g_p^{(k)}-\lim_{Y\to-\infty}g_p^{(k)},
	h_p^{(k+1)}-H_{k+1,\infty}).
	\]
	\begin{align}\label{decay-behavior-prandtl-3}
		 & {\sum_{i+j\leq m}\int_{-\infty}^0\int_0^{2\pi}
		 \big|\partial_\theta^i\partial_Y^j\mathcal Q_k\big|^2
		 \langle Y\rangle^{2l}\,\ud\theta\,\ud Y
		 \leq C(m,l,k)\eta^2,}\quad m,l\geq 0, \nonumber \\
		 & \int_0^{2\pi}v_p^{(k+1)}(\theta,Y)\,\ud\theta=\int_0^{2\pi}h_p^{(k+1)}(\theta,Y)\,\ud\theta=0,\qquad \forall \ Y\leq 0,
	\end{align}
	Here $A_{k\infty}=\lim_{Y\to-\infty}u_p^{(k)}$,
	$V_{k+1,\infty}=H_{k+1,\infty}=0$ for $2\leq k\leq10$, while the two
	top-order normal traces are those defined above.  For $2\leq k\leq10$,
	$\lim_{Y\to-\infty}g_p^{(k)}=0$; no such normalization is imposed at
	$k=11$.  All these limits are bounded by $C(k)\eta$.
\end{proposition}
\begin{proof}
{
The induction hypothesis gives the displayed weighted bounds for
$f_p^{(k)}$, $\mathcal G_p^{(k)}$, and their derivatives.  After lifting the
boundary values of $u_p^{(k)}$ and $g_p^{(k)}$, the divergence and scalar
induction equations recover the normal variables exactly as in the
first-order problem.  The resulting equation for the two tangential
variables has the same operator $\mathcal L$ as above.  Its inverse and the
source bounds give existence, uniqueness, and
\eqref{decay-behavior-prandtl-3}.

For $2\leq k\leq10$, the normal traces are absorbed into the next Euler
profiles, and integration of the scalar induction equation, as in the
first-order problem, gives $\lim_{Y\to-\infty}g_p^{(k)}=0$.  At $k=11$ the
remaining far-field limits are simply retained.
}
\end{proof}
For completeness, the source $g_{k-1}$ used above is
	{\small \begin{align*}
			g_{k-1}(\theta,Y)= & \kappa\partial_{YY}v_p^{(k-1)}
			+\kappa Y\partial_{YY}v_p^{(k-2)}+\kappa\partial_{Y}v_p^{(k-2)}
			+\kappa\partial_{\theta\theta}v_p^{(k-3)} \\[5pt]
			                   & -\sum_{i=0}^{k-3}(-1)^i\frac{Y^i}{i!}\partial_\theta \bar{u}_p^{(k-i-1)}
			-v_p^{(k-3)}-Y\partial_Yp_p^{(k-1)}                                                                                                                                       \\[5pt]
			                   & -\sum_{i+j=k-1}[\bar{u}_p^{(i)}\partial_\theta v_p^{(j)}
			+v_p^{(i)}Y\partial_Y v_p^{(j)}-\bar{u}_p^{(i)}\bar{u}_p^{(j)}]
			-\sum_{i+j=k} v_p^{(i)}Y\partial_Y v_p^{(j)} \\[5pt]
			                   & -\sum_{l=0}^{k-2}\sum_{i+j=k-l}\frac{\partial_r^lv_e^{(i)}(\theta,1)}{l!}Y^{l}\partial_Y v_p^{(j)} \\[5pt]
			                   & -\sum_{l=0}^{k-1}\sum_{i+j=k-l-1}\Big(\frac{\partial_r^l\bar{u}_e^{(i)}(\theta,1)}{l!}Y^{l}\partial_\theta v_p^{(j)}
			+\bar{u}_p^{(i)}\frac{\partial_r^l\partial_\theta v_e^{(j)}(\theta,1)}{l!}Y^l\Big)                                                                                                                                                                                                         \\[5pt]
			                   & +\sum_{l=0}^{k-1}\sum_{i+j=k-l-1}\Big(\frac{\partial_r^l\bar{u}_e^{(i)}(\theta,1)}{l!}Y^{l}\bar{u}_p^{(j)}
			+\bar{u}_p^{(i)}\frac{\partial_r^l\bar{u}_e^{(j)}(\theta,1)}{l!}Y^l\Big)                                                                                                                                                                                                                   \\[5pt]
			                   & -\sum_{l=0}^{k-3}\sum_{i+j=k-l-1}\Big(\frac{\partial_r^lv_e^{(i)}(\theta,1)}{l!}Y^{l+1}\partial_Y v_p^{(j)}
			+v_p^{(i)}\frac{\partial_r^l(r\partial_r v_e^{(j)})(\theta,1)}{l!}Y^l\Big)                                                                                                                                                                                                                 \\
			                   & +\sum_{i+j=k-1}[g_p^{(i)}\partial_\theta h_p^{(j)}
			+h_p^{(i)}Y\partial_Y h_p^{(j)}-g_p^{(i)}g_p^{(j)}] \\[5pt]
			                   & +\sum_{i+j=k} h_p^{(i)}Y\partial_Y h_p^{(j)}
			+\sum_{l=0}^{k-2}\sum_{i+j=k-l}\frac{\partial_r^lh_e^{(i)}(\theta,1)}{l!}Y^{l}\partial_Y h_p^{(j)}                   \\[5pt]
			                   & +\sum_{l=0}^{k-1}\sum_{i+j=k-l-1}\Big(\frac{\partial_r^lg_e^{(i)}(\theta,1)}{l!}Y^{l}\partial_\theta h_p^{(j)}
			+g_p^{(i)}\frac{\partial_r^l\partial_\theta h_e^{(j)}(\theta,1)}{l!}Y^l\Big)                                                                                                                                                                                                               \\[5pt]
			                   & -\sum_{l=0}^{k-1}\sum_{i+j=k-l-1}\Big(\frac{\partial_r^lg_e^{(i)}(\theta,1)}{l!}Y^{l}g_p^{(j)}
			+g_p^{(i)}\frac{\partial_r^lg_e^{(j)}(\theta,1)}{l!}Y^l\Big)                                                                                                                                                                                                                               \\[5pt]
			                   & +\sum_{l=0}^{k-3}\sum_{i+j=k-l-1}\Big(\frac{\partial_r^lh_e^{(i)}(\theta,1)}{l!}Y^{l+1}\partial_Y h_p^{(j)}
			+h_p^{(i)}\frac{\partial_r^l(r\partial_r h_e^{(j)})(\theta,1)}{l!}Y^l\Big).
		\end{align*}}
The same normal equation defines $p_p^{(12)}$, normalized by
$p_p^{(12)}(\theta,0)=0$.
For $k\leq11$ the source decays rapidly.  At $k=12$ it may have
a constant limit, so $p_p^{(12)}$ has at most affine growth.

\subsection{Approximate solutions}

In this subsection, we construct an approximate solution of the MHD equations \eqref{mhd-polar}.
Fix $\chi\in C^\infty([0,1])$ such that
$\chi=0$ on $[0,1/2]$ and $\chi=1$ on $[3/4,1]$.
Set
\begin{align*}
	u_p^a(\theta,r)&:=\chi(r)\Big(u_p^{(0)}(\theta,Y)+\sum_{i=1}^{11}\varepsilon^i\bar{u}_p^{(i)}(\theta,Y)\Big), \\
	v_p^a(\theta,r)&:=\chi(r)\Big(\sum_{i=1}^{12}\varepsilon^i v_p^{(i)}(\theta,Y)\Big),                          \\
	g_p^a(\theta,r)&:=\chi(r)\Big(g_p^{(0)}(\theta,Y)+\sum_{i=1}^{11}\varepsilon^ig_p^{(i)}(\theta,Y)\Big), \\
	h_p^a(\theta,r)&:=\chi(r)\Big(\sum_{i=1}^{12}\varepsilon^i h_p^{(i)}(\theta,Y)\Big),                          \\
	p_p^a(\theta,r)&:=\chi^2(r)\Big(\sum_{i=1}^{12}\varepsilon^i p_p^{(i)}(\theta,Y)\Big),
\end{align*}
and
\begin{align*}
	u_e^a(\theta,r)&:=u_e(r)+\sum_{i=1}^{11}\varepsilon^i \bar{u}_e^{(i)}(\theta,r),&
	v_e^a(\theta,r)&:=\sum_{i=1}^{11}\varepsilon^i v_e^{(i)}(\theta,r), \\
	g_e^a(\theta,r)&:=g_e(r)+\sum_{i=1}^{11}\varepsilon^i g_e^{(i)}(\theta,r),&
	h_e^a(\theta,r)&:=\sum_{i=1}^{11}\varepsilon^i h_e^{(i)}(\theta,r), \\
	p_e^a(\theta,r)&:=p_e(r)+\sum_{i=1}^{11}\varepsilon^i p_e^{(i)}(\theta,r).
\end{align*}
We construct an approximate solution
\begin{align}\label{approximate-solution}
	u^a(\theta,r)&:=u_e^a(\theta,r)+u_p^a(\theta,r)+\varepsilon^{12} u_L(\theta,r),\nonumber \\[3pt]
	v^a(\theta,r)&:=v_e^a(\theta,r)+v_p^a(\theta,r),\nonumber                                   \\[3pt]
	g^a(\theta,r)&:=g_e^a(\theta,r)+g_p^a(\theta,r),\nonumber       \\[3pt]
	p^a(\theta,r)&:=p_e^a(\theta,r)+p_p^a(\theta,r),
\end{align}
where the velocity corrector $u_L(\theta,r)$ is constructed in
Appendix~\ref{app:correctors} and satisfies
\begin{align*}
	u_L(\theta,1)=0,\qquad
	\|\partial_\theta^j\partial_r^ku_L\|_2
	\leq C(j,k)\eta\varepsilon^{-k}.
\end{align*}
It is chosen so that $(u^a,v^a)$ is divergence-free.
Using the zero-mean periodic primitive $\mathcal I$ defined above, define the
approximate stream functions and magnetic potentials by
\begin{align*}
	\Phi_e^a(\theta,r)&:=\int_r^1 u_e^a(\theta,\rho)\,\ud\rho
	+\mathcal I\big[v_e^a(\cdot,1)\big](\theta),\\
	\Phi_p^a(\theta,r)&:=\int_r^1 u_p^a(\theta,\rho)\,\ud\rho
	+\mathcal I\big[v_p^a(\cdot,1)\big](\theta),\\
	\Phi_L(\theta,r)&:=\int_r^1 u_L(\theta,\rho)\,\ud\rho,\\
	\Pi_e^a(\theta,r)&:=-\int_0^r g_e^a(\theta,\rho)\,\ud\rho,\\
	\Pi_p^a(\theta,r)&:=-\int_0^r g_p^a(\theta,\rho)\,\ud\rho.
\end{align*}
Then
\[
	\Phi^a=\Phi_e^a+\Phi_p^a+\varepsilon^{12}\Phi_L,\qquad
	\Pi^a=\Pi_e^a+\Pi_p^a,
\]
and
\[
	u^a=-\Phi^a_r,\quad v^a=\frac{\Phi^a_\theta}{r},\qquad
	g^a=-\Pi^a_r,\quad h^a=\frac{\Pi^a_\theta}{r}.
\]
Thus the approximate magnetic field is divergence-free by construction and
its prescribed tangential trace is unchanged.  To compare its normal
component with the profiles used in the expansion, set
\[
\bar h^a=h_e^a+h_p^a,
\qquad K_g=\partial_\theta g^a+\partial_r(r\bar h^a).
\]
The profile divergence equations give
\begin{equation}\label{magnetic-normal-correction}
h^a=\bar h^a+h_c^a,
\qquad
h_c^a(\theta,r)=-\frac1r\int_0^rK_g(\theta,\rho)\,\ud\rho.
\end{equation}
Write $h_p^{(12)}=H_{12,\infty}+\widetilde h_p^{(12)}$.  The explicit defect
formula in Appendix~\ref{app:correctors} gives
\[
K_g=\varepsilon^{12}\partial_r(r\chi H_{12,\infty})
+\widetilde K_g.
\]
The term $\widetilde K_g$ contains only decaying Prandtl profiles near the
boundary, and hence its radial primitive gains one power of $\varepsilon$.
Thus,
\begin{align}\label{magnetic-normal-correction-estimate}
h_c^a&=-\varepsilon^{12}\chi H_{12,\infty}+\widetilde h_c^a,
&|\partial_\theta^j\partial_r^k\widetilde h_c^a|
&\leq C(j,k)\eta\varepsilon^{13-k},\nonumber\\
|\partial_\theta^jh_c^a|+|\partial_\theta^j\partial_rh_c^a|
+\varepsilon|\partial_\theta^j\partial_r^2h_c^a|
&\leq C(j)\eta\varepsilon^{12}.&&
\end{align}
No boundary value of $h^a$ is prescribed, so this normal correction requires
no further boundary corrector.  We have, in particular,
\[
u^a_\theta+rv^a_r+v^a=0,
\qquad g^a_\theta+rh^a_r+h^a=0.
\]
Moreover, $(u^a, v^a,g^a, h^a)$ satisfies the following boundary conditions:
\begin{align*}
	u^a(\theta+2\pi,r) & =u^a(\theta,r), \ v^a(\theta+2\pi,r)=v^a(\theta,r),                                    \\[5pt]
	g^a(\theta+2\pi,r) & =g^a(\theta,r), \ h^a(\theta+2\pi,r)=h^a(\theta,r),                                    \\[5pt]
	u^a(\theta, 1)     & =\alpha+\eta f(\theta), \ v^a(\theta,1)=0,\ g^a(\theta, 1)=\beta+\eta \widehat f(\theta).
\end{align*}
In addition, $v^a$ and $h^a$ satisfy
\begin{align*}
	\int_{0}^{2\pi}v^a\,\ud\theta=\int_{0}^{2\pi}h^a\,\ud\theta=0.
\end{align*}

Since the constant \(A_{i\infty}\) is absorbed into
\(\bar u_e^{(i)}\) through the correction \(A_{i\infty}r\), estimates
\eqref{prandtl-estimate},
\eqref{estimate-of-first-combined-linearized-euler-equation1},
\eqref{decay-behavior-prandtl-1},
\eqref{estimate-of-fourth-linearized-euler-equation}, and
\eqref{decay-behavior-prandtl-3} imply
\begin{align}\label{estimate-on-euler-parts}
	\begin{aligned}
		 & \|u^a_e-ar\|_{L^\infty(\Omega)}
		 +\|(\partial_r(u^a_e-ar),\partial_\theta u^a_e)\|_{L^\infty(\Omega)}
		 \leq C\varepsilon\eta,\\
		 & \|v^a_e\|_{L^\infty(\Omega)}
		 +\|(\partial_rv^a_e,\partial_\theta v^a_e)\|_{L^\infty(\Omega)}
		 \leq C\varepsilon\eta,\\[5pt]
		 & {\|g^a_e-br\|_{L^\infty(\Omega)}
		 +\|(\partial_r(g^a_e-br),\partial_\theta g^a_e)\|_{L^\infty(\Omega)}
		 \leq C\varepsilon\eta,}\\
		 & {\|h^a_e\|_{L^\infty(\Omega)}
		 +\|(\partial_rh^a_e,\partial_\theta h^a_e)\|_{L^\infty(\Omega)}
		 \leq C\varepsilon\eta,}\\[5pt]
		 & |\partial_\theta u_e^a(\theta,r)+v_e^a(\theta,r)|\leq C\varepsilon \eta r, \ \ |\partial_\theta v_e^a(\theta,r)-u_e^a(\theta,r)+ar|\leq C\varepsilon \eta r, \ \forall (\theta, r)\in \Omega
	\end{aligned}
\end{align}
and
\begin{align*}
	\begin{aligned}
		&\|(Y^j\partial_Y^ku_p^a,Y^j\partial_\theta^ku_p^a)\|_{L^\infty(\Omega)}
		+\|(Y^j\partial_Y^kg_p^a,Y^j\partial_\theta^kg_p^a)\|_{L^\infty(\Omega)}
		\leq C\eta,\\
		&\|(Y^j\partial_Y^kv_p^a,Y^j\partial_\theta^kv_p^a)\|_{L^\infty(\Omega)}
		+\|(Y^j\partial_Y^kh_p^a,Y^j\partial_\theta^kh_p^a)\|_{L^\infty(\Omega)}
		\leq C\varepsilon\eta,
		\qquad \forall j\leq 2,\ k\leq 1.
	\end{aligned}
\end{align*}
The nondecaying top limits satisfy the same bounds since
$\varepsilon^{11}|Y|^j\leq1$ and
$\varepsilon^{12}|Y|^j\leq\varepsilon$ for $j\leq2$; the possible affine
part of $\varepsilon^{12}p_p^{(12)}$ gives only
$O(\varepsilon^{11})$ cutoff terms.

Finally, set
\begin{align*}
	R_u^a= & u^a u^a_\theta+rv^a u^a_r+u^a v^a-g^a g^a_\theta-rh^a g^a_r-g^a h^a+p^a_\theta-\kappa \varepsilon^2\big(ru^a_{rr}+ u^a_{r}+\frac{u^a_{\theta\theta}}{r}+\frac{2}{r}v^a_\theta-\frac{ u^a}{r}\big), \\[5pt]
	R_v^a= & u^a v^a_\theta+rv^av^a_r-(u^a)^2-g^a h^a_\theta-rh^a h^a_r+(g^a)^2 +rp^a_r-\kappa\varepsilon^2\big(rv^a_{rr}+v^a_{r}+\frac{ v^a_{\theta\theta}}{r}-\frac{2}{r} u^a_\theta-\frac{ v^a}{r}\big),     \\[5pt]
	R_g^a= & -u^a rh^a+rv^a g^a-\varepsilon^2 ( rg^a_{r}+g^a-h^a_\theta)+2\beta\varepsilon^2r
\end{align*}
and use these definitions throughout the residual estimate.
\begin{proposition}[Residual bounds]\label{prop:residual-estimate}
The approximate solution satisfies
\begin{align}\label{remainder-estimate}
\left(\int_0^1\int_0^{2\pi}
\frac{(r\partial_rR_u^a)^2+(R_v^a)^2+(R_g^a)^2}{r}
\,\ud\theta\,\ud r\right)^{1/2}
\leq C\varepsilon^{10},
\end{align}
and
\begin{align}\label{remainder-estimate-in-polar-variables}
\int_0^1\int_0^{2\pi}
\frac{{|\partial_rR_u^a-r^{-1}\partial_\theta R_v^a|^2}
+|R_g^a|^2+|\partial_\theta R_g^a|^2}{r}
\,\ud\theta\,\ud r
\leq C\varepsilon^{20}.
\end{align}
Moreover,
\begin{align}\label{eq:pointwise-residual}
{\left|\partial_rR_u^a(\theta,r)
-\frac1r\partial_\theta R_v^a(\theta,r)\right|}
+|R_v^a(\theta,r)|+|R_g^a(\theta,r)|
\leq C\varepsilon^{10}r.
\end{align}
Let $\mathbf R_M^a$ be the Cartesian momentum residual, so
that $R_u^a$ and $R_v^a$ are $r$ times its tangential and radial
components, and let $R_\Pi^a=R_g^a/r$ be the scalar magnetic-potential
residual.  Then
\begin{equation}\label{cartesian-residual-H1}
\|\mathbf R_M^a\|_{H^1(B_1)}+
\|R_\Pi^a\|_{H^1(B_1)}\leq C\varepsilon^{10}.
\end{equation}
{
For every fixed \(0<\rho<1\), one also has
\begin{equation}\label{cartesian-residual-interior-H2}
\|\mathbf R_M^a\|_{H^2(B_\rho)}+
\|R_\Pi^a\|_{H^2(B_\rho)}\leq C_\rho\varepsilon^{10}.
\end{equation}
}
\end{proposition}
\begin{proof}
{\color{black}
On the region where \(\chi=1\), substitution of the expansion shows that
the coefficients through order \(\varepsilon^{11}\) vanish.  The leading
order follows from \eqref{prandtl-problem-near-1} and \eqref{euler-1}, the
first correction from \eqref{first-linearized-prandtl-problem-near-1} and
\eqref{equation-for-second-pressure}, and the remaining orders from
\eqref{outer-4-order-equation} and
\eqref{third-linearized-prandtl-problem-near-1}.  Hence the uncancelled bulk
terms start at order \(\varepsilon^{12}\).  One physical radial derivative
costs at most \(\varepsilon^{-1}\).

On the support of \(\chi'\), every term containing a decaying Prandtl profile
is \(O(\varepsilon^N)\) for any fixed \(N\).  The affine part of
\(\varepsilon^{12}p_p^{(12)}\) contributes \(O(\varepsilon^{11})\).  The
largest differentiated corrector terms are of order \(\varepsilon^{10}\).
Indeed, \eqref{magnetic-normal-correction-estimate} and the profile bounds give
\[
 \partial_r\!\left(h_c^a\partial_rg_p^{(0)}\right)
 =(\partial_rh_c^a)(\partial_rg_p^{(0)})
 +h_c^a\partial_{rr}g_p^{(0)}=O(\varepsilon^{10}),
\]
and the terms containing \(\varepsilon^{12}u_L\) satisfy the same bound by
Appendix~\ref{app:correctors}.  Therefore
\[
 \|\mathbf R_M^a\|_{H^1(\{r\geq1/2\})}
 +\|R_\Pi^a\|_{H^1(\{r\geq1/2\})}
 \leq C\varepsilon^{10}.
\]}

It remains to justify the factor at the polar center.  Write the smooth
Cartesian momentum residual as
\(\mathbf R_M^a=F_\theta^a\mathbf e_\theta+F_r^a\mathbf e_r\).
By definition,
\[
 R_u^a=rF_\theta^a,\qquad R_v^a=rF_r^a,\qquad
 R_g^a=rR_\Pi^a,
\]
and the polar curl identity gives
\begin{equation*}
\partial_rR_u^a-\frac1r\partial_\theta R_v^a
=r\,\operatorname{curl}\mathbf R_M^a.
\end{equation*}
On \(r\leq1/2\), the Prandtl and corrector terms vanish.  The regular Fourier
conditions for the smooth Euler profiles make
\(\mathbf R_M^a,R_\Pi^a\) smooth Cartesian functions at \(r=0\), and
{\color{black}the same coefficient cancellations give a
\(C\varepsilon^{10}\) bound for their \(C^1\)-norms.}  The three identities above prove
\eqref{eq:pointwise-residual}; integration proves
\eqref{remainder-estimate}--\eqref{remainder-estimate-in-polar-variables}
and \eqref{cartesian-residual-H1}.  {On every fixed disk \(B_\rho\) with
\(\rho<1\), the Prandtl profiles and their physical derivatives are rapidly
decaying.  Differentiating the same residual identities once more therefore
gives \eqref{cartesian-residual-interior-H2}.}
\end{proof}

\section{Linear stability estimates for error equations}

In this section, we derive the error equations and establish linear stability estimates.

\subsection{Error equations}
We write the error equations in terms of the velocity stream function and the magnetic potential.
Using the operator $\mathcal I$ defined above, set
\begin{align*}
	\Phi(\theta,r) & =\int_r^1 (u^\varepsilon-u^a)(\theta,\rho)\,\ud\rho,                                                                                                    \\
	\Pi(\theta,r)  & =\int_r^1 (g^\varepsilon-g^a)(\theta,\rho)\,\ud\rho+\mathcal{I}\big[(h^\varepsilon-h^a)(\cdot,1)\big](\theta).
\end{align*}
Then
\[
	u=-\Phi_r,\quad v=\frac{\Phi_\theta}{r},\qquad
	g=-\Pi_r,\quad h=\frac{\Pi_\theta}{r}.
\]
The approximate stream function and magnetic potential are defined in the same way, so that
\[
	u^a=-\Phi^a_r,\quad v^a=\frac{\Phi^a_\theta}{r},\qquad
	g^a=-\Pi^a_r,\quad h^a=\frac{\Pi^a_\theta}{r}.
\]
After eliminating the pressure, the error equations for $\Phi$ and $\Pi$ read
\begin{align}\label{error-equation}
	\left\{
	\begin{array}{lll}
		\Phi^a_\theta(\Delta\Phi)_r+\Phi_\theta(\Delta\Phi^a)_r
		-\Phi^a_r(\Delta\Phi)_\theta-\Phi_r(\Delta\Phi^a)_\theta                                             \\[3pt]
		\qquad -\Pi^a_\theta(\Delta\Pi)_r-\Pi_\theta(\Delta\Pi^a)_r
		+\Pi^a_r(\Delta\Pi)_\theta+\Pi_r(\Delta\Pi^a)_\theta
		-\kappa\varepsilon^2 r\Delta^2\Phi=F_1,                                                             \\[7pt]
		-\Phi^a_r\Pi_\theta-\Phi_r\Pi^a_\theta
		+\Phi^a_\theta\Pi_r+\Phi_\theta\Pi^a_r-\varepsilon^2 r\Delta\Pi=F_2,                                 \\[7pt]
		\Phi (\theta+2\pi,r)=\Phi(\theta,r), \ \Pi(\theta+2\pi,r)=\Pi(\theta,r),                             \\
		\Phi(\theta,1)=\Phi_r(\theta,1)=0,                                                                   \\
		\Pi_r(\theta,1)=0.
	\end{array}
	\right.
\end{align}
where
\begin{align*}
	 F_1=&\ {\partial_rR_u^a-\frac1r\partial_\theta R_v^a}
	-\big[\Phi_\theta(\Delta\Phi)_r-\Phi_r(\Delta\Phi)_\theta
	-\Pi_\theta(\Delta\Pi)_r+\Pi_r(\Delta\Pi)_\theta\big],              \\
	 F_2=&\ R^a_g+\Phi_r\Pi_\theta-\Phi_\theta\Pi_r.
\end{align*}

\subsection{Linear stability estimate}

In this subsection, we consider the linear equations
\eqref{error-equation} in the polar variables $(r,\theta)$.  All scalar
potentials and test functions are smooth Cartesian functions on the closed
disk, so the integrations by parts below have no boundary contribution at
$r=0$.  We also recall that the reduction to $\alpha>0$ and the relation
$a=\alpha+O(\eta^2)$ imply
\[
a\geq \frac{\alpha}{2}>0
\]
when $\eta_0$ is sufficiently small.  The coercive estimates use only
$|a^2-b^2|>0$.

For later reference we record the Hardy and polar Hessian identities used repeatedly.

\begin{lemma}[Boundary Hardy inequalities]
Let $q=q(\theta,r)$ be smooth and periodic in $\theta$. If $q(\theta,1)=0$, then
\[
\int_0^1\int_0^{2\pi}r\left|\frac{q}{1-r}\right|^2\,\ud\theta\,\ud r
\leq C\int_0^1\int_0^{2\pi}r|q_r|^2\,\ud\theta\,\ud r.
\]
If in addition $q_r(\theta,1)=0$, then
\[
\int_0^1\int_0^{2\pi}r\left|\frac{q}{(1-r)^2}\right|^2\,\ud\theta\,\ud r
\leq C\int_0^1\int_0^{2\pi}r|q_{rr}|^2\,\ud\theta\,\ud r.
\]
The same estimates hold after applying tangential derivatives.
\end{lemma}
\begin{proof}
The first estimate is the one-dimensional Hardy inequality on $[1/2,1]$, followed by an ordinary Poincar\'e estimate on $[0,1/2]$. Applying the first estimate to $q_r$ and integrating $q(r)=-\int_r^1q_s(s)\,\ud s$ once more gives the second estimate.
\end{proof}

\begin{lemma}[Polar Hessian identities]\label{lem:polar-hessian}
For a smooth scalar function $q$ on the disk, set
\[
|\nabla_p^2q|^2=q_{rr}^2+2\left(\frac{q_{r\theta}}r-\frac{q_\theta}{r^2}\right)^2
+\left(\frac{q_r}{r}+\frac{q_{\theta\theta}}{r^2}\right)^2.
\]
If $q=q_r=0$ at $r=1$, then
\[
\int_\Omega r(\Delta q)^2=\int_\Omega r|\nabla_p^2q|^2.
\]
If only $q_r=0$ at $r=1$, then
\[
\int_\Omega r(\Delta q)^2
=\int_\Omega r|\nabla_p^2q|^2
+\int_0^{2\pi}|q_\theta(\theta,1)|^2\,\ud\theta.
\]
Thus the first identity applies to $q=\Phi$ and $q=\Phi_\theta$, while the second applies to $q=\Pi$ and $q=\Pi_\theta$, with boundary terms $\Pi_\theta^2$ and $\Pi_{\theta\theta}^2$, respectively.
\end{lemma}
\begin{proof}
Expand $(\Delta q)^2-|\nabla_p^2q|^2$ and integrate once in $r$ and once in $\theta$. Periodicity removes the tangential boundary terms, smoothness removes the contribution at $r=0$, and the stated terms remain at $r=1$.
\end{proof}

{\color{black}
\begin{lemma}
Let \(d(r)=1-r\).  For every fixed \(j,k,q\geq0\) needed below, the
approximate fields satisfy
\begin{align*}
\|\partial_\theta^k\partial_r^j(u^a-ar,g^a-br)\|_\infty
&\leq C_{j,k}\eta\varepsilon^{-j},\\
\|\partial_\theta^k\partial_r^j(v^a,h^a)\|_\infty
&\leq C_{j,k}\eta\varepsilon^{1-j},\\
\|d^q\partial_\theta^k\partial_r^j(u_p^a,g_p^a)\|_\infty
&\leq C_{j,k,q}\eta\varepsilon^{q-j},\\
\|d^q\partial_\theta^k\partial_r^j(v_p^a,h_p^a)\|_\infty
&\leq C_{j,k,q}\eta\varepsilon^{q+1-j}.
\end{align*}
The Euler parts additionally obey
\begin{align*}
|\partial_\theta u_e^a+v_e^a|
+|\partial_\theta v_e^a-u_e^a+ar|
+|\partial_\theta g_e^a+h_e^a|
+|\partial_\theta h_e^a-g_e^a+br|
\leq C\varepsilon\eta r.
\end{align*}
All estimates remain valid for the velocity and magnetic correctors at the
orders at which they occur.
\end{lemma}
}
\begin{proof}
For a Prandtl profile,
\(\partial_r=\varepsilon^{-1}\partial_Y\) and
\(d=\varepsilon|Y|\), so {\color{black}the weighted bounds follow directly from}
the arbitrary polynomial weights in
\eqref{prandtl-estimate}, \eqref{decay-behavior-prandtl-1}, and
\eqref{decay-behavior-prandtl-3}.  The Euler bounds follow from
\eqref{estimate-on-euler-parts} and the regular polar Fourier
representations.  Appendix~\ref{app:correctors} and
\eqref{magnetic-normal-correction-estimate} give the corrector bounds.
Adding the pieces proves the lemma.
\end{proof}

For a function $f$, let $f^0$ be its zero Fourier mode in $\theta$ and
\(\widetilde f=f-f^0\).  All integrals below are taken over \(\Omega\) with
respect to \(\ud\theta\,\ud r\) unless a weight is displayed.  Define
\[
	G=a\Pi_\theta-b\Phi_\theta-\frac{\Pi^a_\theta}{r}\Phi_r^0+\frac{\Phi^a_\theta}{r}\Pi_r^0 .
\]
We introduce three energies with different $\varepsilon$ scales:
\begin{align*}
	\left\{
	\begin{array}{lll}
		\mathbb{E}= & \displaystyle
		\varepsilon^2\int r\Big(
		|\nabla^2_{\!p}\Phi|^2+|\nabla^2_{\!p}\Pi|^2
		+|\nabla^2_{\!p}\Phi_\theta|^2+|\nabla^2_{\!p}\Pi_\theta|^2\Big),\\[7pt]
		\mathbb{P}= & \displaystyle
		\int\left(r\Phi_{r\theta}^2+\frac{\Phi_{\theta\theta}^2}{r}
		+r\Pi_{r\theta}^2+\frac{\Pi_{\theta\theta}^2}{r}\right),\\[7pt]
		\mathbb{L}= & \displaystyle
		\frac{1}{\varepsilon^2}\int rG^2.
	\end{array}
	\right.
\end{align*}
We also write
\[
	\mathbb P_\Phi=\int\left(r\Phi_{r\theta}^2+\frac{\Phi_{\theta\theta}^2}{r}\right),
	\qquad
	\mathbb P_\Pi=\int\left(r\Pi_{r\theta}^2+\frac{\Pi_{\theta\theta}^2}{r}\right),
\]
so that $\mathbb P=\mathbb P_\Phi+\mathbb P_\Pi$.

\subsubsection{A combination estimate}

This subsection establishes the combination estimate which replaces the missing boundary control for $\Pi$. The quantity $G$ is designed so that the leading part of the second equation in \eqref{error-equation} is $rG$, while the remaining terms either contain small approximate profiles or have good boundary behavior.
\begin{lemma}
	Let $(\Phi,\Pi)$ be a smooth solution of \eqref{error-equation}. Then
	\begin{align}\label{e:combination-estimate}
		\mathbb{L}+b\int r\Delta\Pi\,\Phi_\theta
		\leq C\eta(\mathbb{E}+\mathbb{P}+\mathbb{L})
		+\frac{C}{\varepsilon^2}\int \frac{F_2^2}{r}.
	\end{align}
\end{lemma}
\begin{proof}
	We first rewrite the second equation in \eqref{error-equation}. Since
	$u^a=-\Phi^a_r$ and $g^a=-\Pi^a_r$, it becomes
	\begin{align*}
		&u^a\Pi_\theta-g^a\Phi_\theta-\Pi^a_\theta\Phi_r+\Phi^a_\theta\Pi_r-\varepsilon^2 r\Delta\Pi=F_2 .
	\end{align*}
	Using the definition of $G$ and separating the zero Fourier modes of $\Phi_r$ and $\Pi_r$, this identity is equivalent to
	\begin{align*}
		rG+(u^a-ar)\Pi_\theta-(g^a-br)\Phi_\theta
		-\Pi^a_\theta\widetilde{\Phi}_r
		+\Phi^a_\theta\widetilde{\Pi}_r
		-\varepsilon^2 r\Delta\Pi=F_2.
	\end{align*}
	Multiplying this equation by $\varepsilon^{-2}G$ and integrating over $\Omega$, we obtain
	\begin{align*}
		\mathbb{L}-\int r\Delta\Pi\,G
		=I_1+I_2+I_3+I_4,
	\end{align*}
	where
	\begin{align*}
		I_1=&-\frac1{\varepsilon^2}\int (u^a-ar)\Pi_\theta G,\\
		I_2=&\ \frac1{\varepsilon^2}\int (g^a-br)\Phi_\theta G,\\
		I_3=&\ \frac1{\varepsilon^2}\int \Pi^a_\theta\widetilde{\Phi}_rG
		-\frac1{\varepsilon^2}\int \Phi^a_\theta\widetilde{\Pi}_rG,\\
		I_4=&\ \frac1{\varepsilon^2}\int F_2G.
	\end{align*}

	We first estimate $I_1$. {Split
	$u^a-ar=(u_e^a-ar+\varepsilon^{12}u_L)+u_p^a$.} For the Euler part,
	\begin{align*}
		{\left|\frac1{\varepsilon^2}\int
		(u_e^a-ar+\varepsilon^{12}u_L)\Pi_\theta G\right|}
		&\leq {\frac1{\varepsilon^2}\int
		|u_e^a-ar+\varepsilon^{12}u_L|\,|\Pi_\theta||G|}\\
		&\leq \frac{C\eta}{\varepsilon}\int |\Pi_\theta||G|\\
		&\leq \eta\mathbb{L}+C\eta\int \frac{\Pi_\theta^2}{r}
		\leq \eta\mathbb{L}+C\eta\mathbb{P}.
	\end{align*}
	For the Prandtl part, use the definition of $G$ to eliminate $\Pi_\theta$:
	\[
		\Pi_\theta=\frac1a\left(G+b\Phi_\theta+\frac{\Pi^a_\theta}{r}\Phi_r^0-\frac{\Phi^a_\theta}{r}\Pi_r^0\right).
	\]
	Then
	\begin{align*}
		&-\frac1{\varepsilon^2}\int u_p^a\Pi_\theta G\\
		=&-\frac1{a\varepsilon^2}\int u_p^aG^2
		-\frac{b}{a\varepsilon^2}\int u_p^a\Phi_\theta G
		-\frac1{a\varepsilon^2}\int u_p^a\left(\frac{\Pi^a_\theta}{r}\Phi_r^0-\frac{\Phi^a_\theta}{r}\Pi_r^0\right)G .
	\end{align*}
	The first term is bounded by $C\eta\mathbb{L}$. For the second term, since
	$\Phi_\theta(\theta,1)=0$, weighted Hardy's inequality gives
	\begin{align*}
		\left|\frac1{\varepsilon^2}\int u_p^a\Phi_\theta G\right|
		&\leq {\frac{1}{\varepsilon^2}\int} (1-r)|u_p^a|
		\left|\frac{\Phi_\theta}{1-r}\right||G| \\
           &\leq {\frac{C\eta}{\varepsilon}\int}
		r\left|\frac{\Phi_\theta}{1-r}\right||G|\\
		&\leq \eta\mathbb{L}+C\eta\int r\Phi_{r\theta}^2
		\leq \eta\mathbb{L}+C\eta\mathbb{P}.
	\end{align*}
	The last term contains the normal approximate fields $\Pi^a_\theta/r$ and $\Phi^a_\theta/r$, hence
	\begin{align*}
		&\left|\frac1{\varepsilon^2}\int u_p^a\left(\frac{\Pi^a_\theta}{r}\Phi_r^0-\frac{\Phi^a_\theta}{r}\Pi_r^0\right)G\right|\\
		\leq& {\frac{1}{\varepsilon^2}\int}(1-r)|u_p^a|
		\left|\frac{\Pi^a_\theta}{r}\frac{\Phi_r^0}{1-r}
		-\frac{\Phi^a_\theta}{r}\frac{\Pi_r^0}{1-r}\right||G|\\
		\leq& {\frac{C\eta}{\varepsilon}\int} r
		\left|\frac{\Pi^a_\theta}{r}\frac{\Phi_r^0}{1-r}
		-\frac{\Phi^a_\theta}{r}\frac{\Pi_r^0}{1-r}\right||G|\\
		\leq& C\eta\int r(|\Phi_{rr}^0|+|\Pi_{rr}^0|)|G|\\
		\leq& \eta\mathbb{L}+C\eta\mathbb{E}.
	\end{align*}
	Thus,
	\begin{align}\label{e:i1-polar-estimate}
		|I_1|\leq C\eta(\mathbb{E}+\mathbb{P}+\mathbb{L}).
	\end{align}

	The term $I_2$ is treated similarly, but no substitution by $G$ is needed. Splitting
	$g^a-br=(g_e^a-br)+g_p^a$, we get
	\begin{align*}
		\left|\frac1{\varepsilon^2}\int (g_e^a-br)\Phi_\theta G\right|
		&\leq \frac1{\varepsilon^2}\int |g_e^a-br|\,|\Phi_\theta||G|\\
		&\leq \frac{C\eta}{\varepsilon}\int |\Phi_\theta||G|
		\leq \eta\mathbb{L}+C\eta\mathbb{P},
	\end{align*}
	and, using again $\Phi_\theta(\theta,1)=0$,
	\begin{align*}
		\left|\frac1{\varepsilon^2}\int g_p^a\Phi_\theta G\right|
		&\leq {\frac{1}{\varepsilon^2}\int} (1-r)|g_p^a|
		\left|\frac{\Phi_\theta}{1-r}\right||G|\\
		&\leq {\frac{C\eta}{\varepsilon}\int}
		r\left|\frac{\Phi_\theta}{1-r}\right||G|\\
		&\leq \eta\mathbb{L}+C\eta\int r\Phi_{r\theta}^2
		\leq \eta\mathbb{L}+C\eta\mathbb{P}.
	\end{align*}
	Thus
	\begin{align}\label{e:i2-polar-estimate}
		|I_2|\leq C\eta(\mathbb{P}+\mathbb{L}).
	\end{align}

	For $I_3$, we use the smallness of the normal components of the approximate field:
	\begin{align*}
		|I_3|
		&\leq \frac{C\eta}{\varepsilon}\int r\big(|\widetilde{\Phi}_r|+|\widetilde{\Pi}_r|\big)|G|\\
		&\leq \eta\mathbb{L}
		+C\eta\int r\big(\widetilde{\Phi}_r^{\,2}+\widetilde{\Pi}_r^{\,2}\big)
		\leq C\eta(\mathbb{P}+\mathbb{L}).
	\end{align*}
	Combining \eqref{e:i1-polar-estimate}, \eqref{e:i2-polar-estimate}, and the last bound gives
	\begin{align*}
		|I_1|+|I_2|+|I_3|\leq C\eta(\mathbb{E}+\mathbb{P}+\mathbb{L}).
	\end{align*}
	For the source term,
	\begin{align*}
		|I_4|\leq \delta\mathbb{L}
		+\frac{C_\delta}{\varepsilon^2}\int\frac{F_2^2}{r}.
	\end{align*}

	It remains to identify the diffusion contribution. By the definition of $G$,
	\begin{align*}
		-\int r\Delta\Pi\,G
		=&-a\int r\Delta\Pi\,\Pi_\theta
		+b\int r\Delta\Pi\,\Phi_\theta\\
		&+\int\Delta\Pi\,\Pi^a_\theta\Phi_r^0
		-\int\Delta\Pi\,\Phi^a_\theta\Pi_r^0.
	\end{align*}
	The first term on the right is zero by periodicity in $\theta$ and by the self-adjointness of $\Delta$ with respect to the polar measure $r\,\ud r\,\ud\theta$:
	\[
		\int r\Delta\Pi\,\Pi_\theta=0.
	\]
	For the last two terms, using
	\[
		\Pi_r(\theta,1)=\Phi_r(\theta,1)=0,\qquad
		\partial_\theta\Phi_r^0=\partial_\theta\Pi_r^0=0,
	\]
	we obtain
	\begin{align}\label{type-1-estimate1}
		&\int\Delta\Pi\,\Pi^a_\theta\Phi_r^0
		-\int\Delta\Pi\,\Phi^a_\theta\Pi_r^0\nonumber\\
		=&-\int r\Pi_r\left[\left(\frac{\Pi^a_\theta}{r}\right)_r\Phi_r^0
		+\frac{\Pi^a_\theta}{r}\Phi_{rr}^0
		-\left(\frac{\Phi^a_\theta}{r}\right)_r\Pi_r^0
		-\frac{\Phi^a_\theta}{r}\Pi_{rr}^0\right]\\
		&-\int \frac{\Pi_\theta}{r}\left[
		\left(\frac{\Pi^a_\theta}{r}\right)_\theta\Phi_r^0
		-\left(\frac{\Phi^a_\theta}{r}\right)_\theta\Pi_r^0\right].\nonumber
	\end{align}
	Notice that the zero mode in the first integral drops out:
	\[
		\int r\Pi_r^0\left[\left(\frac{\Pi^a_\theta}{r}\right)_r\Phi_r^0
		+\frac{\Pi^a_\theta}{r}\Phi_{rr}^0
		-\left(\frac{\Phi^a_\theta}{r}\right)_r\Pi_r^0
		-\frac{\Phi^a_\theta}{r}\Pi_{rr}^0\right]=0,
	\]
	because the bracket has zero Fourier mode equal to zero. Thus
	\begin{align}\label{type-1-estimate2}
		&\left|\int r\Pi_r\left[\left(\frac{\Pi^a_\theta}{r}\right)_r\Phi_r^0
		+\frac{\Pi^a_\theta}{r}\Phi_{rr}^0
		-\left(\frac{\Phi^a_\theta}{r}\right)_r\Pi_r^0
		-\frac{\Phi^a_\theta}{r}\Pi_{rr}^0\right]\right|\nonumber\\
		&=\left|\int r\widetilde{\Pi}_r\left[\left(\frac{\Pi^a_\theta}{r}\right)_r\Phi_r^0
		+\frac{\Pi^a_\theta}{r}\Phi_{rr}^0
		-\left(\frac{\Phi^a_\theta}{r}\right)_r\Pi_r^0
		-\frac{\Phi^a_\theta}{r}\Pi_{rr}^0\right]\right|\\
		&\leq C\varepsilon\eta\int r|\widetilde{\Pi}_r|
		\left(\left|\frac{\Phi_r^0}{1-r}\right|
		+\left|\frac{\Pi_r^0}{1-r}\right|
		+|\Phi_{rr}^0|+|\Pi_{rr}^0|\right)\nonumber\\
		&\leq C\varepsilon\eta\,\mathbb P^{1/2}
		\left(\int r\big((\Phi_{rr}^0)^2+(\Pi_{rr}^0)^2\big)\right)^{1/2}
		\leq C\eta(\mathbb E+\mathbb P),\nonumber
	\end{align}
	For the second integral in \eqref{type-1-estimate1}, we also have
	\begin{align*}
		&\left|\int \frac{\Pi_\theta}{r}\left[
		\left(\frac{\Pi^a_\theta}{r}\right)_\theta\Phi_r^0
		-\left(\frac{\Phi^a_\theta}{r}\right)_\theta\Pi_r^0\right]\right|\\
		&\leq C\varepsilon\eta\int \frac{|\Pi_\theta|}{r}(1-r)
		\left(\left|\frac{\Phi_r^0}{1-r}\right|
		+\left|\frac{\Pi_r^0}{1-r}\right|\right)\\
		&\leq C\varepsilon\eta
		\left(\int\frac{\Pi_\theta^2}{r}\right)^{1/2}
		\left(\int r\big((\Phi_{rr}^0)^2+(\Pi_{rr}^0)^2\big)\right)^{1/2}
		\leq C\eta(\mathbb E+\mathbb P).
	\end{align*}
	Thus,
	\begin{align*}
		\left|\int\Delta\Pi\,\Pi^a_\theta\Phi_r^0
		-\int\Delta\Pi\,\Phi^a_\theta\Pi_r^0\right|
		\leq C\eta(\mathbb E+\mathbb P).
	\end{align*}
	Combining the estimates above and taking $\delta>0$ fixed sufficiently small gives \eqref{e:combination-estimate}.
\end{proof}

\subsubsection{Energy estimate for velocity}

In this subsection, we establish the basic velocity energy estimate in polar variables.

\begin{lemma}\label{basic-energy-estimate}
	Let $(\Phi,\Pi)$ be a smooth solution of \eqref{error-equation}. Then
	\begin{align}\label{e:basic-energy-estimate}
		\varepsilon^2\int r|\nabla^2_{\!p}\Phi|^2
		+b\int r(\Delta\Pi)_\theta\Phi
		\leq C\eta(\mathbb{E}+\mathbb{P}+\mathbb{L})
		+\frac{C}{\varepsilon^2}\int\frac{F_1^2}{r}.
	\end{align}
\end{lemma}
\begin{proof}
{
	Multiplying the first equation in \eqref{error-equation} by $-\Phi$ and integrating over $\Omega$, we obtain
	\begin{align*}
		&\kappa\varepsilon^2\int r\Delta^2\Phi\,\Phi
		-\int \Phi^a_\theta(\Delta\Phi)_r\Phi
		-\int \Phi_\theta(\Delta\Phi^a)_r\Phi
		+\int \Phi^a_r(\Delta\Phi)_\theta\Phi
		+\int \Phi_r(\Delta\Phi^a)_\theta\Phi\\
		&\quad+\int \Pi^a_\theta(\Delta\Pi)_r\Phi
		+\int \Pi_\theta(\Delta\Pi^a)_r\Phi
		-\int \Pi^a_r(\Delta\Pi)_\theta\Phi
		-\int \Pi_r(\Delta\Pi^a)_\theta\Phi
		=-\int F_1\Phi .
	\end{align*}
	For the diffusion term, the boundary conditions $\Phi=\Phi_r=0$ at $r=1$ give
	\[
		\kappa\varepsilon^2\int r\Delta^2\Phi\,\Phi
		=\kappa\varepsilon^2\int r(\Delta\Phi)^2
		\geq c\varepsilon^2\int r|\nabla^2_{\!p}\Phi|^2.
	\]
	The principal Euler contribution is
	\[
		-\int ar(\Delta\Phi)_\theta\Phi+\int br(\Delta\Pi)_\theta\Phi
		=b\int r(\Delta\Pi)_\theta\Phi .
	\]
	Thus
	\[
		u_{p,c}^a:=u_p^a+{\color{black}\varepsilon^{12}}u_L,
		\qquad h_{p,c}^a:=h_p^a+h_c^a.
	\]
	Then
	\[
		\partial_\theta u_{p,c}^a+\partial_r(rv_p^a)=0,
		\qquad
		\partial_\theta g_p^a+\partial_r(rh_{p,c}^a)=0.
	\]
	\[
		\kappa\varepsilon^2\int r\Delta^2\Phi\,\Phi
		+b\int r(\Delta\Pi)_\theta\Phi
		\leq |I_1|+|I_2|+|I_3|+|I_4|+|I_5|,
	\]
	where
	\begin{align*}
		I_1=&-\int (u_e^a-ar)(\Delta\Phi)_\theta\Phi
		+\int (g_e^a-br)(\Delta\Pi)_\theta\Phi
		-\int rv_e^a(\Delta\Phi)_r\Phi
		+\int rh_e^a(\Delta\Pi)_r\Phi,\\
		I_2=&-\int u_{p,c}^a(\Delta\Phi)_\theta\Phi
		+\int g_p^a(\Delta\Pi)_\theta\Phi
		-\int rv_p^a(\Delta\Phi)_r\Phi
		+\int rh_{p,c}^a(\Delta\Pi)_r\Phi,\\
		I_3=&-\int \Phi_\theta(\Delta\Phi^a)_r\Phi
		+\int \Phi_r(\Delta\Phi^a)_\theta\Phi
		-\int \Pi_r(\Delta\Pi^a)_\theta\Phi,\\
		I_4=&\int \Pi_\theta(\Delta\Pi^a)_r\Phi,\qquad
		I_5=-\int F_1\Phi .
	\end{align*}
	By the Euler divergence identities and \eqref{estimate-on-euler-parts},
	\[
		\partial_\theta u_e^a+(rv_e^a)_r
		=\partial_\theta g_e^a+(rh_e^a)_r=0.
	\]
	\begin{align*}
		|I_1|
		&\leq \left|\int \Delta\Phi\big((u_e^a-ar)\Phi_\theta+rv_e^a\Phi_r\big)\right|
		+\left|\int \Delta\Pi\big((g_e^a-br)\Phi_\theta+rh_e^a\Phi_r\big)\right|\\
		&\leq I_{12}+I_{13},
	\end{align*}
	where
	\begin{align*}
		I_{12}=&\ \left|\int \Delta\Phi(u_e^a-ar)\Phi_\theta\right|
		+\left|\int \Delta\Pi(g_e^a-br)\Phi_\theta\right|,\\
		I_{13}=&\ \left|\int r v_e^a\Delta\Phi\,\Phi_r\right|
		+\left|\int r h_e^a\Delta\Pi\,\Phi_r\right|.
	\end{align*}
	For $I_{12}$,
	\begin{align*}
		I_{12}=&\left|\int \Delta\Phi(u_e^a-ar)\Phi_\theta\right|
		+\left|\int \Delta\Pi(g_e^a-br)\Phi_\theta\right|\\
		\leq& \int\big(|u_e^a-ar||\Delta \Phi|
		+|g_e^a-br||\Delta \Pi|\big)|\Phi_\theta|\\
		\leq& C\eta\varepsilon \int \big(|\Delta \Phi|+|\Delta \Pi|\big)|\Phi_\theta|\\
		\leq& C\eta\varepsilon\left(\int r\big(|\nabla^2_{\!p}\Phi|^2
		+|\nabla^2_{\!p}\Pi|^2\big)\right)^{1/2}
		\left(\int \frac{\Phi^2_{\theta\theta}}{r}\right)^{1/2}.
	\end{align*}
	For $I_{13}$, since $\int v_e^a\,\ud\theta=\int h_e^a\,\ud\theta=0$,
	\begin{align*}
		\int rv_e^a\Delta\Phi\,\Phi_r
		=&\int rv_e^a(\Delta\Phi^0)\widetilde{\Phi}_r
		-\int\widetilde{\Phi}_r\big((rv_e^a)_r\Phi_r+rv_e^a\Phi_{rr}\big)
		+\int v_e^a\widetilde{\Phi}_r\Phi_r\\
		&-\int\frac{\widetilde{\Phi}_\theta}{r^2}
		\big((rv_e^a)_\theta\Phi_r+rv_e^a\Phi_{r\theta}\big),\\
		\int rh_e^a\Delta\Pi\,\Phi_r
		=&\int rh_e^a(\Delta\Pi^0)\widetilde{\Phi}_r
		-\int\widetilde{\Pi}_r\big((rh_e^a)_r\Phi_r+rh_e^a\Phi_{rr}\big)
		+\int h_e^a\widetilde{\Pi}_r\Phi_r\\
		&-\int\frac{\widetilde{\Pi}_\theta}{r^2}
		\big((rh_e^a)_\theta\Phi_r+rh_e^a\Phi_{r\theta}\big).
	\end{align*}
	Hence
	\begin{align*}
		I_{13}
		&\leq C\varepsilon\eta
		\left(\int\left(r\Phi_{r\theta}^2+r\Pi_{r\theta}^2
		+\frac{\Phi_{\theta\theta}^2}{r}
		+\frac{\Pi_{\theta\theta}^2}{r}\right)\right)^{1/2}
		\left(\int r\big(|\nabla^2_{\!p}\Phi|^2
		+|\nabla^2_{\!p}\Pi|^2\big)\right)^{1/2}\\
		&\quad+C\varepsilon\eta
		\int\left(r\Phi_{r\theta}^2+r\Pi_{r\theta}^2
		+\frac{\Phi_{\theta\theta}^2}{r}+\frac{\Pi_{\theta\theta}^2}{r}\right).
	\end{align*}
	Thus,
	\begin{align*}
		|I_1|
		&\leq C\varepsilon\eta
		\left(\int\left(r\Phi_{r\theta}^2+r\Pi_{r\theta}^2
		+\frac{\Phi_{\theta\theta}^2}{r}
		+\frac{\Pi_{\theta\theta}^2}{r}\right)\right)^{1/2}
		\left(\int r\big(|\nabla^2_{\!p}\Phi|^2
		+|\nabla^2_{\!p}\Pi|^2\big)\right)^{1/2}\\
		&\quad+C\varepsilon\eta
		\int\left(r\Phi_{r\theta}^2+r\Pi_{r\theta}^2
		+\frac{\Phi_{\theta\theta}^2}{r}+\frac{\Pi_{\theta\theta}^2}{r}\right)
		\leq \delta\mathbb E+C_\delta\eta\mathbb P .
	\end{align*}
	For $I_2$, using
	\[
		(u_{p,c}^a)_\theta+(rv_p^a)_r
		=(g_p^a)_\theta+(rh_{p,c}^a)_r=0,
	\]
	The estimates for $u_L$ and \eqref{magnetic-normal-correction-estimate}
	show that $u_{p,c}^a$ and $h_{p,c}^a$ satisfy the same weighted boundary
	layer bounds used below for $u_p^a$ and $h_p^a$.  We have
	\begin{align*}
		I_2
		=&\int\big(u_{p,c}^a\Phi_\theta+rv_p^a\Phi_r\big)\Delta\Phi
		-\int\big(g_p^a\Phi_\theta+rh_{p,c}^a\Phi_r\big)\Delta\Pi .
	\end{align*}
	For the tangential Prandtl terms,
	\begin{align*}
		&\left|\int u_{p,c}^a\Phi_\theta\Delta\Phi\right|
		+\left|\int g_p^a\Phi_\theta\Delta\Pi\right|\\
		&\leq \int(1-r)|u_{p,c}^a|
		\left|\frac{\Phi_\theta}{1-r}\right||\Delta\Phi|
		+\int(1-r)|g_p^a|
		\left|\frac{\Phi_\theta}{1-r}\right||\Delta\Pi|\\
		&\leq C\eta \int r\Phi_{r\theta}^2
		+C\eta \varepsilon^2 \int r\big(|\nabla^2_{\!p}\Phi|^2
		+|\nabla^2_{\!p}\Pi|^2\big).
	\end{align*}
	For the normal Prandtl terms, since
	$\int v_p^a\,\ud\theta=\int h_{p,c}^a\,\ud\theta=0$,
	\begin{align*}
		\int rv_p^a\Phi_r\Delta\Phi
		=&\int rv_p^a(\Delta\Phi^0)\widetilde{\Phi}_r
		-\int\widetilde{\Phi}_r\big((rv_p^a)_r\Phi_r+rv_p^a\Phi_{rr}\big)
		+\int v_p^a\widetilde{\Phi}_r\Phi_r\\
		&-\int\frac{\widetilde{\Phi}_\theta}{r^2}
		\big((rv_p^a)_\theta\Phi_r+rv_p^a\Phi_{r\theta}\big),\\
		\int rh_{p,c}^a\Phi_r\Delta\Pi
		=&\int rh_{p,c}^a(\Delta\Pi^0)\widetilde{\Phi}_r
		-\int\widetilde{\Pi}_r\big((rh_{p,c}^a)_r\Phi_r+h_{p,c}^a r\Phi_{rr}\big)
		+\int h_{p,c}^a\widetilde{\Pi}_r\Phi_r\\
		&-\int\frac{\widetilde{\Pi}_\theta}{r^2}
		\big((rh_{p,c}^a)_\theta\Phi_r+rh_{p,c}^a\Phi_{r\theta}\big).
	\end{align*}
	Then
	\begin{align*}
		&\left|\int rv_p^a\Phi_r\Delta\Phi\right|
		+\left|\int rh_{p,c}^a\Phi_r\Delta\Pi\right|\\
		&\leq C\int r
		\left(|\widetilde{\Phi}_r|+|\widetilde{\Pi}_r|
		+\frac{|\widetilde{\Phi}_\theta|+|\widetilde{\Pi}_\theta|}{r}\right)\\
		&\qquad\times\left[
		(|rv_p^a|+|rh_{p,c}^a|)
		\big(|\Phi_{rr}|+|\Pi_{rr}|+|\Phi_{r\theta}|\big)\right.\\
		&\qquad\qquad\left.
		+(1-r)\big(|(rv_p^a)_r|+|(rh_{p,c}^a)_r|\big)
		\frac{|\Phi_r|+|\Pi_r|}{1-r}
		+(|v_p^a|+|h_{p,c}^a|)\frac{|\Phi_\theta|}{r}
		\right]\\
		&\leq C\varepsilon\eta\int r
		\left(|\widetilde{\Phi}_r|+|\widetilde{\Pi}_r|
		+\frac{|\widetilde{\Phi}_\theta|+|\widetilde{\Pi}_\theta|}{r}\right)
		\left(|\Phi_{rr}|+|\Pi_{rr}|+|\Phi_{r\theta}|
		+\frac{|\Phi_r|+|\Pi_r|}{1-r}+\frac{|\Phi_\theta|}{r}\right)\\
		&\leq C\eta\int\left(r\Phi_{r\theta}^2+r\Pi_{r\theta}^2
		+\frac{\Phi_{\theta\theta}^2}{r}+\frac{\Pi_{\theta\theta}^2}{r}\right)
		+C\eta\varepsilon^2\int r\big(|\nabla^2_{\!p}\Phi|^2
		+|\nabla^2_{\!p}\Pi|^2\big).
	\end{align*}
	Thus
	\begin{align*}
		|I_2|
		&\leq C\eta\int\left(r\Phi_{r\theta}^2+r\Pi_{r\theta}^2
		+\frac{\Phi_{\theta\theta}^2}{r}+\frac{\Pi_{\theta\theta}^2}{r}\right)\\
		&\quad+C\eta\varepsilon^2\int r\big(|\nabla^2_{\!p}\Phi|^2
		+|\nabla^2_{\!p}\Pi|^2\big)
		\leq \delta\mathbb E+C_\delta\eta\mathbb P .
	\end{align*}
	For $I_3$, notice that
	\[
		\Delta\Phi_e^a=(\Delta\Phi_e^a)^0
		=-2a-2\sum_{i=1}^{11}\varepsilon^i A_{i\infty}.
	\]
	Using
	\[
		(1-r)^3\big(|(\Delta\Phi_p^a)_r|+|(\Delta\Pi_p^a)_\theta|\big)
		+(1-r)^2|(\Delta\Phi_p^a)_\theta|
		+{\color{black}\varepsilon^{12}}\big(|(\Delta\Phi_L)_r|
		+|(\Delta\Phi_L)_\theta|\big)
		\leq C\varepsilon\eta,
	\]
	together with
	\[
		(1-r)^2|(\Delta\Pi_p^a)_r|\leq C\eta,\qquad
		(1-r)^3|(\Delta\Pi_p^a)_r|\leq C\varepsilon\eta,
	\]
	and Hardy's inequality, we have
	\begin{align*}
	I_3&=-\int \Phi_\theta(\Delta\Phi^a)_r\Phi
	+\int \Phi_r(\Delta\Phi^a)_\theta\Phi
	-\int \Pi_r(\Delta\Pi^a)_\theta\Phi,\\
		&=-\int \Phi_\theta\big((\Delta\Phi_p^a)_r+{\color{black}\varepsilon^{12}}(\Delta\Phi_L)_r\big)\Phi
		+\int \widetilde{\Phi}_r
		\big((\Delta\Phi_p^a)_\theta+{\color{black}\varepsilon^{12}}(\Delta\Phi_L)_\theta\big)\Phi\\
		&\quad+\int \Phi_r
		\big((\Delta\Phi_p^a)_\theta+{\color{black}\varepsilon^{12}}(\Delta\Phi_L)_\theta\big)\widetilde{\Phi}
		{-}\int \widetilde{\Phi}_r
		\big((\Delta\Phi_p^a)_\theta+{\color{black}\varepsilon^{12}}(\Delta\Phi_L)_\theta\big)\widetilde{\Phi}\\
		&\quad-\int \Pi_r
		(\Delta\Pi_p^a)_\theta\Phi\\
		&\leq C\varepsilon\eta\int
		\left(\left|\frac{\Phi_\theta}{1-r}\right|
		+\left|\frac{\Phi_r}{1-r}\right|
		+\left|\frac{\Pi_r}{1-r}\right|\right)
		\left(\left|\frac{\Phi}{(1-r)^2}\right|
		+\left|\frac{\widetilde{\Phi}}{1-r}\right|\right)\\
		&\leq C\eta(\mathbb E+\mathbb P).
	\end{align*}
	For $I_4$, using
$
		\Pi_\theta=\frac1a\left(G+b\Phi_\theta
		+\frac{\Pi^a_\theta}{r}\Phi_r^0-\frac{\Phi^a_\theta}{r}\Pi_r^0\right),
$
	and $(\Delta\Pi_e^a)_r=0$,
	we obtain
	\begin{align*}
		|I_4|
		&\leq C\left|\int G(\Delta\Pi_p^a)_r\Phi\right|\\
		&\quad+C\left|\int \Phi_\theta
		(\Delta\Pi_p^a)_r\Phi\right|\\
		&\quad+C\left|\int
		\left(\frac{\Pi^a_\theta}{r}\Phi_r^0-\frac{\Phi^a_\theta}{r}\Pi_r^0\right)
		(\Delta\Pi_p^a)_r\Phi\right|\\
		&\leq C\int
		(1-r)^2|(\Delta\Pi_p^a)_r|
		|G|\left|\frac{\Phi}{(1-r)^2}\right|\\
		&\quad+C\int
		(1-r)^3|(\Delta\Pi_p^a)_r|
		\left|\frac{\Phi_\theta}{1-r}\right|
		\left|\frac{\Phi}{(1-r)^2}\right|\\
		&\quad+C\int
		(1-r)^3|(\Delta\Pi_p^a)_r|
		\left(\left|\frac{\Phi_r^0}{1-r}\right|
		+\left|\frac{\Pi_r^0}{1-r}\right|\right)
		\left|\frac{\Phi}{(1-r)^2}\right|\\
		&\leq C\eta\int |G|\left|\frac{\Phi}{(1-r)^2}\right|
		+C(\varepsilon\eta+\varepsilon^{10}\eta)\int
		\left|\frac{\Phi_\theta}{1-r}\right|
		\left|\frac{\Phi}{(1-r)^2}\right|\\
		&\quad+C\varepsilon^2\eta\int
		\left(\left|\frac{\Phi_r^0}{1-r}\right|
		+\left|\frac{\Pi_r^0}{1-r}\right|\right)
		\left|\frac{\Phi}{(1-r)^2}\right|\\
		&\leq C\eta
		\left(\frac1{\varepsilon^2}\int rG^2\right)^{1/2}
		\left(\varepsilon^2\int r\left|\frac{\Phi}{(1-r)^2}\right|^2\right)^{1/2}\\
		&\quad+C(\varepsilon\eta+\varepsilon^{10}\eta)
		\left(\int\left(r\Phi_{r\theta}^2+\frac{\Phi_{\theta\theta}^2}{r}\right)\right)^{1/2}
		\left(\int r\left|\frac{\Phi}{(1-r)^2}\right|^2\right)^{1/2}\\
		&\quad+C\varepsilon^2\eta
		\left(\int r\big((\Phi_{rr}^0)^2+(\Pi_{rr}^0)^2\big)\right)^{1/2}
		\left(\int r\left|\frac{\Phi}{(1-r)^2}\right|^2\right)^{1/2}\\
		&\leq C\eta(\mathbb E+\mathbb P+\mathbb L).
	\end{align*}
	Finally,
	\[
		|I_5|\leq \delta\varepsilon^2\int r|\nabla^2_{\!p}\Phi|^2
		+\frac{C_\delta}{\varepsilon^2}\int\frac{F_1^2}{r}.
	\]
	Choosing $\delta$ small gives \eqref{e:basic-energy-estimate}.}
\end{proof}

\subsubsection{Energy estimate for the magnetic field}

In this subsection, we establish the basic magnetic energy estimate.
\begin{lemma}
	Let $(\Phi,\Pi)$ be a smooth solution of \eqref{error-equation}. Then
	for any $\delta>0$,
	\begin{align}\label{e:basic-magnetic-energy-estimate}
		\varepsilon^2\int r|\nabla^2_{\!p}\Pi|^2
		&\leq \delta\varepsilon^2\int r|\nabla^2_{\!p}\Phi|^2
		+C_\delta\mathbb L+C_\delta\eta(\mathbb E+\mathbb P)
		+\frac{C_\delta}{\varepsilon^2}\int\frac{F_2^2}{r}.
	\end{align}
\end{lemma}
\begin{proof}
	Multiplying the second equation in \eqref{error-equation} by $-\Delta\Pi$ gives
	\begin{align*}
		&-\int u^a\Pi_\theta\Delta\Pi
		+\int g^a\Phi_\theta\Delta\Pi
		+\int\Pi^a_\theta\Phi_r\Delta\Pi
		-\int\Phi^a_\theta\Pi_r\Delta\Pi
		+\varepsilon^2\int r(\Delta\Pi)^2\\
		&\qquad=-\int F_2\Delta\Pi .
	\end{align*}
	By Lemma~\ref{lem:polar-hessian} and $\Pi_r=0$ on $r=1$,
	\begin{align*}
		\varepsilon^2\int r(\Delta\Pi)^2
		-\varepsilon^2\int r|\nabla^2_{\!p}\Pi|^2
		=\varepsilon^2\int_{r=1}\Pi_\theta^2\geq0.
	\end{align*}
	Moreover,
	\begin{align*}
		-\int ar\Pi_\theta\Delta\Pi=0 .
	\end{align*}
	Thus
	\begin{align*}
		\varepsilon^2\int r(\Delta\Pi)^2
		\leq I_1+I_2+I_3+I_4+I_5,
	\end{align*}
	where
	\begin{align*}
		I_1=&\left|\int br\Phi_\theta\Delta\Pi\right| =\left|b\int r\Delta\Phi\,\Pi_\theta\right|,\qquad
		I_2=\left|\int (u^a-ar)\Pi_\theta\Delta\Pi\right|,\\
		I_3=&\left|\int (g^a-br)\Phi_\theta\Delta\Pi\right|,\qquad
		I_4=\left|\int\Pi^a_\theta\Phi_r\Delta\Pi
		-\int\Phi^a_\theta\Pi_r\Delta\Pi\right|,\qquad
		I_5=\left|\int F_2\Delta\Pi\right|.
	\end{align*}
	For $I_1$, using
	\[
		\Pi_\theta=\frac1a\left(G+b\Phi_\theta
		+\frac{\Pi^a_\theta}{r}\Phi_r^0
		-\frac{\Phi^a_\theta}{r}\Pi_r^0\right),
	\]
	we have
	\begin{align*}
		I_1
		&\leq C\left|\int r\Delta\Phi\,G\right|
		+C\left|\int r\Delta\Phi\,\Phi_\theta\right|\\
		&\quad+C\left|\int
		\left(\Pi^a_\theta\Phi_r^0-\Phi^a_\theta\Pi_r^0\right)
		\Delta\Phi\right|.
	\end{align*}
	Since $\int r\Delta\Phi\,\Phi_\theta=0$ and
	\begin{align*}
		&\int
		\left(\Pi^a_\theta\Phi_r^0-\Phi^a_\theta\Pi_r^0\right)
		\Delta\Phi\\
		=&\int
		\left(\Pi^a_\theta\Phi_r^0-\Phi^a_\theta\Pi_r^0\right)
		\Delta\widetilde{\Phi}\\
		=&-\int r\widetilde{\Phi}_r
		\left[\left(\frac{\Pi^a_\theta}{r}\right)_r\Phi_r^0
		+\frac{\Pi^a_\theta}{r}\Phi_{rr}^0
		-\left(\frac{\Phi^a_\theta}{r}\right)_r\Pi_r^0
		-\frac{\Phi^a_\theta}{r}\Pi_{rr}^0\right]\\
		&-\int\frac{\Phi_\theta}{r}
		\left[\left(\frac{\Pi^a_\theta}{r}\right)_\theta\Phi_r^0
		-\left(\frac{\Phi^a_\theta}{r}\right)_\theta\Pi_r^0\right],
	\end{align*}
	we get
	\begin{align*}
		I_1
		&\leq C\int r|\Delta\Phi|\,|G|
		+C\varepsilon\eta\int r|\widetilde{\Phi}_r|
		\left(\left|\frac{\Phi_r^0}{1-r}\right|
		+\left|\frac{\Pi_r^0}{1-r}\right|
		+|\Phi_{rr}^0|+|\Pi_{rr}^0|\right)\\
		&\quad+C\varepsilon\eta\int\frac{|\Phi_\theta|}{r}(1-r)
		\left(\left|\frac{\Phi_r^0}{1-r}\right|
		+\left|\frac{\Pi_r^0}{1-r}\right|\right)\\
		&\leq \delta\varepsilon^2\int r|\nabla^2_{\!p}\Phi|^2
		+C_\delta\mathbb L+C_\delta\eta(\mathbb E+\mathbb P).
	\end{align*}
	For $I_2$,
	\begin{align*}
		I_2
		&\leq \left|\int (u_e^a-ar+{\color{black}\varepsilon^{12}}u_L)
		\Pi_\theta\Delta\Pi\right|
		+\left|\int u_p^a\Pi_\theta\Delta\Pi\right|\\
		&\leq \int |u_e^a-ar+{\color{black}\varepsilon^{12}}u_L|
		|\Pi_\theta||\Delta\Pi|
		+\left|\int u_p^a\Pi_\theta\Delta\Pi\right|\\
		&\leq C\varepsilon\eta\int
		\sqrt r|\Delta\Pi|\frac{|\Pi_\theta|}{\sqrt r}
		+\left|\int u_p^a\Pi_\theta\Delta\Pi\right|\\
		&\leq C\varepsilon\eta\int
		\sqrt r|\Delta\Pi|\frac{|\Pi_\theta|}{\sqrt r}
		+C\left|\int u_p^aG\Delta\Pi\right|
		+C\left|\int u_p^a\Phi_\theta\Delta\Pi\right|\\
		&\quad+C\left|\int u_p^a
		\left(\frac{\Pi^a_\theta}{r}\Phi_r^0
		-\frac{\Phi^a_\theta}{r}\Pi_r^0\right)\Delta\Pi\right|\\
		&\leq C\varepsilon\eta\int
		\sqrt r|\Delta\Pi|\frac{|\Pi_\theta|}{\sqrt r}
		+C\eta\int r|G||\Delta\Pi|
		+C\varepsilon\eta\int r
		\left|\frac{\Phi_\theta}{1-r}\right||\Delta\Pi|\\
		&\quad+C\varepsilon^2\eta\int r
		\left(\left|\frac{\Phi_r^0}{1-r}\right|
		+\left|\frac{\Pi_r^0}{1-r}\right|\right)|\Delta\Pi|\\
		&\leq \delta\varepsilon^2\int r(\Delta\Pi)^2
		+C_\delta\eta(\mathbb E+\mathbb P+\mathbb L).
	\end{align*}
	{For $I_3$,
	\begin{align*}
		I_3
		&\leq \left|\int (g_e^a-br)
		\Phi_\theta\Delta\Pi\right|
		+\left|\int g_p^a\Phi_\theta\Delta\Pi\right|\\
		&\leq \int |g_e^a-br|
		|\Phi_\theta||\Delta\Pi|
		+\int(1-r)|g_p^a|
		\left|\frac{\Phi_\theta}{1-r}\right||\Delta\Pi|\\
		&\leq C\varepsilon\eta\int
		\sqrt r|\Delta\Pi|\frac{|\Phi_\theta|}{\sqrt r}
		+C\varepsilon\eta\int r
		\left|\frac{\Phi_\theta}{1-r}\right||\Delta\Pi|\\
		&\leq \delta\varepsilon^2\int r(\Delta\Pi)^2
		+C_\delta\eta\mathbb P .
	\end{align*}}
	For $I_4$, since
	\[
		\int\left(\Pi^a_\theta\Phi_r^0
		-\Phi^a_\theta\Pi_r^0\right)\Delta\Pi^0=0,
	\]
	we get
	\begin{align*}
		I_4
		&\leq C\varepsilon\eta\int r
		\left(|\widetilde{\Phi}_r|+|\widetilde{\Pi}_r|\right)|\Delta\Pi|
		+\left|\int
		\left(\Pi^a_\theta\Phi_r^0
		-\Phi^a_\theta\Pi_r^0\right)\Delta\widetilde{\Pi}\right|.
	\end{align*}
	The last term is
	\begin{align*}
		&\int
		\left(\Pi^a_\theta\Phi_r^0
		-\Phi^a_\theta\Pi_r^0\right)\Delta\widetilde{\Pi}\\
		=&-\int r\widetilde{\Pi}_r
		\left[\left(\frac{\Pi^a_\theta}{r}\right)_r\Phi_r^0
		+\frac{\Pi^a_\theta}{r}\Phi_{rr}^0
		-\left(\frac{\Phi^a_\theta}{r}\right)_r\Pi_r^0
		-\frac{\Phi^a_\theta}{r}\Pi_{rr}^0\right]\\
		&-\int\frac{\Pi_\theta}{r}
		\left[\left(\frac{\Pi^a_\theta}{r}\right)_\theta\Phi_r^0
		-\left(\frac{\Phi^a_\theta}{r}\right)_\theta\Pi_r^0\right].
	\end{align*}
	Hence
	\begin{align*}
		I_4
		&\leq C\varepsilon\eta\int r
		\left(|\widetilde{\Phi}_r|+|\widetilde{\Pi}_r|\right)|\Delta\Pi|\\
		&\quad+C\varepsilon\eta\int r|\widetilde{\Pi}_r|
		\left(\left|\frac{\Phi_r^0}{1-r}\right|
		+\left|\frac{\Pi_r^0}{1-r}\right|
		+|\Phi_{rr}^0|+|\Pi_{rr}^0|\right)\\
		&\quad+C\varepsilon\eta\int\frac{|\Pi_\theta|}{r}(1-r)
		\left(\left|\frac{\Phi_r^0}{1-r}\right|
		+\left|\frac{\Pi_r^0}{1-r}\right|\right)\\
		&\leq \delta\varepsilon^2\int r(\Delta\Pi)^2
		+C_\delta\eta(\mathbb E+\mathbb P).
	\end{align*}
	Finally,
	\begin{align*}
		I_5
		\leq \delta\varepsilon^2\int r(\Delta\Pi)^2
		+\frac{C_\delta}{\varepsilon^2}\int\frac{F_2^2}{r}.
	\end{align*}
	Choosing the small constants in front of
	$\varepsilon^2\int r(\Delta\Pi)^2$ and using
	$\int r(\Delta\Pi)^2\geq\int r|\nabla^2_{\!p}\Pi|^2$
	gives \eqref{e:basic-magnetic-energy-estimate}.
\end{proof}

\subsubsection{Positivity estimate for velocity}

In this subsection, we establish the following positivity estimate.
\begin{lemma}
	Let $(\Phi,\Pi)$ be a smooth solution of \eqref{error-equation}. Then, for any $\delta>0$, there exists $\eta_0>0$ such that for any $\eta\in (0,\eta_0)$,
	\begin{align}\label{e:positivity-estimate}
		\int\left(r\Phi_{r\theta}^2+\frac{\Phi_{\theta\theta}^2}{r}\right)
		\leq \delta\mathbb{E}+C_\delta\mathbb{L}
		+C_\delta\eta(\mathbb E+\mathbb P+\mathbb L)
		+\frac{C_\delta}{\varepsilon^2}\int\frac{F_1^2}{r}.
	\end{align}
\end{lemma}
\begin{proof}
	Multiplying the first equation in \eqref{error-equation} by $-\Phi_\theta$ gives
	\begin{align*}
		|I_1|\leq I_2+I_3+I_4+I_5+I_6,
	\end{align*}
	where
	\begin{align*}
		I_1=&-\int ar(\Delta\Phi)_\theta\Phi_\theta
		+\int br(\Delta\Pi)_\theta\Phi_\theta,\\
		I_2=&\ \kappa\varepsilon^2\left|\int r\Delta^2\Phi\,\Phi_\theta\right|,\\
		I_3=&\left|\int (u^a-ar)(\Delta\Phi)_\theta\Phi_\theta
		-\int (g^a-br)(\Delta\Pi)_\theta\Phi_\theta
		+\int rv^a(\Delta\Phi)_r\Phi_\theta
		-\int rh^a(\Delta\Pi)_r\Phi_\theta\right|,\\
		I_4=&\left|\int \Phi_\theta(\Delta\Phi^a)_r\Phi_\theta
		-\int \Phi_r(\Delta\Phi^a)_\theta\Phi_\theta
		+\int \Pi_r(\Delta\Pi^a)_\theta\Phi_\theta\right|,\\
		I_5=&\left|\int \Pi_\theta(\Delta\Pi^a)_r\Phi_\theta\right|,
		\qquad
		I_6=\left|\int F_1\Phi_\theta\right|.
	\end{align*}
	For $I_1$,
	\begin{align*}
		I_1
		=&-\int ar(\Delta\Phi)_\theta\Phi_\theta
		+\int br(\Delta\Pi)_\theta\Phi_\theta\\
		=&\ a\int\left(r\Phi_{r\theta}^2+\frac{\Phi_{\theta\theta}^2}{r}\right)
		+b\int r\Pi_\theta(\Delta\Phi)_\theta\\
		=&\ a\int\left(r\Phi_{r\theta}^2+\frac{\Phi_{\theta\theta}^2}{r}\right)
		+\frac{b}{a}\int r
		\left(G+b\Phi_\theta
		+\frac{\Pi^a_\theta}{r}\Phi_r^0
		-\frac{\Phi^a_\theta}{r}\Pi_r^0\right)(\Delta\Phi)_\theta\\
		=&\left(a-\frac{b^2}{a}\right)
		\int\left(r\Phi_{r\theta}^2+\frac{\Phi_{\theta\theta}^2}{r}\right)
		+\frac{b}{a}\int rG(\Delta\Phi)_\theta\\
		&+\frac{b}{a}\int
		\left(\Pi^a_\theta\Phi_r^0-\Phi^a_\theta\Pi_r^0\right)(\Delta\Phi)_\theta .
	\end{align*}
	Since $a=\alpha+O(\eta^2)$, $b=\beta$, and $|\alpha|\neq|\beta|$,
	\[
		\left|a-\frac{b^2}{a}\right|
		=\frac{|a^2-b^2|}{|a|}\geq c_0>0
	\]
	for $\eta$ sufficiently small. Multiplying the entire identity by $\operatorname{sgn}((a^2-b^2)/a)$ if necessary, the coefficient contributes with the favorable sign. Hence
	\begin{align*}
		c_0\int\left(r\Phi_{r\theta}^2+\frac{\Phi_{\theta\theta}^2}{r}\right)
		&\leq |I_1|+\left|\int rG(\Delta\Phi)_\theta\right|\\
		&\quad+\left|\frac{b}{a}\int
		\left(\Pi^a_\theta\Phi_r^0-\Phi^a_\theta\Pi_r^0\right)
		(\Delta\Phi)_\theta\right|.
	\end{align*}
	Moreover,
	\begin{align*}
		\left|\int rG(\Delta\Phi)_\theta\right|
		&\leq
		\left(\frac1{\varepsilon^2}\int rG^2\right)^{1/2}
		\left(\varepsilon^2\int r|(\Delta\Phi)_\theta|^2\right)^{1/2}
		\leq C_\delta\mathbb L+\delta\mathbb E ,
	\end{align*}
	For the last term in the lower bound for $I_1$, the zero mode again
	drops out after one $\theta$-integration, as in
	\eqref{type-1-estimate1}--\eqref{type-1-estimate2}. Hence
	\begin{align*}
		&\left|\frac{b}{a}\int
		\left(\Pi^a_\theta\Phi_r^0-\Phi^a_\theta\Pi_r^0\right)(\Delta\Phi)_\theta\right|\\
		=&\left|\frac{b}{a}\int
		\left(\Pi^a_{\theta\theta}\Phi_r^0-\Phi^a_{\theta\theta}\Pi_r^0\right)
		\Delta\widetilde{\Phi}\right|\\
		\leq& \delta\mathbb E+C_\delta\eta\mathbb P .
	\end{align*}
	Thus
	\begin{align*}
		c\int\left(r\Phi_{r\theta}^2+\frac{\Phi_{\theta\theta}^2}{r}\right)
		\leq |I_1|+C_\delta\mathbb L+\delta\mathbb E+C_\delta\eta\mathbb P .
	\end{align*}
	For $I_2$,
	\begin{align*}
		I_2
		=\kappa\varepsilon^2\left|\int r\Delta\Phi(\Delta\Phi)_\theta\right|=0 .
	\end{align*}
	{For $I_3$, we divide it into the Euler part and the Prandtl part:
	\begin{align*}
		I_3\leq I_{31}+I_{32},
	\end{align*}
	where
	\begin{align*}
		I_{31}=&\left|\int (u_e^a-ar+\varepsilon^{12}u_L)(\Delta\Phi)_\theta\Phi_\theta
		-\int (g_e^a-br)(\Delta\Pi)_\theta\Phi_\theta\right.\\
		&\left.\quad+\int rv_e^a(\Delta\Phi)_r\Phi_\theta
		-\int r(h_e^a+h_c^a)(\Delta\Pi)_r\Phi_\theta\right|,\\
		I_{32}=&\left|\int u_p^a(\Delta\Phi)_\theta\Phi_\theta
		-\int g_p^a(\Delta\Pi)_\theta\Phi_\theta
		+\int rv_p^a(\Delta\Phi)_r\Phi_\theta
		-\int rh_p^a(\Delta\Pi)_r\Phi_\theta\right|.
	\end{align*}
	For $I_{31}$,
	\begin{align*}
		I_{31}
		&\leq \int\big(|u_e^a-ar+\varepsilon^{12}u_L|
		|(\Delta\Phi)_\theta|
		+|g_e^a-br|
		|(\Delta\Pi)_\theta|\big)|\Phi_\theta|\\
		&\quad+\int r\big(|v_e^a||\Delta\Phi|
		+(|h_e^a|+|h_c^a|)|\Delta\Pi|\big)|\Phi_{r\theta}|\\
		&\leq C\varepsilon\eta\int
		\left(|(\Delta\Phi)_\theta|
		+|(\Delta\Pi)_\theta|\right)
		|\Phi_\theta|+C\varepsilon\eta\int r
		\left(|\Delta\Phi|
		+|\Delta\Pi|\right)
		|\Phi_{r\theta}|\\
		&\leq C\eta \mathbb E+C\eta\mathbb P .
	\end{align*}}
	For $I_{32}$,
	\begin{align*}
		I_{32}
		&\leq \int(1-r)
		\big(|u_p^a||(\Delta\Phi)_\theta|
		+|g_p^a||(\Delta\Pi)_\theta|\big)
		\left|\frac{\Phi_\theta}{1-r}\right|\\
		&\quad+\int(1-r)\big(|(rv_p^a)_r||\Delta\Phi|
		+|(rh_p^a)_r||\Delta\Pi|\big)
		\left|\frac{\Phi_\theta}{1-r}\right|\\
		&\quad+\int\big(|rv_p^a||\Delta\Phi|
		+|rh_p^a||\Delta\Pi|\big)|\Phi_{r\theta}|\\
		&\leq C\varepsilon\eta\int
		\left(|(\Delta\Phi)_\theta|+|(\Delta\Pi)_\theta|
		+|\Delta\Phi|+|\Delta\Pi|\right)
		\left|\frac{\Phi_\theta}{1-r}\right|\\
		&\quad+C\varepsilon\eta\int
		\left(|\Delta\Phi|+|\Delta\Pi|\right)|\Phi_{r\theta}|\\
		&\leq C\varepsilon\eta
		\left(\int r\big(|\nabla^2_{\!p}\Phi_\theta|^2
		+|\nabla^2_{\!p}\Pi_\theta|^2
		+|\nabla^2_{\!p}\Phi|^2
		+|\nabla^2_{\!p}\Pi|^2\big)\right)^{1/2}\times
		\left(\int r\Phi_{r\theta}^2\right)^{1/2}\\
		&\leq C\eta \mathbb E+C\eta\mathbb P .
	\end{align*}
	Therefore,
	\begin{align*}
		I_3\leq  C\eta \mathbb E+C\eta\mathbb P .
	\end{align*}
	For $I_4$, decomposing $\Delta\Phi^a$ and $\Delta\Pi^a$ into Euler,
	Prandtl and corrector parts as in the estimate of $I_3$ in the proof of
	Lemma \ref{basic-energy-estimate} gives
	\begin{align*}
		I_4 \leq C\eta \mathbb E+C\eta\mathbb P .
	\end{align*}
	{For $I_5$, since
	$
		(\Delta\Pi_e^a)_r=0,
	$
	we have
	$
		(\Delta\Pi^a)_r=(\Delta\Pi_p^a)_r,
	$
	and the definition of $G$ gives
	\begin{align*}
		I_5
		&\leq C\left|\int G(\Delta\Pi_p^a)_r\Phi_\theta\right|
		+C\left|\int \Phi_\theta(\Delta\Pi_p^a)_r\Phi_\theta\right|\\
		&\quad+C\left|\int
		\left(\frac{\Pi^a_\theta}{r}\Phi_r^0-\frac{\Phi^a_\theta}{r}\Pi_r^0\right)
		(\Delta\Pi_p^a)_r
		\Phi_\theta\right|\\
		&\leq C\int r(1-r)|(\Delta\Pi_p^a)_r|
		|G|\left|\frac{\Phi_\theta}{1-r}\right|
		+C\int r(1-r)^2|(\Delta\Pi_p^a)_r|
		\left|\frac{\Phi_\theta}{1-r}\right|^2\\
		&\quad+C\int r(1-r)^2|(\Delta\Pi_p^a)_r|
		\left(\left|\frac{\Phi_r^0}{1-r}\right|
		+\left|\frac{\Pi_r^0}{1-r}\right|\right)
		\left|\frac{\Phi_\theta}{1-r}\right|\\
		&\leq \frac{C\eta}{\varepsilon}\int r|G|
		\left|\frac{\Phi_\theta}{1-r}\right|
		+C\eta\int r
		\left|\frac{\Phi_\theta}{1-r}\right|^2\\
		&\quad+C\varepsilon\eta\int r
		\left(\left|\frac{\Phi_r^0}{1-r}\right|
		+\left|\frac{\Pi_r^0}{1-r}\right|\right)
		\left|\frac{\Phi_\theta}{1-r}\right|\\
		&\leq C\eta \mathbb E+C\eta\mathbb P+C\eta \mathbb L .
	\end{align*}}
	For $I_6$,
	\begin{align*}
		I_6
		&\leq
		\left(\frac1{\varepsilon^2}\int\frac{F_1^2}{r}\right)^{1/2}
		\left(\varepsilon^2\int r\Phi_\theta^2\right)^{1/2}\\
		&\leq \delta\int\left(r\Phi_{r\theta}^2
		+\frac{\Phi_{\theta\theta}^2}{r}\right)
		+\frac{C_\delta}{\varepsilon^2}\int\frac{F_1^2}{r}.
	\end{align*}
	Combining these estimates gives
	\eqref{e:positivity-estimate}.
\end{proof}

\subsubsection{Positivity estimate for the magnetic field}

\begin{lemma}
	Let $(\Phi,\Pi)$ be a smooth solution of \eqref{error-equation}. Then
	for any $\delta>0$,
	\begin{align}\label{e:positivity-estimate-mag}
		\mathbb P_\Pi
		\leq \delta\mathbb E+C_\delta(\mathbb P_\Phi+\mathbb L)
		+C_\delta\eta(\mathbb{E}+\mathbb{P}+\mathbb{L})
		+\frac{C_\delta}{\varepsilon^2}\int\frac{F_{2,\theta}^2}{r}.
	\end{align}

\end{lemma}
\begin{proof}
	Multiplying the second equation in \eqref{error-equation} by $-(\Delta\Pi)_\theta$ gives
	\begin{align*}
		-\int u^a\Pi_\theta(\Delta\Pi)_\theta
		+\int \Pi^a_\theta\Phi_r(\Delta\Pi)_\theta
		-\int \Phi^a_\theta\Pi_r(\Delta\Pi)_\theta
		+\int g^a\Phi_\theta(\Delta\Pi)_\theta\\
		+\varepsilon^2\int r\Delta\Pi(\Delta\Pi)_\theta
		=-\int F_2(\Delta\Pi)_\theta .
	\end{align*}
	Since
	\[
		\varepsilon^2\int r\Delta\Pi(\Delta\Pi)_\theta=0,\qquad
		-\int ar\Pi_\theta(\Delta\Pi)_\theta
		=a\int\left(r\Pi_{r\theta}^2+\frac{\Pi_{\theta\theta}^2}{r}\right),
	\]
	we have
	\begin{align*}
		a\int\left(r\Pi_{r\theta}^2+\frac{\Pi_{\theta\theta}^2}{r}\right)
		\leq I_1+I_2+I_3+I_4+I_5,
	\end{align*}
	where
	\begin{align*}
		I_1=&\left|\int (u^a-ar)\Pi_\theta(\Delta\Pi)_\theta\right|,\\
		I_2=&\left|\int \Pi^a_\theta\Phi_r(\Delta\Pi)_\theta
		-\int \Phi^a_\theta\Pi_r(\Delta\Pi)_\theta\right|,\\
		I_3=&\left|\int (g^a-br)\Phi_\theta(\Delta\Pi)_\theta\right|,\\
		I_4=&\left|b\int r\Phi_\theta(\Delta\Pi)_\theta\right|,\qquad
		I_5=\left|\int F_2(\Delta\Pi)_\theta\right|.
	\end{align*}
	For $I_1$,
	\begin{align*}
		I_1
		&\leq \left|\int (u_e^a-ar+{\color{black}\varepsilon^{12}}u_L)
		\Pi_\theta(\Delta\Pi)_\theta\right|
		+\left|\int u_p^a\Pi_\theta(\Delta\Pi)_\theta\right|\\
		&\leq \int |u_e^a-ar+{\color{black}\varepsilon^{12}}u_L|
		|\Pi_\theta||(\Delta\Pi)_\theta|
		+\left|\int u_p^a\Pi_\theta(\Delta\Pi)_\theta\right|\\
		&\leq C\varepsilon\eta\int
		\sqrt r|(\Delta\Pi)_\theta|
		\frac{|\Pi_\theta|}{\sqrt r}
		+\left|\int u_p^a\Pi_\theta(\Delta\Pi)_\theta\right|.
	\end{align*}
	For the Prandtl part, using
	\[
		\Pi_\theta=\frac1a\left(G+b\Phi_\theta
		+\frac{\Pi^a_\theta}{r}\Phi_r^0-\frac{\Phi^a_\theta}{r}\Pi_r^0\right),
	\]
	we have
	\begin{align*}
		\left|\int u_p^a\Pi_\theta(\Delta\Pi)_\theta\right|
		&\leq C\left|\int u_p^aG(\Delta\Pi)_\theta\right|
		+C\left|\int u_p^a\Phi_\theta(\Delta\Pi)_\theta\right|\\
		&\quad+C\left|\int u_p^a
		\left(\frac{\Pi^a_\theta}{r}\Phi_r^0
		-\frac{\Phi^a_\theta}{r}\Pi_r^0\right)(\Delta\Pi)_\theta\right|\\
		&\leq C\eta\int r|G|\,|(\Delta\Pi)_\theta|+C\int(1-r)|u_p^a|
		\left|\frac{\Phi_\theta}{1-r}\right|
		|(\Delta\Pi)_\theta|\\
		&\quad+C\int(1-r)|u_p^a|
		\left|\frac{\Pi^a_\theta}{r}\frac{\Phi_r^0}{1-r}
		-\frac{\Phi^a_\theta}{r}\frac{\Pi_r^0}{1-r}\right|
		|(\Delta\Pi)_\theta|\\
		&\leq C\eta\int r|G|\,|(\Delta\Pi)_\theta|
		+C\varepsilon\eta\int r
		\left|\frac{\Phi_\theta}{1-r}\right|
		|(\Delta\Pi)_\theta|\\
		&\quad+C\varepsilon^2\eta\int r
		\left(\left|\frac{\Phi_r^0}{1-r}\right|
		+\left|\frac{\Pi_r^0}{1-r}\right|\right)
		|(\Delta\Pi)_\theta|\\
		&\leq \delta\mathbb E+C_\delta\eta(\mathbb P+\mathbb L).
	\end{align*}
	Thus
	\begin{align*}
		I_1
		\leq \delta\mathbb E+C_\delta\eta(\mathbb P+\mathbb L).
	\end{align*}
	For $I_2$,
	\begin{align*}
		I_2
		=&\left|-\int\left(\Pi^a_{\theta\theta}\Phi_r
		+\Pi^a_\theta\Phi_{r\theta}\right)\Delta\Pi
		+\int\left(\Phi^a_{\theta\theta}\Pi_r
		+\Phi^a_\theta\Pi_{r\theta}\right)\Delta\Pi\right|\\
		&\leq C\varepsilon\eta\int r
		\left(|\Phi_{r\theta}|+|\Pi_{r\theta}|\right)|\Delta\Pi|+\left|\int
		\left(\Pi^a_{\theta\theta}\Phi_r
		-\Phi^a_{\theta\theta}\Pi_r\right)\Delta\Pi\right|.
	\end{align*}
	Since
	\[
		\int\left(\Pi^a_{\theta\theta}\Phi_r^0
		-\Phi^a_{\theta\theta}\Pi_r^0\right)\Delta\Pi^0=0,
	\]
	we have
	\begin{align*}
		&\left|\int
		\left(\Pi^a_{\theta\theta}\Phi_r
		-\Phi^a_{\theta\theta}\Pi_r\right)\Delta\Pi\right|\\
		&\leq C\varepsilon\eta\int r
		\left(|\widetilde{\Phi}_r|+|\widetilde{\Pi}_r|\right)|\Delta\Pi|+\left|\int
		\left(\Pi^a_{\theta\theta}\Phi_r^0
		-\Phi^a_{\theta\theta}\Pi_r^0\right)\Delta\widetilde{\Pi}\right|\\
		&\leq C\varepsilon\eta\int r
		\left(|\widetilde{\Phi}_r|+|\widetilde{\Pi}_r|\right)|\Delta\Pi|\\
		&\quad+C\varepsilon\eta\int r
		\left(|\Phi_{rr}^0|+|\Pi_{rr}^0|
		+\left|\frac{\Phi_r^0}{1-r}\right|
		+\left|\frac{\Pi_r^0}{1-r}\right|\right)
		\left(|\widetilde{\Pi}_r|+\frac{|\Pi_\theta|}{r}\right)\\
		&\leq \delta\mathbb E+C_\delta\eta\mathbb P .
	\end{align*}
	Thus
	\begin{align*}
		I_2\leq \delta\mathbb E+C_\delta\eta\mathbb P .
	\end{align*}
	{For $I_3$,
	\begin{align*}
		I_3
		&\leq \left|\int (g_e^a-br)
		\Phi_\theta(\Delta\Pi)_\theta\right|
		+\left|\int g_p^a\Phi_\theta(\Delta\Pi)_\theta\right|\\
		&\leq \int |g_e^a-br|
		|\Phi_\theta||(\Delta\Pi)_\theta|
		+\int(1-r)|g_p^a|
		\left|\frac{\Phi_\theta}{1-r}\right||(\Delta\Pi)_\theta|\\
		&\leq C\varepsilon\eta\int
		\sqrt r|(\Delta\Pi)_\theta|
		\frac{|\Phi_\theta|}{\sqrt r}
		+\int(1-r)|g_p^a|
		\left|\frac{\Phi_\theta}{1-r}\right||(\Delta\Pi)_\theta|\\
		&\leq C\eta\varepsilon\int
		r\left|\frac{\Phi_\theta}{1-r}\right||(\Delta\Pi)_\theta|\\
		&\quad+C\varepsilon\eta\int
		\sqrt r|(\Delta\Pi)_\theta|
		\frac{|\Phi_\theta|}{\sqrt r}
		\leq \delta\mathbb E+C_\delta\eta\mathbb P .
	\end{align*}}
	For $I_4$,
	\begin{align*}
		b\int r\Phi_\theta(\Delta\Pi)_\theta
		&=b\int r\Phi_\theta
		\left(\Pi_{rr\theta}+\frac{\Pi_{r\theta}}{r}
		+\frac{\Pi_{\theta\theta\theta}}{r^2}\right)\\
		&=-b\int r\Phi_{r\theta}\Pi_{r\theta}
		-b\int\frac{\Phi_{\theta\theta}\Pi_{\theta\theta}}{r}.
	\end{align*}
	Thus
	\begin{align*}
		I_4
		&\leq C\int\left(r|\Phi_{r\theta}\Pi_{r\theta}|
		+\frac{|\Phi_{\theta\theta}\Pi_{\theta\theta}|}{r}\right)\\
		&\leq C\mathbb P_\Phi+\frac{a}{2}\mathbb P_\Pi .
	\end{align*}
	For $I_5$,
	\begin{align*}
		I_5
		=\left|\int F_{2,\theta}\Delta\Pi\right|
		\leq \delta\mathbb{E}
		+\frac{C_\delta}{\varepsilon^2}\int\frac{F_{2,\theta}^2}{r}.
	\end{align*}
	Choosing $\eta$ sufficiently small gives \eqref{e:positivity-estimate-mag}.
\end{proof}

\subsubsection{Refined higher-order energy estimate}

In this subsection, we establish the following higher-order energy estimate of \eqref{error-equation}.

\begin{lemma}
	Let $(\Phi,\Pi)$ be a smooth solution of \eqref{error-equation}. Then
	\begin{align}\label{e:positive-estimate2}
		&\varepsilon^2\int r\big(|\nabla^2_{\!p}\Phi_\theta|^2+|\nabla^2_{\!p}\Pi_\theta|^2\big)\nonumber\\
		&\qquad\leq C\eta(\mathbb{E}+\mathbb{P}+\mathbb{L})
		+\frac{C}{\varepsilon^2}\int\frac{F_1^2+F_2^2+F_{2,\theta}^2}{r}.
	\end{align}
\end{lemma}
\begin{proof}
	We multiply the first and second equations in \eqref{error-equation} by
	\[
		\frac{r\Phi_{\theta\theta}}{u_e+u_p^a},\qquad
		\frac{r(\Delta\Pi)_{\theta\theta}}{u_e+u_p^a},
	\]
	respectively.  Here and below the quotient
	$r/(u_e+u_p^a)$ is understood by continuous extension at $r=0$; it equals
	$1/a$ near the origin because $u_p^a=0$ there.  This gives
	\begin{align*}
		&-\kappa\varepsilon^2\int r\Delta^2\Phi
		\frac{r\Phi_{\theta\theta}}{u_e+u_p^a}
		-\varepsilon^2\int r\Delta\Pi
		\frac{r(\Delta\Pi)_{\theta\theta}}{u_e+u_p^a}\\
		&\qquad= I_1+I_2,
	\end{align*}
	where
	\begin{align*}
		I_1
		=&-\int\Big[
		\Phi^a_\theta(\Delta\Phi)_r+\Phi_\theta(\Delta\Phi^a)_r
		-\Phi^a_r(\Delta\Phi)_\theta-\Phi_r(\Delta\Phi^a)_\theta\\
		&\qquad
		-\Pi^a_\theta(\Delta\Pi)_r-\Pi_\theta(\Delta\Pi^a)_r
		+\Pi^a_r(\Delta\Pi)_\theta+\Pi_r(\Delta\Pi^a)_\theta
		\Big]\frac{r\Phi_{\theta\theta}}{u_e+u_p^a}\\
		&-\int\Big[-\Phi^a_r\Pi_\theta-\Phi_r\Pi^a_\theta
		+\Phi^a_\theta\Pi_r+\Phi_\theta\Pi^a_r\Big]
		\frac{r(\Delta\Pi)_{\theta\theta}}{u_e+u_p^a},\\
		I_2
		=&\int F_1\frac{r\Phi_{\theta\theta}}{u_e+u_p^a}
		+\int F_2\frac{r(\Delta\Pi)_{\theta\theta}}{u_e+u_p^a}.
	\end{align*}
	First,
	\begin{align*}
		&-\kappa\varepsilon^2\int r\Delta^2\Phi
		\frac{r\Phi_{\theta\theta}}{u_e+u_p^a}
		-\varepsilon^2\int r\Delta\Pi
		\frac{r(\Delta\Pi)_{\theta\theta}}{u_e+u_p^a}\\
		=&-\frac{\kappa\varepsilon^2}{a}\int r\Delta^2\Phi\,\Phi_{\theta\theta}
		-\frac{\varepsilon^2}{a}\int r\Delta\Pi(\Delta\Pi)_{\theta\theta}\\
		&-\kappa\varepsilon^2\int r
		\left(\frac{r}{u_e+u_p^a}-\frac1a\right)
		\Delta^2\Phi\,\Phi_{\theta\theta}
		-\varepsilon^2\int r
		\left(\frac{r}{u_e+u_p^a}-\frac1a\right)
		\Delta\Pi(\Delta\Pi)_{\theta\theta}.
	\end{align*}
	For the constant-coefficient part,
	\begin{align*}
		-\int r\Delta^2\Phi\,\Phi_{\theta\theta}
		&=\int r\Delta^2\Phi_\theta\,\Phi_\theta
		=\int r(\Delta\Phi_\theta)^2
		=\int r|\nabla^2_{\!p}\Phi_\theta|^2,\\
		-\int r\Delta\Pi(\Delta\Pi)_{\theta\theta}
		&=\int r|(\Delta\Pi)_\theta|^2
		\geq\int r|\nabla^2_{\!p}\Pi_\theta|^2 .
	\end{align*}
	The last equality and inequality follow from
	Lemma~\ref{lem:polar-hessian}.
	For the coefficient error, since $u_e=ar$,
	\begin{align*}
		\frac{r}{u_e+u_p^a}-\frac1a
		&=-\frac{u_p^a}{a(u_e+u_p^a)},\\
		\left(\frac{r}{u_e+u_p^a}\right)_\theta
		&=-\frac{r(u_p^a)_\theta}{(u_e+u_p^a)^2},\\
		\left(\frac{r}{u_e+u_p^a}\right)_r
		&=\frac{u_p^a-r(u_p^a)_r}{(u_e+u_p^a)^2}.
	\end{align*}
	Thus
	\begin{align*}
		&\left|\frac{r}{u_e+u_p^a}-\frac1a\right|
		+\left|\left(\frac{r}{u_e+u_p^a}\right)_\theta\right|
		+(1-r)\left(
		\left|\left(\frac{r}{u_e+u_p^a}\right)_r\right|
		+\left|\left(\frac{r}{u_e+u_p^a}\right)_{r\theta}\right|\right)\\
		&\qquad+\varepsilon(1-r)\left(
		\left|\Delta\left(\frac{r}{u_e+u_p^a}\right)\right|
		+\left|\left(\Delta\left(\frac{r}{u_e+u_p^a}\right)\right)_\theta\right|\right)
		\leq C\eta ,
	\end{align*}
	we obtain
	\begin{align*}
		&-\int r
		\left(\frac{r}{u_e+u_p^a}-\frac1a\right)
		\Delta^2\Phi\,\Phi_{\theta\theta}\\
		=&\int r\left(\frac{r}{u_e+u_p^a}-\frac1a\right)
		(\Delta\Phi)_\theta^2
		+\int r\left(\frac{r}{u_e+u_p^a}\right)_\theta
		\Delta\Phi(\Delta\Phi)_\theta\\
		&-2\int r\Delta\Phi
		\left[
		\left(\frac{r}{u_e+u_p^a}\right)_r\Phi_{r\theta\theta}
		+\frac1{r^2}\left(\frac{r}{u_e+u_p^a}\right)_\theta
		\Phi_{\theta\theta\theta}\right]\\
		&-\int r\Delta\Phi\,
		\Delta\left(\frac{r}{u_e+u_p^a}\right)\Phi_{\theta\theta},
	\end{align*}
	and
	\begin{align*}
		&-2\int r\Delta\Phi
		\left(\frac{r}{u_e+u_p^a}\right)_r\Phi_{r\theta\theta}\\
		&\quad=2\int r\left(\frac{r}{u_e+u_p^a}\right)_r
		(\Delta\Phi)_\theta\Phi_{r\theta}
		+2\int r\left(\frac{r}{u_e+u_p^a}\right)_{r\theta}
		\Delta\Phi\,\Phi_{r\theta},\\
		&-\int r\Delta\Phi\,
		\Delta\left(\frac{r}{u_e+u_p^a}\right)\Phi_{\theta\theta}\\
		&\quad=\int r(\Delta\Phi)_\theta
		\Delta\left(\frac{r}{u_e+u_p^a}\right)\Phi_\theta
		+\int r\Delta\Phi
		\left(\Delta\left(\frac{r}{u_e+u_p^a}\right)\right)_\theta
		\Phi_\theta .
	\end{align*}
	For the first two terms, using the coefficient bounds,
	\begin{align*}
		&\varepsilon^2\left|
		\int r\left(\frac{r}{u_e+u_p^a}-\frac1a\right)
		(\Delta\Phi)_\theta^2
		+\int r\left(\frac{r}{u_e+u_p^a}\right)_\theta
		\Delta\Phi(\Delta\Phi)_\theta\right|\\
		&\leq C\eta\varepsilon^2
		\int r\big((\Delta\Phi)_\theta^2+(\Delta\Phi)^2\big)\\
		&\leq C\eta\varepsilon^2\int r|\nabla^2_{\!p}\Phi_\theta|^2
		+C\eta\varepsilon^2\int r|\nabla^2_{\!p}\Phi|^2 .
	\end{align*}
	For the terms containing $\Phi_{r\theta\theta}$ and
	$\Phi_{\theta\theta\theta}$, using $\Phi_{r\theta}=0$ on $r=1$ and Hardy's
	inequality,
	\begin{align*}
		&\varepsilon^2\left|
		\int r\Delta\Phi
		\left[
		\left(\frac{r}{u_e+u_p^a}\right)_r\Phi_{r\theta\theta}
		+\frac1{r^2}\left(\frac{r}{u_e+u_p^a}\right)_\theta
		\Phi_{\theta\theta\theta}\right]\right|\\
		&\leq C\eta\varepsilon^2\int r
		\big(|(\Delta\Phi)_\theta|+|\Delta\Phi|\big)
		\left|\frac{\Phi_{r\theta}}{1-r}\right|
		+C\eta\varepsilon^2\int r|\Delta\Phi|
		\frac{|\Phi_{\theta\theta\theta}|}{r^2}\\
		&\leq C\eta\varepsilon^2
		\int r\big((\Delta\Phi)_\theta^2+(\Delta\Phi)^2\big)
		+C\eta\varepsilon^2\int r\left|\frac{\Phi_{r\theta}}{1-r}\right|^2
		+C\eta\varepsilon^2\int \frac{\Phi_{\theta\theta\theta}^2}{r^3}\\
		&\leq C\eta\varepsilon^2\int r|\nabla^2_{\!p}\Phi_\theta|^2
		+C\eta\varepsilon^2\int r|\nabla^2_{\!p}\Phi|^2 .
	\end{align*}
	For the last term, using $\Phi_\theta=0$ on $r=1$ and Hardy's inequality,
	\begin{align*}
		&\varepsilon^2\left|
		\int r\Delta\Phi\,
		\Delta\left(\frac{r}{u_e+u_p^a}\right)\Phi_{\theta\theta}\right|\\
		&\leq C\varepsilon^2\int r
		\big(|(\Delta\Phi)_\theta|+|\Delta\Phi|\big)
		(1-r)\left|\Delta\left(\frac{r}{u_e+u_p^a}\right)\right|
		\left|\frac{\Phi_\theta}{1-r}\right|\\
		&\quad+C\varepsilon^2\int r|\Delta\Phi|
		(1-r)\left|\left(\Delta\left(\frac{r}{u_e+u_p^a}\right)\right)_\theta\right|
		\left|\frac{\Phi_\theta}{1-r}\right|\\
		&\leq C\eta\varepsilon
		\int r\big(|(\Delta\Phi)_\theta|+|\Delta\Phi|\big)
		\left|\frac{\Phi_\theta}{1-r}\right|\\
		&\leq C\eta\varepsilon^2\int r|\nabla^2_{\!p}\Phi_\theta|^2
		+C\eta\varepsilon^2\int r|\nabla^2_{\!p}\Phi|^2
		+C\eta\int r\Phi_{r\theta}^2 .
	\end{align*}
	Hence
	\begin{align*}
		&\varepsilon^2\left|
		\int r
		\left(\frac{r}{u_e+u_p^a}-\frac1a\right)
		\Delta^2\Phi\,\Phi_{\theta\theta}\right|\\
		&\leq C\eta\varepsilon^2\int r|\nabla^2_{\!p}\Phi_\theta|^2
		+C\eta\varepsilon^2\int r|\nabla^2_{\!p}\Phi|^2
		+C\eta\int\left(r\Phi_{r\theta}^2+\frac{\Phi_{\theta\theta}^2}{r}\right).
	\end{align*}
	Similarly,
	\begin{align*}
		&-\int r
		\left(\frac{r}{u_e+u_p^a}-\frac1a\right)
		\Delta\Pi(\Delta\Pi)_{\theta\theta}\\
		=&\int r\left(\frac{r}{u_e+u_p^a}-\frac1a\right)
		|(\Delta\Pi)_\theta|^2
		+\int r\left(\frac{r}{u_e+u_p^a}\right)_\theta
		\Delta\Pi(\Delta\Pi)_\theta,
	\end{align*}
	so that
	\begin{align*}
		&\varepsilon^2\left|
		\int r
		\left(\frac{r}{u_e+u_p^a}-\frac1a\right)
		\Delta\Pi(\Delta\Pi)_{\theta\theta}\right|\\
		&\leq C\eta\varepsilon^2\int r|\nabla^2_{\!p}\Pi_\theta|^2
		+C\eta\varepsilon^2\int r|\nabla^2_{\!p}\Pi|^2 .
	\end{align*}
	Therefore,
	\begin{align*}
		&-\kappa\varepsilon^2\int r\Delta^2\Phi
		\frac{r\Phi_{\theta\theta}}{u_e+u_p^a}
		-\varepsilon^2\int r\Delta\Pi
		\frac{r(\Delta\Pi)_{\theta\theta}}{u_e+u_p^a}\\
		&\qquad\geq c\varepsilon^2\int r
		\big(|\nabla^2_{\!p}\Phi_\theta|^2
		+|\nabla^2_{\!p}\Pi_\theta|^2\big)
		-C\eta(\mathbb E+\mathbb P+\mathbb L).
	\end{align*}
	{Since
	\begin{align*}
		-\Phi_r^a&=u_e+u_p^a+(u_e^a-ar)+\varepsilon^{12}u_L,\\
		\Pi_r^a&=-br-g_p^a-(g_e^a-br),
	\end{align*}
	the leading velocity parts satisfy
	\begin{align*}
		&-\int (u_e+u_p^a)(\Delta\Phi)_\theta
		\frac{r\Phi_{\theta\theta}}{u_e+u_p^a}
		=-\int r(\Delta\Phi)_\theta\Phi_{\theta\theta}=0,\\
		&-\int (u_e+u_p^a)\Pi_\theta
		\frac{r(\Delta\Pi)_{\theta\theta}}{u_e+u_p^a}
		=-\int r\Pi_\theta(\Delta\Pi)_{\theta\theta}=0,
	\end{align*}
	Hence
	\begin{align*}
		I_1
		=&\underbrace{
		\int (br+g_p^a)(\Delta\Pi)_\theta
		\frac{r\Phi_{\theta\theta}}{u_e+u_p^a}
		+\int (br+g_p^a)\Phi_\theta
		\frac{r(\Delta\Pi)_{\theta\theta}}{u_e+u_p^a}}_{I_{1,1}}\\
		&+\underbrace{
		\int (g_e^a-br)(\Delta\Pi)_\theta
		\frac{r\Phi_{\theta\theta}}{u_e+u_p^a}
		+\int (g_e^a-br)\Phi_\theta
		\frac{r(\Delta\Pi)_{\theta\theta}}{u_e+u_p^a}}_{I_{1,2}}\\
		&+\underbrace{\begin{aligned}[t]
		&-\int\Big[
		(u_e^a-ar+\varepsilon^{12}u_L)(\Delta\Phi)_\theta
		+rv_e^a(\Delta\Phi)_r
		\Big]\frac{r\Phi_{\theta\theta}}{u_e+u_p^a}
		\end{aligned}}_{I_{1,3}}\\
		&+\underbrace{-\int rv_p^a(\Delta\Phi)_r
		\frac{r\Phi_{\theta\theta}}{u_e+u_p^a}}_{I_{1,4}}\\
		&+\underbrace{
		\int r(h_e^a+h_c^a)(\Delta\Pi)_r
		\frac{r\Phi_{\theta\theta}}{u_e+u_p^a}
		+\int r(h_e^a+h_c^a)\Phi_r
		\frac{r(\Delta\Pi)_{\theta\theta}}{u_e+u_p^a}}_{I_{1,5}}\\
		&+\underbrace{
		\int rh_p^a(\Delta\Pi)_r
		\frac{r\Phi_{\theta\theta}}{u_e+u_p^a}
		+\int rh_p^a\Phi_r
		\frac{r(\Delta\Pi)_{\theta\theta}}{u_e+u_p^a}}_{I_{1,6}}\\
		&+\underbrace{-\int\Big[
		\Phi_\theta(\Delta\Phi^a)_r-\Phi_r(\Delta\Phi^a)_\theta
		+\Pi_r(\Delta\Pi^a)_\theta\Big]
		\frac{r\Phi_{\theta\theta}}{u_e+u_p^a}}_{I_{1,7}}\\
		&+\underbrace{\int \Pi_\theta(\Delta\Pi^a)_r
		\frac{r\Phi_{\theta\theta}}{u_e+u_p^a}}_{I_{1,8}}\\
		&+\underbrace{-\int (u_e^a-ar+\varepsilon^{12}u_L)\Pi_\theta
		\frac{r(\Delta\Pi)_{\theta\theta}}{u_e+u_p^a}}_{I_{1,9}}\\
		&+\underbrace{-\int\Phi^a_\theta\Pi_r
		\frac{r(\Delta\Pi)_{\theta\theta}}{u_e+u_p^a}}_{I_{1,10}} .
	\end{align*}
	Thus $|I_1|\leq |I_{1,1}|+\cdots+|I_{1,10}|$.
	For $I_{1,1}$, we first integrate by parts in $\theta$:
	\begin{align*}
		I_{1,1}
		&=\int (br+g_p^a)\frac{r}{u_e+u_p^a}
		\left((\Delta\Pi)_\theta\Phi_\theta\right)_\theta\\
		&=-\int
		\partial_\theta\left((br+g_p^a)
		\frac{r}{u_e+u_p^a}\right)
		\Phi_\theta(\Delta\Pi)_\theta .
	\end{align*}
	Using
	\[
		(1-r)\left|
		\partial_\theta\left((br+g_p^a)\frac{r}{u_e+u_p^a}\right)
		\right|
		\leq C\varepsilon\eta\, r ,
	\]
	and then Hardy's inequality, we have
	\begin{align*}
		|I_{1,1}|
		&\leq C\int
		(1-r)\left|
		\partial_\theta\left((br+g_p^a)\frac{r}{u_e+u_p^a}\right)
		\right|
		\left|\frac{\Phi_\theta}{1-r}\right|
		|(\Delta\Pi)_\theta|\\
		&\leq C\varepsilon\eta\int r
		\left|\frac{\Phi_\theta}{1-r}\right|
		|(\Delta\Pi)_\theta|\\
		&\leq C\eta(\mathbb E+\mathbb P+\mathbb L).
	\end{align*}
	For $I_{1,2}$, the two terms must be kept together before integration by
	parts in $\theta$:
	\begin{align*}
		I_{1,2}
		&=\int (g_e^a-br)
		\frac{r}{u_e+u_p^a}
		\left((\Delta\Pi)_\theta\Phi_\theta\right)_\theta\\
		&=-\int
		\partial_\theta\left((g_e^a-br)
		\frac{r}{u_e+u_p^a}\right)
		\Phi_\theta(\Delta\Pi)_\theta .
	\end{align*}
	Using
	\[
		(1-r)\left|
		\partial_\theta\left((g_e^a-br)
		\frac{r}{u_e+u_p^a}\right)
		\right|
		\leq C\varepsilon\eta\, r ,
	\]
	and Hardy's inequality, we get
	\begin{align*}
		|I_{1,2}|
		&\leq C\int
		(1-r)\left|
		\partial_\theta\left((g_e^a-br)
		\frac{r}{u_e+u_p^a}\right)
		\right|
		\left|\frac{\Phi_\theta}{1-r}\right|
		|(\Delta\Pi)_\theta|\\
		&\leq C\varepsilon\eta\int r
		\left|\frac{\Phi_\theta}{1-r}\right|
		|(\Delta\Pi)_\theta|\\
		&\leq C\eta(\mathbb E+\mathbb P+\mathbb L).
	\end{align*}}
	For $I_{1,3}$, the tangential velocity remainder is estimated directly,
	whereas the normal Euler velocity term is first integrated by parts in $r$
	and then the high-order factors are written as a full $\theta$-derivative.
	For example,
	\begin{align*}
		&\int rv_e^a(\Delta\Phi)_r
		\frac{r\Phi_{\theta\theta}}{u_e+u_p^a}\\
		&\quad=-\int rv_e^a\frac{r}{u_e+u_p^a}
		\Delta\Phi\,\Phi_{r\theta\theta}
		-\int\partial_r\left(rv_e^a\frac{r}{u_e+u_p^a}\right)
		\Delta\Phi\,\Phi_{\theta\theta},
	\end{align*}
	where
	\begin{align*}
		&-\int rv_e^a\frac{r}{u_e+u_p^a}
		\Delta\Phi\,\Phi_{r\theta\theta}\\
		&\quad=-\int rv_e^a\frac{r}{u_e+u_p^a}
		\left(\Delta\Phi\,\Phi_{r\theta}\right)_\theta
		+\int rv_e^a\frac{r}{u_e+u_p^a}
		(\Delta\Phi)_\theta\Phi_{r\theta}\\
		&\quad=\int\partial_\theta\left(rv_e^a
		\frac{r}{u_e+u_p^a}\right)\Delta\Phi\,\Phi_{r\theta}
		+\int rv_e^a\frac{r}{u_e+u_p^a}
		(\Delta\Phi)_\theta\Phi_{r\theta},\\
		&-\int\partial_r\left(rv_e^a\frac{r}{u_e+u_p^a}\right)
		\Delta\Phi\,\Phi_{\theta\theta}\\
		&\quad=-\int\partial_r\left(rv_e^a\frac{r}{u_e+u_p^a}\right)
		\left(\Delta\Phi\,\Phi_\theta\right)_\theta
		+\int\partial_r\left(rv_e^a\frac{r}{u_e+u_p^a}\right)
		(\Delta\Phi)_\theta\Phi_\theta\\
		&\quad=\int\partial_\theta\partial_r\left(rv_e^a
		\frac{r}{u_e+u_p^a}\right)\Delta\Phi\,\Phi_\theta
		+\int\partial_r\left(rv_e^a\frac{r}{u_e+u_p^a}\right)
		(\Delta\Phi)_\theta\Phi_\theta .
	\end{align*}
	Therefore
	\begin{align*}
		|I_{1,3}|
		&\leq C\varepsilon\eta\int r
		|(\Delta\Phi)_\theta|
		\frac{|\Phi_{\theta\theta}|}{r}\\
		&\quad+C\varepsilon\eta\int r
		\big(|(\Delta\Phi)_\theta|+|\Delta\Phi|\big)
		\left(|\Phi_{r\theta}|+\frac{|\Phi_\theta|}{r}\right)\\
		&\leq C\eta(\mathbb E+\mathbb P+\mathbb L).
	\end{align*}
	For $I_{1,4}$, we treat the normal Prandtl velocity term in the same way, but we
	first display the Prandtl weights:
	\begin{align*}
		&\int rv_p^a(\Delta\Phi)_r
		\frac{r\Phi_{\theta\theta}}{u_e+u_p^a}\\
		&\quad=-\int rv_p^a\frac{r}{u_e+u_p^a}
		\Delta\Phi\,\Phi_{r\theta\theta}
		-\int\partial_r\left(rv_p^a\frac{r}{u_e+u_p^a}\right)
		\Delta\Phi\,\Phi_{\theta\theta}.
	\end{align*}
	Next,
	\begin{align*}
		&-\int rv_p^a\frac{r}{u_e+u_p^a}
		\Delta\Phi\,\Phi_{r\theta\theta}\\
		&\quad=-\int rv_p^a\frac{r}{u_e+u_p^a}
		\left(\Delta\Phi\,\Phi_{r\theta}\right)_\theta
		+\int rv_p^a\frac{r}{u_e+u_p^a}
		(\Delta\Phi)_\theta\Phi_{r\theta}\\
		&\quad=\int\partial_\theta\left(rv_p^a
		\frac{r}{u_e+u_p^a}\right)\Delta\Phi\,\Phi_{r\theta}
		+\int rv_p^a\frac{r}{u_e+u_p^a}
		(\Delta\Phi)_\theta\Phi_{r\theta},\\
		&-\int\partial_r\left(rv_p^a\frac{r}{u_e+u_p^a}\right)
		\Delta\Phi\,\Phi_{\theta\theta}\\
		&\quad=-\int\partial_r\left(rv_p^a\frac{r}{u_e+u_p^a}\right)
		\left(\Delta\Phi\,\Phi_\theta\right)_\theta
		+\int\partial_r\left(rv_p^a\frac{r}{u_e+u_p^a}\right)
		(\Delta\Phi)_\theta\Phi_\theta\\
		&\quad=\int\partial_\theta\partial_r\left(rv_p^a
		\frac{r}{u_e+u_p^a}\right)\Delta\Phi\,\Phi_\theta
		+\int\partial_r\left(rv_p^a\frac{r}{u_e+u_p^a}\right)
		(\Delta\Phi)_\theta\Phi_\theta .
	\end{align*}
	Hence
	\begin{align*}
		|I_{1,4}|
		&\leq C\int r|rv_p^a|
		|(\Delta\Phi)_\theta|\,|\Phi_{r\theta}|
		+C\int r|(rv_p^a)_\theta|
		|\Delta\Phi|\,|\Phi_{r\theta}|\\
		&\quad+C\int r(1-r)
		\left|\partial_r\left(rv_p^a\frac{r}{u_e+u_p^a}\right)\right|
		|(\Delta\Phi)_\theta|
		\left|\frac{\Phi_\theta}{1-r}\right|\\
		&\quad+C\int r(1-r)
		\left|\partial_\theta\partial_r\left(rv_p^a
		\frac{r}{u_e+u_p^a}\right)\right|
		|\Delta\Phi|\left|\frac{\Phi_\theta}{1-r}\right|\\
		&\leq C\varepsilon\eta\int r
		\big(|(\Delta\Phi)_\theta|+|\Delta\Phi|\big)
		|\Phi_{r\theta}|\\
		&\quad+C\varepsilon\eta\int r
		\big(|(\Delta\Phi)_\theta|+|\Delta\Phi|\big)
		\left|\frac{\Phi_\theta}{1-r}\right|\\
		&\leq C\eta(\mathbb E+\mathbb P+\mathbb L).
	\end{align*}
	{For $I_{1,5}$, set $h_{e,c}^a=h_e^a+h_c^a$.  The Euler bounds together
	with \eqref{magnetic-normal-correction-estimate} give the same coefficient
	bounds for $h_{e,c}^a$ as for $h_e^a$.  A combined integration by parts gives
	\begin{align*}
		I_{1,5}
		&=\int rh_{e,c}^a\frac{r}{u_e+u_p^a}
		\left((\Delta\Pi)_r\Phi_{\theta\theta}
		+\Phi_r(\Delta\Pi)_{\theta\theta}\right)\\
		&=\int rh_{e,c}^a\frac{r}{u_e+u_p^a}
		(\Delta\Pi)_r\Phi_{\theta\theta}
		+\int\partial_{\theta\theta}\left(rh_{e,c}^a
		\frac{r}{u_e+u_p^a}\Phi_r\right)\Delta\Pi\\
		&=\int rh_{e,c}^a\frac{r}{u_e+u_p^a}
		(\Delta\Pi)_r\Phi_{\theta\theta}
		+\int\partial_{\theta\theta}\left(rh_{e,c}^a
		\frac{r}{u_e+u_p^a}\right)\Phi_r\Delta\Pi\\
		&\quad+2\int\partial_\theta\left(rh_{e,c}^a
		\frac{r}{u_e+u_p^a}\right)\Phi_{r\theta}\Delta\Pi
		+\int rh_{e,c}^a\frac{r}{u_e+u_p^a}
		\Phi_{r\theta\theta}\Delta\Pi\\
		&=\int rh_{e,c}^a\frac{r}{u_e+u_p^a}
		(\Phi_{\theta\theta}\Delta\Pi)_r
		+\int\partial_{\theta\theta}\left(rh_{e,c}^a
		\frac{r}{u_e+u_p^a}\right)\Phi_r\Delta\Pi\\
		&\quad+2\int\partial_\theta\left(rh_{e,c}^a
		\frac{r}{u_e+u_p^a}\right)\Phi_{r\theta}\Delta\Pi\\
		&=-\int\partial_r\left(rh_{e,c}^a\frac{r}{u_e+u_p^a}\right)
		\Phi_{\theta\theta}\Delta\Pi\\
		&\quad+\int\partial_{\theta\theta}\left(rh_{e,c}^a
		\frac{r}{u_e+u_p^a}\right)\Phi_r\Delta\Pi\\
		&\quad+2\int\partial_\theta\left(rh_{e,c}^a
		\frac{r}{u_e+u_p^a}\right)\Phi_{r\theta}\Delta\Pi .
	\end{align*}
	Using the Euler bounds and Hardy's inequality for the $\Phi_r$ factor, we have
	\begin{align*}
		|I_{1,5}|
		&\leq C\varepsilon\eta\int r|\Delta\Pi|
		\left(\frac{|\Phi_{\theta\theta}|}{r}+|\Phi_{r\theta}|\right)\\
		&\quad+C\int (1-r)\left|
		\partial_{\theta\theta}\left(rh_{e,c}^a\frac{r}{u_e+u_p^a}\right)\right|
		\left|\frac{\Phi_r}{1-r}\right||\Delta\Pi|\\
		&\leq C\varepsilon\eta\int r|\Delta\Pi|
		\left(\frac{|\Phi_{\theta\theta}|}{r}+|\Phi_{r\theta}|\right)
		+C\varepsilon\eta\int r
		\left|\frac{\Phi_r}{1-r}\right||\Delta\Pi|\\
		&\leq C\eta(\mathbb E+\mathbb P+\mathbb L).
	\end{align*}}
	For $I_{1,6}$, the Prandtl coefficient is handled instead as in
	$I_{1,4}$. First,
	\begin{align*}
		I_{1,6}
		&=\int rh_p^a\frac{r}{u_e+u_p^a}
		\left((\Delta\Pi)_r\Phi_{\theta\theta}
		+\Phi_r(\Delta\Pi)_{\theta\theta}\right).
	\end{align*}
	For the first term, we integrate by parts in $r$ and then move one
	$\theta$ derivative from the high-order factor:
	\begin{align*}
		&\int rh_p^a\frac{r}{u_e+u_p^a}
		(\Delta\Pi)_r\Phi_{\theta\theta}\\
		&=-\int rh_p^a\frac{r}{u_e+u_p^a}
		\Delta\Pi\,\Phi_{r\theta\theta}
		-\int\partial_r\left(rh_p^a
		\frac{r}{u_e+u_p^a}\right)\Delta\Pi\,\Phi_{\theta\theta}\\
		&=\int\partial_\theta\left(rh_p^a
		\frac{r}{u_e+u_p^a}\right)\Delta\Pi\,\Phi_{r\theta}
		+\int rh_p^a\frac{r}{u_e+u_p^a}
		(\Delta\Pi)_\theta\Phi_{r\theta}\\
		&\quad+\int\partial_\theta\partial_r\left(rh_p^a
		\frac{r}{u_e+u_p^a}\right)\Delta\Pi\,\Phi_\theta
		+\int\partial_r\left(rh_p^a
		\frac{r}{u_e+u_p^a}\right)(\Delta\Pi)_\theta\Phi_\theta .
	\end{align*}
	For the second term, we first integrate by parts in $\theta$:
	\begin{align*}
		&\int rh_p^a\frac{r}{u_e+u_p^a}
		\Phi_r(\Delta\Pi)_{\theta\theta}\\
		&=-\int rh_p^a\frac{r}{u_e+u_p^a}
		\Phi_{r\theta}(\Delta\Pi)_\theta
		-\int\partial_\theta\left(rh_p^a
		\frac{r}{u_e+u_p^a}\right)\Phi_r(\Delta\Pi)_\theta .
	\end{align*}
	Combining the two identities, the two terms with coefficient
	$rh_p^a r/(u_e+u_p^a)$ cancel, and
	\begin{align*}
		I_{1,6}
		&=\int\partial_\theta\left(rh_p^a
		\frac{r}{u_e+u_p^a}\right)\Delta\Pi\,\Phi_{r\theta}\\
		&\quad+\int\partial_\theta\partial_r\left(rh_p^a
		\frac{r}{u_e+u_p^a}\right)\Delta\Pi\,\Phi_\theta
		+\int\partial_r\left(rh_p^a
		\frac{r}{u_e+u_p^a}\right)(\Delta\Pi)_\theta\Phi_\theta\\
		&\quad-\int\partial_\theta\left(rh_p^a
		\frac{r}{u_e+u_p^a}\right)\Phi_r(\Delta\Pi)_\theta .
	\end{align*}
	Thus, displaying the Prandtl weights before applying Hardy's inequality,
	\begin{align*}
		|I_{1,6}|
		&\leq C\int r\left|
		\partial_\theta\left(rh_p^a\frac{r}{u_e+u_p^a}\right)\right|
		|\Delta\Pi|\,|\Phi_{r\theta}|\\
		&\quad+C\int r(1-r)\left|
		\partial_\theta\partial_r\left(rh_p^a
		\frac{r}{u_e+u_p^a}\right)\right|
		|\Delta\Pi|\left|\frac{\Phi_\theta}{1-r}\right|\\
		&\quad+C\int r(1-r)\left|
		\partial_r\left(rh_p^a\frac{r}{u_e+u_p^a}\right)\right|
		|(\Delta\Pi)_\theta|\left|\frac{\Phi_\theta}{1-r}\right|\\
		&\quad+C\int r(1-r)\left|
		\partial_\theta\left(rh_p^a\frac{r}{u_e+u_p^a}\right)\right|
		|(\Delta\Pi)_\theta|\left|\frac{\Phi_r}{1-r}\right|\\
		&\leq C\varepsilon\eta\int r
		\big(|\Delta\Pi|+|(\Delta\Pi)_\theta|\big)
		\left(|\Phi_{r\theta}|
		+\left|\frac{\Phi_\theta}{1-r}\right|\right)\\
		&\quad+C\varepsilon^2\eta\int r
		|(\Delta\Pi)_\theta|
		\left|\frac{\Phi_r}{1-r}\right|\\
		&\leq C\eta(\mathbb E+\mathbb P+\mathbb L).
	\end{align*}
	For $I_{1,7}$, we first move one $\theta$ derivative from
	$\Phi_{\theta\theta}$:
	\begin{align*}
		I_{1,7}
		=&\frac12\int
		\partial_\theta\left((\Delta\Phi^a)_r
		\frac{r}{u_e+u_p^a}\right)\Phi_\theta^2\\
		&-\int(\Delta\Phi^a)_\theta
		\frac{r}{u_e+u_p^a}\Phi_{r\theta}\Phi_\theta
		-\int\partial_\theta\left((\Delta\Phi^a)_\theta
		\frac{r}{u_e+u_p^a}\right)\Phi_r\Phi_\theta\\
		&+\int(\Delta\Pi^a)_\theta
		\frac{r}{u_e+u_p^a}\Pi_{r\theta}\Phi_\theta
		+\int\partial_\theta\left((\Delta\Pi^a)_\theta
		\frac{r}{u_e+u_p^a}\right)\Pi_r\Phi_\theta .
	\end{align*}
	Since the Euler parts of $\Delta\Phi^a$ and $\Delta\Pi^a$ are constants,
	only the Prandtl and $L$ parts remain in these coefficients. We use
	\begin{align*}
		&(1-r)^2\left|\partial_\theta\left((\Delta\Phi_p^a)_r
		\frac{r}{u_e+u_p^a}\right)\right|
		+(1-r)\big(|(\Delta\Phi_p^a)_\theta|
		+|(\Delta\Pi_p^a)_\theta|\big)\leq C\eta,\\
		&(1-r)^2\left(
		\left|\partial_\theta\left((\Delta\Phi_p^a)_\theta
		\frac{r}{u_e+u_p^a}\right)\right|
		+\left|\partial_\theta\left((\Delta\Pi_p^a)_\theta
		\frac{r}{u_e+u_p^a}\right)\right|\right)
		\leq C\varepsilon\eta .
	\end{align*}
	Thus
	\begin{align*}
		|I_{1,7}|
		&\leq C\int r(1-r)^2
		\left|\partial_\theta\left((\Delta\Phi_p^a)_r
		\frac{r}{u_e+u_p^a}\right)\right|
		\left|\frac{\Phi_\theta}{1-r}\right|^2\\
		&\quad+C\int r(1-r)
		\big(|(\Delta\Phi_p^a)_\theta|
		+|(\Delta\Pi_p^a)_\theta|\big)
		\big(|\Phi_{r\theta}|+|\Pi_{r\theta}|\big)
		\left|\frac{\Phi_\theta}{1-r}\right|\\
		&\quad+C\int r(1-r)^2
		\left(
		\left|\partial_\theta\left((\Delta\Phi_p^a)_\theta
		\frac{r}{u_e+u_p^a}\right)\right|
		+\left|\partial_\theta\left((\Delta\Pi_p^a)_\theta
		\frac{r}{u_e+u_p^a}\right)\right|\right)\\
		&\qquad\qquad\times
		\left(\left|\frac{\Phi_r}{1-r}\right|
		+\left|\frac{\Pi_r}{1-r}\right|\right)
		\left|\frac{\Phi_\theta}{1-r}\right|\\
		&\quad+C\varepsilon^{10}\eta\int r
		\left(\left|\frac{\Phi_\theta}{1-r}\right|
		+|\Phi_{r\theta}|+|\Pi_{r\theta}|
		+\left|\frac{\Phi_r}{1-r}\right|
		+\left|\frac{\Pi_r}{1-r}\right|\right)
		\left|\frac{\Phi_\theta}{1-r}\right|\\
		&\leq C\eta\int r
		\left|\frac{\Phi_\theta}{1-r}\right|^2\\
		&\quad+C\eta\int r
		\big(|\Phi_{r\theta}|+|\Pi_{r\theta}|\big)
		\left|\frac{\Phi_\theta}{1-r}\right|\\
		&\quad+C\varepsilon\eta\int r
		\left(\left|\frac{\Phi_r}{1-r}\right|
		+\left|\frac{\Pi_r}{1-r}\right|\right)
		\left|\frac{\Phi_\theta}{1-r}\right|\\
		&\leq C\eta(\mathbb E+\mathbb P+\mathbb L).
	\end{align*}
	For $I_{1,8}$, we do not use the definition of $G$ again.  The relevant
	term is
	\[
		\int \Pi_\theta(\Delta\Pi^a)_r
		\frac{r\Phi_{\theta\theta}}{u_e+u_p^a}.
	\]
	If this term is estimated by replacing $\Pi_\theta$ with $G$, then one
	either loses the $\varepsilon$ order or needs one more derivative of
	$\Phi$. Thus we use the second
	equation in \eqref{error-equation} and then integrate by parts in
	$\theta$ term by term:
	\[
		\Pi_\theta
		=\frac{F_2+\varepsilon^2 r\Delta\Pi+\Pi^a_\theta\Phi_r-\Phi^a_\theta\Pi_r
		+g^a\Phi_\theta-(u^a-u_e-u_p^a)\Pi_\theta}
		{u_e+u_p^a},
	\]
	Then
	\begin{align*}
		|I_{1,8}|
		&\leq C\left|\int F_2(\Delta\Pi^a)_r
		\frac{r\Phi_{\theta\theta}}{(u_e+u_p^a)^2}\right|
		+C\varepsilon^2\left|\int r\Delta\Pi(\Delta\Pi^a)_r
		\frac{r\Phi_{\theta\theta}}{(u_e+u_p^a)^2}\right|\\
		&\quad+C\left|\int
		(\Pi^a_\theta\Phi_r-\Phi^a_\theta\Pi_r+g^a\Phi_\theta)
		(\Delta\Pi^a)_r
		\frac{r\Phi_{\theta\theta}}{(u_e+u_p^a)^2}\right|\\
		&\quad+C\left|\int (u^a-u_e-u_p^a)\Pi_\theta
		(\Delta\Pi^a)_r
		\frac{r\Phi_{\theta\theta}}{(u_e+u_p^a)^2}\right|.
	\end{align*}
	For the source part, we integrate by parts in $\theta$:
	\begin{align*}
		&\left|\int F_2(\Delta\Pi^a)_r
		\frac{r\Phi_{\theta\theta}}{(u_e+u_p^a)^2}\right|\\
		&\leq C\left|\int F_{2,\theta}(\Delta\Pi^a)_r
		\frac{r\Phi_\theta}{(u_e+u_p^a)^2}\right|
		+C\left|\int F_2
		\partial_\theta\left((\Delta\Pi^a)_r
		\frac{r}{(u_e+u_p^a)^2}\right)\Phi_\theta\right|\\
		&\leq \frac{C\eta}{\varepsilon}\int
		\big(|F_{2,\theta}|+|F_2|\big)
		\left|\frac{\Phi_\theta}{1-r}\right|\\
		&\leq \frac{C}{\varepsilon^2}
		\int\frac{F_2^2+F_{2,\theta}^2}{r}
		+C\eta(\mathbb E+\mathbb P+\mathbb L).
	\end{align*}
	For the $\varepsilon^2 r\Delta\Pi$ part, the same integration by parts
	gives
	\begin{align*}
		&\varepsilon^2\left|\int r\Delta\Pi(\Delta\Pi^a)_r
		\frac{r\Phi_{\theta\theta}}{(u_e+u_p^a)^2}\right|\\
		&\leq C\varepsilon^2\left|\int r(\Delta\Pi)_\theta
		(\Delta\Pi^a)_r\frac{r\Phi_\theta}{(u_e+u_p^a)^2}\right|\\
		&\quad+C\varepsilon^2\left|\int r\Delta\Pi\,
		\partial_\theta\left((\Delta\Pi^a)_r
		\frac{r}{(u_e+u_p^a)^2}\right)\Phi_\theta\right|\\
		&\leq C\varepsilon\eta\int r
		\big(|(\Delta\Pi)_\theta|+|\Delta\Pi|\big)
		\left|\frac{\Phi_\theta}{1-r}\right|\\
		&\leq C\eta(\mathbb E+\mathbb P+\mathbb L).
	\end{align*}
	For the terms containing
	$\Pi^a_\theta\Phi_r-\Phi^a_\theta\Pi_r+g^a\Phi_\theta$, we again move the
	$\theta$ derivative from $\Phi_{\theta\theta}$:
	\begin{align*}
		&\left|\int
		(\Pi^a_\theta\Phi_r-\Phi^a_\theta\Pi_r+g^a\Phi_\theta)
		(\Delta\Pi^a)_r
		\frac{r\Phi_{\theta\theta}}{(u_e+u_p^a)^2}\right|\\
		&\leq C\eta\int r
		\left(|\Phi_{r\theta}|+|\Pi_{r\theta}|
		+\left|\frac{\Phi_\theta}{1-r}\right|\right)
		\left|\frac{\Phi_\theta}{1-r}\right|\\
		&\quad+C\varepsilon\eta\int r
		\left(\left|\frac{\Phi_r}{1-r}\right|
		+\left|\frac{\Pi_r}{1-r}\right|\right)
		\left|\frac{\Phi_\theta}{1-r}\right|\\
		&\leq C\eta(\mathbb E+\mathbb P+\mathbb L).
	\end{align*}
	For the remaining term, we use
	$u^a-u_e-u_p^a=u_e^a-ar+{\color{black}\varepsilon^{12}}u_L$ before integrating by parts:
	\begin{align*}
		&\left|\int (u^a-u_e-u_p^a)\Pi_\theta
		(\Delta\Pi^a)_r
		\frac{r\Phi_{\theta\theta}}{(u_e+u_p^a)^2}\right|\\
		&\leq C{\eta^2}\int r
		\left(\frac{|\Pi_{\theta\theta}|}{r}
		+\frac{|\Pi_\theta|}{r}\right)
		\left|\frac{\Phi_\theta}{1-r}\right|\\
		&\leq C\eta(\mathbb E+\mathbb P+\mathbb L).
	\end{align*}
	Thus,
	\begin{align*}
		|I_{1,8}|
		&\leq C\eta(\mathbb E+\mathbb P+\mathbb L)
		+\frac{C}{\varepsilon^2}
		\int\frac{F_2^2+F_{2,\theta}^2}{r}.
	\end{align*}
	For $I_{1,9}$, the factor $r/(u_e+u_p^a)$ is bounded and does not give an
	extra power of $r$. Hence we integrate by parts in $\theta$ first:
	\begin{align*}
		|I_{1,9}|
		&=\left|\int (u_e^a-ar+{\color{black}\varepsilon^{12}}u_L)\Pi_\theta
		\frac{r(\Delta\Pi)_{\theta\theta}}{u_e+u_p^a}\right|\\
		&\leq C\left|\int
		\partial_\theta\left((u_e^a-ar+{\color{black}\varepsilon^{12}}u_L)
		\frac{r}{u_e+u_p^a}\right)
		\Pi_\theta(\Delta\Pi)_\theta\right|\\
		&\quad+C\left|\int (u_e^a-ar+{\color{black}\varepsilon^{12}}u_L)
		\frac{r}{u_e+u_p^a}\Pi_{\theta\theta}
		(\Delta\Pi)_\theta\right|\\
		&\leq C\varepsilon\eta\int
		\left(\frac{|\Pi_\theta|}{\sqrt r}
		+\frac{|\Pi_{\theta\theta}|}{\sqrt r}\right)
		\sqrt r\,|(\Delta\Pi)_\theta|\\
		&\leq C\eta(\mathbb E+\mathbb P+\mathbb L).
	\end{align*}
	{For $I_{1,10}$, put
	$c_a=\Phi^a_\theta r/(u_e+u_p^a)$.  One integration by parts in $\theta$
	gives
	\[
	 I_{1,10}=\int((c_a)_\theta\Pi_r+c_a\Pi_{r\theta})(\Delta\Pi)_\theta.
	\]
	After writing $\Pi_r=\Pi_r^0+\widetilde\Pi_r$, the terms containing
	$\widetilde\Pi_r$ or $\Pi_{r\theta}$ satisfy
	\begin{align*}
	&\left|\int
	((c_a)_\theta\widetilde\Pi_r+c_a\Pi_{r\theta})(\Delta\Pi)_\theta\right|\\
	&\qquad\leq C\eta(\mathbb E+\mathbb P+\mathbb L).
	\end{align*}
	It remains to estimate
	\[
	 J_0=\int (c_a)_\theta\Pi_r^0(\Delta\Pi)_\theta.
	\]
	Set $D=(c_a)_\theta/r$.  Since
	\[
	 r(\Delta\Pi)_\theta=(r\Pi_{r\theta})_r
	 +\frac1r\Pi_{\theta\theta\theta},
	\]
	integration by parts first in $r$ and then in $\theta$ gives
	\begin{align*}
	J_0
	&=-\int(D_r\Pi_r^0+D\Pi_{rr}^0)r\Pi_{r\theta}
	-\int\frac{D_\theta}{r}\Pi_r^0\Pi_{\theta\theta}.
	\end{align*}
	The radial boundary term is zero at $r=1$ because
	$\Pi_{r\theta}=0$, and it is zero at $r=0$ by regularity.  The profile
	bounds give, for $0\leq j\leq2$,
	\[
	 \left|\partial_\theta^j\left(\frac {c_a}r\right)\right|
	 +(1-r)\left|\partial_r\partial_\theta^j
	 \left(\frac {c_a}r\right)\right|
	 \leq C\varepsilon\eta.
	\]
	The boundary Hardy inequality for $\Pi_r^0$, the zero-mode terms contained
	in $\mathbb E$, and tangential Poincar\'e for $\widetilde\Pi$ give
	\[
	 |J_0|\leq C\eta(\mathbb E+\mathbb P+\mathbb L).
	\]
	Thus,
	\begin{align*}
		|I_{1,10}|
		&\leq C\eta(\mathbb E+\mathbb P+\mathbb L).
	\end{align*}}
	Therefore,
	\begin{align*}
		|I_1|
		&\leq C\eta(\mathbb E+\mathbb P+\mathbb L)
		+\frac{C}{\varepsilon^2}
		\int\frac{F_2^2+F_{2,\theta}^2}{r}.
	\end{align*}
	For $I_2$,
	\begin{align*}
		|I_2|
		&\leq
		\left(\frac1{\varepsilon^2}\int\frac{F_1^2}{r}\right)^{1/2}
		\left(\varepsilon^2\int\frac{r^3\Phi_{\theta\theta}^2}
		{(u_e+u_p^a)^2}\right)^{1/2}
		+\left|\int F_2\frac{r(\Delta\Pi)_{\theta\theta}}{u_e+u_p^a}\right|.
	\end{align*}
	For the $F_2$ term in $I_2$,
	\begin{align*}
		\left|\int F_2\frac{r(\Delta\Pi)_{\theta\theta}}{u_e+u_p^a}\right|
		&=\left|-\int F_{2,\theta}\frac{r(\Delta\Pi)_\theta}{u_e+u_p^a}
		-\int F_2\left(\frac{r}{u_e+u_p^a}\right)_\theta
		(\Delta\Pi)_\theta\right|\\
		&\leq
		C\left(\frac1{\varepsilon^2}\int\frac{F_2^2+F_{2,\theta}^2}{r}\right)^{1/2}
		\left(\varepsilon^2\int r|\nabla^2_{\!p}\Pi_\theta|^2\right)^{1/2}.
	\end{align*}
	Thus
	\begin{align*}
		|I_2|
		&\leq \delta\varepsilon^2\int r
		\big(|\nabla^2_{\!p}\Phi_\theta|^2
		+|\nabla^2_{\!p}\Pi_\theta|^2\big)
		+\frac{C_\delta}{\varepsilon^2}\int\frac{F_1^2+F_2^2+F_{2,\theta}^2}{r}.
	\end{align*}
	Choosing $\delta$ small enough to absorb only the refined Hessian terms
	gives \eqref{e:positive-estimate2}.  The term
	$C\eta(\mathbb E+\mathbb P+\mathbb L)$ is retained here and is closed after
	this estimate is combined with the preceding energy and positivity
	estimates in Proposition~\ref{prop:linear-stability}.
\end{proof}
\subsubsection{Linear stability estimate}
\begin{proposition}\label{prop:linear-stability}
	Let $(\Phi,\Pi)$ be a smooth solution of \eqref{error-equation}. Then there exist $\varepsilon_0>0$ and $\eta_0>0$ such that, for any $\varepsilon\in (0,\varepsilon_0)$ and $\eta\in(0,\eta_0)$,
	\begin{align}\label{eq:linear-stability}
		\mathbb{E}+\mathbb{P}+\mathbb{L}
		\leq C\frac{1}{\varepsilon^2}\int\frac{F_1^2+F_2^2+F_{2,\theta}^2}{r}.
	\end{align}

\end{proposition}
\begin{proof}
	By \eqref{e:combination-estimate} and \eqref{e:basic-energy-estimate},
	\[
		b\int r\Delta\Pi\,\Phi_\theta
		+b\int r(\Delta\Pi)_\theta\Phi
		=b\int r\Delta\Pi\,\Phi_\theta
		-b\int r\Delta\Pi\,\Phi_\theta=0,
	\]
	so
	\begin{align*}
		&\varepsilon^2\int r|\nabla^2_{\!p}\Phi|^2+\mathbb L\\
		&\qquad\leq C\eta(\mathbb E+\mathbb P+\mathbb L)
		+\frac{C}{\varepsilon^2}\int\frac{F_1^2+F_2^2}{r}.
	\end{align*}
	Using \eqref{e:basic-magnetic-energy-estimate} and the last inequality,
	\begin{align*}
		&\varepsilon^2\int r
		\big(|\nabla^2_{\!p}\Phi|^2+|\nabla^2_{\!p}\Pi|^2\big)
		+\mathbb L\\
		&\qquad\leq C_\delta\eta(\mathbb E+\mathbb P+\mathbb L)
		+\delta\mathbb E
		+\frac{C_\delta}{\varepsilon^2}\int\frac{F_1^2+F_2^2}{r}.
	\end{align*}
	Using \eqref{e:positivity-estimate}, \eqref{e:positivity-estimate-mag}, and
	\eqref{e:positive-estimate2},
	\begin{align*}
		\mathbb P_\Phi
		&\leq \delta\mathbb E+C_\delta\mathbb L
		+C_\delta\eta(\mathbb E+\mathbb P+\mathbb L)
		+\frac{C_\delta}{\varepsilon^2}\int\frac{F_1^2}{r},\\
		\mathbb P_\Pi
		&\leq \delta\mathbb E+C_\delta(\mathbb P_\Phi+\mathbb L)
		+C\eta(\mathbb E+\mathbb P+\mathbb L)
		+\frac{C_\delta}{\varepsilon^2}\int\frac{F_{2,\theta}^2}{r},\\
		\varepsilon^2\int r\big(|\nabla^2_{\!p}\Phi_\theta|^2
		+|\nabla^2_{\!p}\Pi_\theta|^2\big)
		&\leq C\eta(\mathbb E+\mathbb P+\mathbb L)
		+\frac{C}{\varepsilon^2}\int\frac{F_1^2+F_2^2+F_{2,\theta}^2}{r}.
	\end{align*}
	Therefore, after multiplying the first line by a sufficiently large fixed
	constant and adding the last three inequalities,
	\begin{align*}
		\mathbb E+\mathbb P+\mathbb L
		\leq C(\delta+\eta)(\mathbb E+\mathbb P+\mathbb L)
		+\frac{C_\delta}{\varepsilon^2}
		\int\frac{F_1^2+F_2^2+F_{2,\theta}^2}{r}.
		\end{align*}
		Choosing $\delta$ and $\eta$ small gives \eqref{eq:linear-stability}.
		{\color{black}The calculation above was made for smooth functions.
		Approximation by smooth solenoidal Dirichlet fields and smooth mean-zero
		Neumann potentials extends the estimate to the strong solutions used in
		Section~4.}
\end{proof}
\section{Solvability of the error equations}

In this section, we solve the nonlinear error equations \eqref{error-equation}
in the polar variables $(r,\theta)$. Recall that
\[
	u=-\Phi_r,\qquad v=\frac{\Phi_\theta}{r},\qquad
	g=-\Pi_r,\qquad h=\frac{\Pi_\theta}{r}.
\]
The solution constructed below has velocity and magnetic fields in
$H^2(B_1)$, and the equations hold almost everywhere.

\subsection{Linear solvability}

We prove solvability in the original velocity components and magnetic
potential.  Put
\[
\mathbf U=\nabla^\perp\Phi,\qquad
\mathbf H=\nabla^\perp\Pi,\qquad
\mathbf U^a=\nabla^\perp\Phi^a,\qquad
\mathbf B^a=\nabla^\perp\Pi^a,
\]
and impose the normalizations $\int_{B_1}\Pi\,\ud S=0$ and
{$\int_{B_1}q\,\ud S=0$}.  In Cartesian variables the
linearized equations have the form
\begin{equation}\label{component-linear-problem}
\left\{
\begin{aligned}
&(\mathbf U^a\!\cdot\nabla)\mathbf U
+(\mathbf U\!\cdot\nabla)\mathbf U^a
-(\mathbf B^a\!\cdot\nabla)\mathbf H
-(\mathbf H\!\cdot\nabla)\mathbf B^a
-\kappa\varepsilon^2\Delta\mathbf U+\nabla q=\mathbf f,\\
&\mathbf U^a\!\cdot\nabla\Pi
+\mathbf U\!\cdot\nabla\Pi^a-\varepsilon^2\Delta\Pi=f_2,\\
&\nabla\cdot\mathbf U=0,\qquad
\mathbf U|_{\partial B_1}=0,\qquad
\partial_n\Pi|_{\partial B_1}=0,\qquad
\int_{B_1}\Pi\,\ud S=0,\qquad
{\int_{B_1}q\,\ud S=0}.
\end{aligned}
\right.
\end{equation}
Here \(\int_{B_1}f_2\,\ud S=0\), as is required by the Neumann equation.

The principal operator is the triangular system obtained by keeping only
the Laplacian in the magnetic-potential equation and keeping both the Stokes
operator and the highest-order magnetic term in the velocity equation:
\begin{equation}\label{triangular-principal-operator}
\left\{
\begin{aligned}
-\kappa\varepsilon^2\Delta\mathbf U
-(\mathbf B^a\!\cdot\nabla)\nabla^\perp\Pi+\nabla q&=\mathbf f,\\
-\varepsilon^2\Delta\Pi&=f_2,\\
\nabla\cdot\mathbf U=0,\qquad
\mathbf U|_{\partial B_1}=0,\qquad
\partial_n\Pi|_{\partial B_1}&=0,\qquad
\int_{B_1}\Pi\,\ud S=0,\qquad
{\int_{B_1}q\,\ud S=0}.
\end{aligned}
\right.
\end{equation}
{
The mean-zero Neumann problem first determines \(\Pi\) in
\eqref{triangular-principal-operator}, and the Dirichlet Stokes problem then
determines \((\mathbf U,q)\).  Hence the triangular operator
\(\mathcal A_{\rm tr}\) is invertible in the standard strong class
\[
\mathbf U\in H^2(B_1)^2\cap H_0^1(B_1)^2,\qquad
q\in H^1(B_1)\cap L_0^2(B_1),\qquad
\Pi\in H^2(B_1),\quad \partial_n\Pi=0,\quad \int_{B_1}\Pi\,\ud S=0.
\]
Write the full operator as \(\mathcal A=\mathcal A_{\rm tr}+\mathcal K\).
{\color{black}All terms in \(\mathcal K\) contain at most one derivative, and
their coefficients are smooth.  A bounded sequence in the strong class has a
subsequence converging in \(H^1(B_1)\); its image under \(\mathcal K\) then
converges in \(L^2(B_1)\).  The continuity of
\(\mathcal A_{\rm tr}^{-1}\) shows that
\(\mathcal A_{\rm tr}^{-1}\mathcal K\) is compact.}  Thus the full problem is equivalent to an
equation of the form \(I+\mathcal A_{\rm tr}^{-1}\mathcal K\).  Its kernel is
zero by Proposition~\ref{prop:linear-stability}, applied after taking the curl
of the velocity equation and by a standard density argument.  The classical
compact-perturbation alternative therefore gives the unique solvability of
\eqref{component-linear-problem}.
}

{\color{black}
\begin{lemma}[Compatibility of the nonlinear forcing]
\label{lem:neumann-compatibility}
For every divergence-free \(\overline{\mathbf U}\in H_0^1(B_1)^2\) and
\(\overline\Pi\in H^1(B_1)\),
\[
 \int_{B_1}\left(R_\Pi^a
 -\overline{\mathbf U}\cdot\nabla\overline\Pi\right)\,\ud S=0.
\]
\end{lemma}
\begin{proof}
Both transport integrals vanish by integration by parts.  Moreover,
\[
 \varepsilon^2\int_{B_1}\Delta\Pi^a\,\ud S
 =-2\pi\beta\varepsilon^2
 =-2\beta\varepsilon^2|B_1|,
\]
where the first equality follows from the boundary mean of \(g^a\).
The result now follows from
\(R_\Pi^a=2\beta\varepsilon^2
-\mathbf U^a\cdot\nabla\Pi^a+\varepsilon^2\Delta\Pi^a\).
\end{proof}
}

For polar data, choose $\mathbf f$ and $f_2$ so that
${-r\,\operatorname{curl}\mathbf f=F_1}$ and $rf_2=F_2$.
Applying
Proposition~\ref{prop:linear-stability} gives
\begin{align*}
\mathbb E+\mathbb P+\mathbb L
\leq \frac{C}{\varepsilon^2}\int_\Omega
\frac{F_1^2+F_2^2+F_{2,\theta}^2}{r}.
\end{align*}

For a pair $(\Phi,\Pi)$, set
\[
\|(\Phi,\Pi)\|_{\mathcal X}
=(\mathbb E+\mathbb P+\mathbb L)^{1/2}.
\]
\begin{lemma}[Component regularity for the linear problem]
\label{lem:component-linear-regularity}
{\color{black}Let \(\mathbf f\in H^1(B_1)^2\) and \(f_2\in H^1(B_1)\) with
\(\int_{B_1}f_2\,\ud S=0\).}  Let $(\Phi,\Pi)$ be the solution of \eqref{component-linear-problem}, and
write $\mathbf U=\nabla^\perp\Phi$ and $\mathbf H=\nabla^\perp\Pi$.  Then
\begin{align}\label{component-linear-regularity}
&\|(\Phi,\Pi)\|_{\mathcal X}+\varepsilon^3
\big(\|\mathbf U\|_{H^2(B_1)}+\|\mathbf H\|_{H^2(B_1)}\big)\nonumber\\
&\qquad\leq \frac{C}{\varepsilon}
\big(\|\mathbf f\|_{H^1(B_1)}+\|f_2\|_{H^1(B_1)}\big).
\end{align}
\end{lemma}
\begin{proof}
The stability estimate and the definition of $\mathbb E$ give
\[
\|(\Phi,\Pi)\|_{\mathcal X}\leq \frac C\varepsilon
\big(\|\mathbf f\|_{H^1}+\|f_2\|_{H^1}\big),
\qquad
\|\mathbf U\|_{H^1}+\|\mathbf H\|_{H^1}
\leq C\varepsilon^{-1}\|(\Phi,\Pi)\|_{\mathcal X}.
\]
Rewrite the scalar equation as
\[
-\varepsilon^2\Delta\Pi
=f_2-\mathbf U^a\cdot\nabla\Pi-\mathbf U\cdot\nabla\Pi^a.
\]
The normalized Neumann regularity estimate, first in \(H^2\) and then in
\(H^3\), gives
\begin{align*}
\varepsilon^2\|\Pi\|_{H^3}
&\leq C\Big(\|f_2\|_{H^1}
+\|\mathbf U^a\cdot\nabla\Pi
+\mathbf U\cdot\nabla\Pi^a\|_{H^1}\Big)\\
&\leq C\|f_2\|_{H^1}
+C\varepsilon^{-1}\|(\Phi,\Pi)\|_{\mathcal X}.
\end{align*}
Here the second line follows from the coefficient bounds and the boundary
Hardy inequality; no boundary differentiation of the Neumann condition is
used.  Since \(\mathbf H=\nabla^\perp\Pi\), this is the required
\(H^2\)-estimate for \(\mathbf H\).

Next apply the Dirichlet Stokes estimate to the first equation of
\eqref{component-linear-problem}.  The already obtained \(H^2\) control of
\(\mathbf H\), the \(H^1\) control from the stability norm, and the fixed
coefficient bounds yield
\[
\varepsilon^2\big(\|\mathbf U\|_{H^2}+\|q\|_{H^1}\big)
\leq C\big(\|\mathbf f\|_{L^2}+\|f_2\|_{H^1}\big)
+C\varepsilon^{-1}\|(\Phi,\Pi)\|_{\mathcal X}.
\]
Combining the two bounds proves
\eqref{component-linear-regularity}.
\end{proof}

\begin{lemma}[Quadratic forcing estimate]
\label{lem:quadratic-forcing}
For a pair $(\Phi,\Pi)$ define
\[
\| (\Phi,\Pi)\|_{\mathcal Y}
=\|(\Phi,\Pi)\|_{\mathcal X}+\varepsilon^3
\big(\|\nabla^\perp\Phi\|_{H^2}
+\|\nabla^\perp\Pi\|_{H^2}\big).
\]
Let
\[
\mathbf N_M(\mathbf U,\mathbf H)
=-(\mathbf U\cdot\nabla)\mathbf U
+(\mathbf H\cdot\nabla)\mathbf H,
\qquad
N_\Pi(\mathbf U,\Pi)=-\mathbf U\cdot\nabla\Pi.
\]
Then
\begin{equation*}
\|\mathbf N_M\|_{H^1}+\|N_\Pi\|_{H^1}
\leq C\varepsilon^{-6}\|(\Phi,\Pi)\|_{\mathcal Y}^2.
\end{equation*}
The corresponding difference satisfies the same estimate with one factor
replaced by the sum of the two input norms and the other by their difference.
\end{lemma}
\begin{proof}
In two dimensions, $H^2(B_1)$ embeds into $L^\infty(B_1)$ and
$W^{1,4}(B_1)$.  Therefore
\[
\|(\mathbf U\cdot\nabla)\mathbf V\|_{H^1}
\leq C\|\mathbf U\|_{H^2}\|\mathbf V\|_{H^2}.
\]
The same product estimate applies to
$\mathbf U\cdot\nabla\Pi$ because $\nabla\Pi$ is a rotation of
$\mathbf H$.  Since each physical $H^2$ norm is bounded by
$\varepsilon^{-3}\|(\Phi,\Pi)\|_{\mathcal Y}$, the result follows.
\end{proof}

\subsection{Nonlinear solvability}

We apply the contraction mapping theorem to \eqref{error-equation}.
The residual estimate required below is precisely Proposition~\ref{prop:residual-estimate}, in particular \eqref{remainder-estimate-in-polar-variables}.

\begin{proposition}\label{existence-and-error-estimate-of-error-equation}
	There exist $\varepsilon_0>0,\eta_0>0$ such that for any
	$\varepsilon\in (0,\varepsilon_0)$ and $\eta\in (0,\eta_0)$, the error
	equations \eqref{error-equation} have a solution \((\Phi,\Pi)\), unique in
		a ball \(\|(\Phi,\Pi)\|_{\mathcal Y}\leq M\varepsilon^9\) for a fixed
	\(M\), and
	satisfying
	\[
		\|(\Phi,\Pi)\|_{\mathcal X}\leq C\varepsilon^9.
	\]
	Thus the corresponding physical error
	\[
		(u,v,g,h)=\left(-\Phi_r,\frac{\Phi_\theta}{r},
		-\Pi_r,\frac{\Pi_\theta}{r}\right)
	\]
	satisfies $\|(u,v,g,h)\|_\infty\leq C\varepsilon$.
\end{proposition}
\begin{proof}
Let $\bar{\mathbf U}=\nabla^\perp\bar\Phi$ and
$\bar{\mathbf H}=\nabla^\perp\bar\Pi$.  For an input in the ball
\[
\|(\bar\Phi,\bar\Pi)\|_{\mathcal Y}\leq M\varepsilon^9,
\]
solve the linear component problem \eqref{component-linear-problem} with
\[
\mathbf f=-\mathbf R_M^a
-(\bar{\mathbf U}\cdot\nabla)\bar{\mathbf U}
+(\bar{\mathbf H}\cdot\nabla)\bar{\mathbf H},
\qquad
f_2=R_\Pi^a-\bar{\mathbf U}\cdot\nabla\bar\Pi,
\]
where $R_\Pi^a=R_g^a/r$ has the sign fixed in the definition above.  The
compatibility condition for \(f_2\) is exactly
Lemma~\ref{lem:neumann-compatibility}.  By \eqref{cartesian-residual-H1},
Lemma~\ref{lem:component-linear-regularity}, and
Lemma~\ref{lem:quadratic-forcing}, the output satisfies
\begin{align*}
\|(\Phi,\Pi)\|_{\mathcal Y}
&\leq \frac C\varepsilon
\left(\varepsilon^{10}
+\varepsilon^{-6}\|(\bar\Phi,\bar\Pi)\|_{\mathcal Y}^2\right)\\
&\leq C\varepsilon^9+CM^2\varepsilon^{11}.
\end{align*}
Choose $M\geq2C$ and then choose $\varepsilon_0$ so that
$CM^2\varepsilon_0^2\leq M/2$.  The map then preserves the ball.

For two inputs in the ball, the difference estimate in
Lemma~\ref{lem:quadratic-forcing} and the linear estimate give
\[
\|(\Phi_1-\Phi_2,\Pi_1-\Pi_2)\|_{\mathcal Y}
\leq CM\varepsilon^2
\|(\bar\Phi_1-\bar\Phi_2,
\bar\Pi_1-\bar\Pi_2)\|_{\mathcal Y}.
\]
After decreasing $\varepsilon_0$ once more, this is a strict contraction.
Its fixed point solves the original component equations and hence
\eqref{error-equation}.  It satisfies
\[
\|(\Phi,\Pi)\|_{\mathcal Y}\leq C\varepsilon^9,
\qquad
\|(\Phi,\Pi)\|_{\mathcal X}\leq C\varepsilon^9.
\]
Finally, the two-dimensional Sobolev embedding and the definition of
$\mathcal Y$ give
\[
\left\|\left(\Phi_r,\frac{\Phi_\theta}{r},
\Pi_r,\frac{\Pi_\theta}{r}\right)\right\|_\infty
\leq C\big(\|\mathbf U\|_{H^2}+\|\mathbf H\|_{H^2}\big)
\leq C\varepsilon^{-3}\|(\Phi,\Pi)\|_{\mathcal Y}
\leq C\varepsilon^6\leq C\varepsilon.
\]
\end{proof}

{
\begin{lemma}
\label{lem:interior-curl-convergence}
{\color{black}Let
\(\mathbf U=u\mathbf e_\theta+v\mathbf e_r\) and
\(\mathbf H=g\mathbf e_\theta+h\mathbf e_r\) be the velocity and magnetic
corrections obtained above.}  Fix
\[
0<r_0<\rho_1<\rho_2<1.
\]
Suppose
\[
\|\mathbf U\|_{L^\infty(B_1)}+\|\mathbf H\|_{L^\infty(B_1)}
\leq C\varepsilon,\qquad
\|\mathbf U\|_{H^1(B_1)}+\|\mathbf H\|_{H^1(B_1)}
\leq C\varepsilon^8,
\]
and
\[
\|\mathbf R_M^a\|_{H^2(B_{\rho_2})}
+\|R_\Pi^a\|_{H^2(B_{\rho_2})}
\leq C_{\rho_2}\varepsilon^{10}.
\]
Then
\begin{align}
\|\mathbf U\|_{H^2(B_{\rho_1})}
+\|\mathbf H\|_{H^2(B_{\rho_1})}
&\leq C_{\rho_1,\rho_2}\varepsilon^6,\label{eq:interior-H2}\\
\|\mathbf U\|_{H^3(B_{r_0})}
+\|\mathbf H\|_{H^3(B_{r_0})}
&\leq C_{r_0,\rho_1,\rho_2}\varepsilon^4.\label{eq:interior-H3}
\end{align}
In particular,
\[
\|\operatorname{curl}\mathbf U\|_{L^\infty(B_{r_0})}
+\|\operatorname{curl}\mathbf H\|_{L^\infty(B_{r_0})}
\leq C_{r_0,\rho_1,\rho_2}\varepsilon^4.
\]
\end{lemma}
\begin{proof}
On \(B_{\rho_2}\), the exact Cartesian equations are
\begin{align*}
-\kappa\varepsilon^2\Delta\mathbf U+\nabla q
&=\mathcal F_U,\qquad \nabla\cdot\mathbf U=0,\\
-\varepsilon^2\Delta\Pi
&=\mathcal F_\Pi,\qquad \mathbf H=\nabla^\perp\Pi.
\end{align*}
where
\begin{align*}
\mathcal F_U={}&-\mathbf R_M^a
-(\mathbf U^a\cdot\nabla)\mathbf U
-(\mathbf U\cdot\nabla)\mathbf U^a
+(\mathbf B^a\cdot\nabla)\mathbf H
+(\mathbf H\cdot\nabla)\mathbf B^a\\
&-(\mathbf U\cdot\nabla)\mathbf U
+(\mathbf H\cdot\nabla)\mathbf H,\\
\mathcal F_\Pi={}&R_\Pi^a
-\mathbf U^a\cdot\nabla\Pi
-\mathbf U\cdot\nabla\Pi^a
-\mathbf U\cdot\nabla\Pi.
\end{align*}
Equivalently, after applying \(\nabla^\perp\) to the scalar equation,
\begin{align*}
-\varepsilon^2\Delta\mathbf H=\mathcal F_H:={}&
\nabla^\perp R_\Pi^a
-(\mathbf U^a\cdot\nabla)\mathbf H
-(\mathbf U\cdot\nabla)\mathbf B^a
+(\mathbf B^a\cdot\nabla)\mathbf U\\
&+(\mathbf H\cdot\nabla)\mathbf U^a
-(\mathbf U\cdot\nabla)\mathbf H
+(\mathbf H\cdot\nabla)\mathbf U.
\end{align*}

{\color{black}The standard interior Stokes and Poisson estimates give}
\begin{align*}
&\varepsilon^2\big(
\|\mathbf U\|_{H^2(B_{\rho_1})}
+\|\mathbf H\|_{H^2(B_{\rho_1})}\big)\\
&\qquad\leq C_{\rho_1,\rho_2}\Big(
\|\mathcal F_U\|_{L^2(B_{\rho_2})}
+\|\mathcal F_H\|_{L^2(B_{\rho_2})}
+\varepsilon^2\|(\mathbf U,\mathbf H)\|_{H^1(B_{\rho_2})}\Big).
\end{align*}
The approximate coefficients and their fixed interior derivatives are
uniformly bounded.  Thus the assumed \(H^1\) and \(L^\infty\)
bounds imply
\begin{align*}
&\|\mathcal F_U\|_{L^2(B_{\rho_2})}
+\|\mathcal F_H\|_{L^2(B_{\rho_2})}\\
&\quad\leq C_{\rho_2}\Big(
\|(\mathbf U,\mathbf H)\|_{H^1(B_{\rho_2})}
+\|(\mathbf U,\mathbf H)\|_{L^\infty(B_{\rho_2})}
 \|(\mathbf U,\mathbf H)\|_{H^1(B_{\rho_2})}
+\varepsilon^{10}\Big)\\
&\quad\leq C_{\rho_2}\big(\varepsilon^8+\varepsilon^9
+\varepsilon^{10}\big)
\leq C_{\rho_2}\varepsilon^8.
\end{align*}
This proves \eqref{eq:interior-H2}.  On the fixed disk \(B_{\rho_1}\), the
approximate fields and the derivatives of their coefficients are uniformly
bounded.  Since \(H^2\) is an algebra in two dimensions, the preceding
estimate and the interior residual bound give
\[
\|\mathcal F_U\|_{H^1(B_{\rho_1})}
+\|\mathcal F_H\|_{H^1(B_{\rho_1})}
\leq C_{\rho_1,\rho_2}\varepsilon^6.
\]
Using a second cutoff supported in \(B_{\rho_1}\), equal to one on
\(B_{r_0}\), the interior \(H^3\) estimates for the Stokes and Poisson
equations yield
\begin{align*}
&\varepsilon^2\big(
\|\mathbf U\|_{H^3(B_{r_0})}
+\|\mathbf H\|_{H^3(B_{r_0})}\big)\\
&\qquad\leq C_{r_0,\rho_1}\Big(
\|\mathcal F_U\|_{H^1(B_{\rho_1})}
+\|\mathcal F_H\|_{H^1(B_{\rho_1})}
+\varepsilon^2\|(\mathbf U,\mathbf H)\|_{H^2(B_{\rho_1})}\Big)
\leq C_{r_0,\rho_1,\rho_2}\varepsilon^6.
\end{align*}
This proves \eqref{eq:interior-H3}.  Finally,
\(H^2(B_{r_0})\hookrightarrow L^\infty(B_{r_0})\), applied to the first
derivatives of \(\mathbf U\) and \(\mathbf H\), gives the curl estimate.
\end{proof}
}

\section{Proof of Theorem \ref{main-theorem}}
Finally, we give the proof of Theorem \ref{main-theorem}.
\begin{proof}
	Let $(\Phi,\Pi)$ be the solution obtained in Proposition
	\ref{existence-and-error-estimate-of-error-equation}, and set
	\[
		(u,v,g,h)=\left(-\Phi_r,\frac{\Phi_\theta}{r},
		-\Pi_r,\frac{\Pi_\theta}{r}\right).
	\]
	Then
	\[
		(u^\varepsilon,v^\varepsilon,g^\varepsilon,h^\varepsilon)=(u^a,v^a,g^a,h^a)+(u,v,g,h).
	\]
	Moreover,
	\[
		u^\varepsilon_\theta+rv^\varepsilon_r+v^\varepsilon
		=g^\varepsilon_\theta+rh^\varepsilon_r+h^\varepsilon=0.
	\]
	For this solution,
	\begin{align*}
		F_1=&\ {\partial_rR_u^a-\frac1r\partial_\theta R_v^a}
		-\Phi_\theta(\Delta\Phi)_r+\Phi_r(\Delta\Phi)_\theta
		+\Pi_\theta(\Delta\Pi)_r-\Pi_r(\Delta\Pi)_\theta,\\
		F_2=&\ R_g^a+\Phi_r\Pi_\theta-\Phi_\theta\Pi_r,
	\end{align*}
	so \eqref{error-equation} cancels the residual of the approximate
	solution.  Also,
	\[
		\Phi(\theta,1)=\Phi_r(\theta,1)=0,\qquad
		\Pi_r(\theta,1)=0
	\]
	are exactly the required boundary conditions.  The fixed point was
	constructed through \eqref{component-linear-problem}, so it satisfies the
	full momentum equation with the associated normalized Stokes pressure.  The
	second error equation restores the scalar magnetic-flux equation with the constant
	\(2\beta\varepsilon^2\); applying \(\nabla^\perp\) and using
	\eqref{eq:magnetic-flux-constant} gives the full vector induction equation.
	Proposition~\ref{existence-and-error-estimate-of-error-equation} and
	\eqref{approximate-solution} give the stated convergence of the velocity and
	magnetic fields.  It remains to prove convergence of their curls.  Fix
\[
0<r_0<\rho_1<\rho_2<1.
\]
On $B_{\rho_2}$ the stretched coordinate satisfies
$Y\leq-(1-\rho_2)/\varepsilon$. {Hence the decaying Prandtl parts and their
first derivatives are $O(\varepsilon^N)$ there for arbitrary $N$.
The possible constant limit of $\varepsilon^{11}g_p^{(11)}$ is
$O(\varepsilon^{11})$, while the top normal traces and their correctors are
$O(\varepsilon^{12})$; both are smaller than the orders used below.} The
smooth Euler profiles therefore satisfy
\[
\|\operatorname{curl}(u_e^a,v_e^a)-2a\|_{L^\infty(B_{\rho_2})}
+\|\operatorname{curl}(g_e^a,h_e^a)-2b\|_{L^\infty(B_{\rho_2})}
\leq C_{\rho_2}\varepsilon\eta.
\]

{\color{black}Let
\(\mathbf U=u\mathbf e_\theta+v\mathbf e_r\) and
\(\mathbf H=g\mathbf e_\theta+h\mathbf e_r\) be the physical error fields.}
From the definition of \(\mathcal X\), the boundary conditions, the
Fourier-mode Poincar\'e inequalities, and
Proposition~\ref{existence-and-error-estimate-of-error-equation},
\begin{equation}\label{eq:first-interior-error-bounds}
\|\mathbf U\|_{L^\infty}+\|\mathbf H\|_{L^\infty}\leq C\varepsilon,
\qquad
\|\mathbf U\|_{H^1(B_1)}+\|\mathbf H\|_{H^1(B_1)}\leq C\varepsilon^8.
\end{equation}
For the zero Fourier modes, the \(L^2\) part of the second bound follows by
integrating radially from \(r=1\), where the tangential errors vanish; for the
nonzero modes it follows from Poincar\'e's inequality in \(\theta\).  The
derivative part follows from the Hessian terms in \(\mathbb E\) and regularity
at \(r=0\).  {Hence Lemma~\ref{lem:interior-curl-convergence}, with
\eqref{cartesian-residual-interior-H2} and
\eqref{eq:first-interior-error-bounds}, gives
\[
\|\operatorname{curl}\mathbf U\|_{L^\infty(B_{r_0})}
+\|\operatorname{curl}\mathbf H\|_{L^\infty(B_{r_0})}
\leq C_{r_0,\rho_1,\rho_2}\varepsilon^4.
\]}
Combining the Euler, Prandtl, and error contributions yields
\[
\|\omega^\varepsilon-2a\|_{L^\infty(B_{r_0})}
+\|j^\varepsilon-2b\|_{L^\infty(B_{r_0})}\longrightarrow0,
\]
which completes the proof of Theorem~\ref{main-theorem}.
\end{proof}

\appendix

\section{The velocity divergence corrector}\label{app:correctors}
Recall that $\mathcal I$ denotes the zero-mean periodic primitive
in $\theta$.
For the velocity expansion, direct calculation gives
\[
(u_e^a+u_p^a)_\theta+r(v_e^a+v_p^a)_r+(v_e^a+v_p^a)=K_u,
\]
where
\[
K_u=\varepsilon^{12}\chi(r)(Y\partial_Yv_p^{(12)}+v_p^{(12)})
+r\chi'(r)\sum_{i=1}^{12}\varepsilon^iv_p^{(i)}.
\]
The profile divergence equations imply that $K_u$ has zero $\theta$-mean.
Moreover, $K_u(\theta,1)=0$ because $v_p^{(12)}(\theta,0)=0$, and
$K_u=O(\eta\varepsilon^{12})$.  Define
\[
u_L=-\varepsilon^{-12}\mathcal I K_u.
\]
Then $u_L(\theta,1)=0$,
\[
\varepsilon^{12}\partial_\theta u_L+K_u=0,
\]
and
\[
\|\partial_\theta^j\partial_r^ku_L\|_2
\leq C(j,k)\eta\varepsilon^{-k}.
\]
This is the velocity corrector used in \eqref{approximate-solution}.

For comparison, the cutoff magnetic profiles have the defect
\[
K_g=\varepsilon^{12}\chi(r)(Y\partial_Yh_p^{(12)}+h_p^{(12)})
+r\chi'(r)\sum_{i=1}^{12}\varepsilon^ih_p^{(i)}.
\]
Unlike the velocity, the normal magnetic component has no prescribed
boundary value.  We therefore do not introduce a tangential corrector.
Instead, defining the magnetic potential from $g^a$ gives exactly the normal
correction \eqref{magnetic-normal-correction}.  Writing
$h_p^{(12)}=H_{12,\infty}+\widetilde h_p^{(12)}$ in the displayed formula
gives
\[
K_g=\varepsilon^{12}\partial_r(r\chi H_{12,\infty})+\widetilde K_g,
\]
where every term in $\widetilde K_g$ contains a rapidly decaying Prandtl
profile.  Its radial primitive therefore gains one power of $\varepsilon$,
which proves \eqref{magnetic-normal-correction-estimate}.

\section{A conditional MHD Prandtl--Batchelor selection principle}
\label{app:mhd-pb-selection}
This appendix gives a conditional selection argument using the magnetic
potential and the velocity stream function.  It does not require convergence
of second derivatives of the velocity or magnetic fields and assumes no
interior expansion in powers of \(\varepsilon\).

\begin{lemma}[Connected levels and smooth composition]
\label{lem:single-eddy-levels}
Let \(\Phi\in C^\infty(\overline{B_1})\) be constant on
\(\partial B_1\), have no boundary critical point, and have exactly one
critical point in \(B_1\), a nondegenerate extremum.  Then every regular
level of \(\Phi\) is connected.  If
\(\Pi\in C^\infty(\overline{B_1})\) is constant on those levels, then
\(\Pi=\mathcal{F}(\Phi)\) for a function \(\mathcal{F}\) smooth on the closed
range of \(\Phi\).
\end{lemma}
\begin{proof}
Each regular component is a simple closed curve.  If two components are
disjoint, the disk enclosed by each component contains an interior extremum
of \(\Phi\), hence a critical point.  If they are nested, the closed annulus
between them has the same value on both boundary components; unless
\(\Phi\) is constant there, its maximum or minimum is attained in the
interior and is critical.  Both alternatives contradict uniqueness.

The composition \(\Pi=\mathcal{F}(\Phi)\) is smooth on regular level bands by
the implicit-function theorem.  Near the unique critical point \(x_*\), the
Morse lemma gives local coordinates centered at \(x_*\) in which
\(\Phi=\Phi(x_*)\pm(x_1^2+x_2^2)\).  A smooth function constant on the
corresponding circles is a smooth function of \(x_1^2+x_2^2\), so
\(\mathcal{F}\) extends smoothly through the critical value.  The boundary
value follows in the same way from the collar given by
\(\nabla\Phi\ne0\).
\end{proof}

\begin{proposition}[Conditional single-eddy rigidity]
Let \((\mathbf u^\varepsilon,\mathbf B^\varepsilon)\) be smooth
viscous--resistive MHD solutions in \(B_1\), with
\(\mathbf u^\varepsilon\cdot\mathbf n=0\) on the boundary and with the
magnetic circulation prescribed in \eqref{magnetic-boundary-condition}.
Assume a uniform \(L^\infty\) bound and
\[
(\mathbf u^\varepsilon,\mathbf B^\varepsilon)
\longrightarrow(\mathbf u_e,\mathbf B_e)
\quad\text{in }C^1_{\mathrm{loc}}(B_1).
\]
Suppose the smooth ideal limit is a single eddy: its normalized stream
function is constant on the boundary, has exactly one nondegenerate interior
critical point, and no other critical point in \(\overline{B_1}\).
{\color{black}Suppose also that the proportionality function \(m\), determined by
\(\mathbf B_e=m(\Phi_e)\mathbf u_e\), satisfies
\(\inf_{\overline{B_1}}|1-m(\Phi_e)^2|>0\).}  Then
\[
\mathbf u_e=ar\,\mathbf e_\theta,\qquad
\mathbf B_e=\beta r\,\mathbf e_\theta
\]
for a constant \(a\).
\end{proposition}
\begin{proof}

Let
\[
\omega^\varepsilon=\operatorname{curl}\mathbf u^\varepsilon,
\qquad
j^\varepsilon=\operatorname{curl}\mathbf B^\varepsilon,
\qquad
\mathbf u^\varepsilon=\nabla^\perp\Phi^\varepsilon,
\qquad
\mathbf B^\varepsilon=\nabla^\perp\Pi^\varepsilon.
\]
Normalize the velocity stream functions by
\(\Phi^\varepsilon=\Phi_e=0\) on \(\partial B_1\).  By hypothesis,
\[
\sup_{\varepsilon>0}
\big(\|\mathbf u^\varepsilon\|_{L^\infty(B_1)}
+\|\mathbf B^\varepsilon\|_{L^\infty(B_1)}\big)<\infty.
\]
The stream functions are uniformly Lipschitz:
\(\|\nabla\Phi^\varepsilon\|_\infty\leq M\).  Given \(\delta>0\), their
common boundary normalization gives
\[
\sup_{\{1-\delta<|x|<1\}}|\Phi^\varepsilon(x)|
\leq M\delta,
\qquad
\sup_{\{1-\delta<|x|<1\}}|\Phi_e(x)|\leq M\delta.
\]
On \(B_{1-\delta}\), local convergence of the velocities and one fixed
normalization point give uniform convergence of the stream functions.
First choosing \(\delta\) small and then \(\varepsilon\) small proves
\[
\Phi^\varepsilon\longrightarrow\Phi_e
\quad\hbox{uniformly on }\overline{B_1}.
\]
Choose the orientation so that \(\Phi_e>0\) in \(B_1\); if the opposite
orientation occurs, replace \(\Phi_e\) by \(-\Phi_e\).
Lemma~\ref{lem:single-eddy-levels} shows that every regular level is
connected.
Write \(\mathbf B_e=\nabla^\perp\Pi_e\).

{After fixing additive constants, local \(C^1\) convergence of the
fields gives local \(C^2\) convergence of their potentials.  In the weak
formulation the diffusion terms vanish as \(\varepsilon\to0\), so the smooth
limit satisfies ideal MHD without convergence of second derivatives of the
fields.}

The limiting magnetic-potential equation gives
\[
\mathbf u_e\cdot\nabla\Pi_e=0.
\]
The connectedness of the streamlines therefore implies
\[
\Pi_e=\mathcal F(\Phi_e),
\qquad
\mathbf B_e=m(\Phi_e)\mathbf u_e,
\qquad m=\mathcal F'.
\]
The smooth-composition conclusion of
Lemma~\ref{lem:single-eddy-levels} shows that \(\mathcal{F}\), \(m\), and the
compositions below extend smoothly through the center.  The uniform
non-Alfv\'enic hypothesis is
\begin{equation}\label{eq:appendix-uniform-nonalfven}
m(\tau)\ne\pm1
\quad\hbox{for every }\tau\hbox{ in the range of }\Phi_e.
\end{equation}
By continuity and compactness, this is equivalent to a positive lower bound
for \(|1-m(\Phi_e)^2|\).
Taking the curl of the ideal momentum equation gives
\[
\mathbf u_e\cdot\nabla\omega_e
-\mathbf B_e\cdot\nabla j_e=0.
\]
Since \(m=m(\Phi_e)\), one has
\(\mathbf u_e\cdot\nabla m=0\).  Hence
\[
\mathbf u_e\cdot\nabla(\omega_e-mj_e)=0,
\qquad
\omega_e-mj_e=K(\Phi_e)
\]
for a function \(K\) on the range of \(\Phi_e\).  The same level-set argument
shows that \(K\) is smooth, including at the center.
Moreover,
\[
j_e=-\Delta\mathcal F(\Phi_e)
=m\omega_e-m'|\nabla\Phi_e|^2,
\]
and so
\begin{equation}\label{eq:appendix-transformed-balance}
(1-m^2)\omega_e+mm'|\nabla\Phi_e|^2=K(\Phi_e).
\end{equation}

The sign
\(\sigma=\operatorname{sgn}(1-m^2)\) is constant by
\eqref{eq:appendix-uniform-nonalfven}.  Define the strictly increasing
function
\[
H(\tau)=\int_{0}^{\tau}\sqrt{|1-m(s)^2|}\,\ud s,
\qquad
\Xi=H(\Phi_e).
\]
Set \(\lambda(\tau)=H'(\tau)=\sqrt{|1-m(\tau)^2|}\).  Then
\[
H''(\tau)=-\frac{\sigma m(\tau)m'(\tau)}{\lambda(\tau)}.
\]
Using \(\omega_e=-\Delta\Phi_e\) and
\eqref{eq:appendix-transformed-balance}, we obtain
\begin{align*}
-\Delta\Xi
&=\lambda(\Phi_e)\omega_e-H''(\Phi_e)|\nabla\Phi_e|^2\\
&=\frac{\sigma}{\lambda(\Phi_e)}
\left((1-m^2)\omega_e+mm'|\nabla\Phi_e|^2\right)
=\frac{\sigma K(\Phi_e)}{\lambda(\Phi_e)}.
\end{align*}
Since \(H\) is invertible, the right-hand side is a smooth function of
\(\Xi\) alone.  Thus \(\Xi\in C^2(\overline{B_1})\) satisfies a semilinear
Dirichlet equation
\[
-\Delta\Xi=\mathcal N(\Xi).
\]
The function \(\mathcal N\) is smooth, hence locally Lipschitz, on the compact range
of \(\Xi\).  Since \(H(0)=0\), \(H'>0\), and \(\Phi_e>0\) in the disk,
\(\Xi\) is positive in the disk and vanishes on the boundary.  The radial
symmetry theorem of Gidas, Ni, and Nirenberg \cite{GNN} shows that \(\Xi\) is
radial and strictly monotone away from the center.  Since \(H\) is strictly
increasing, \(\Phi_e\), \(\Pi_e\), \(\omega_e\), and \(j_e\) are radial as
well.

It remains to determine the two radial curls.  Let \(I\) be the open interval
between the boundary value and the center value of \(\Phi_e\), and take
\(\chi\in C_c^\infty(I)\).  Uniform convergence of the normalized stream
functions implies that \(\chi(\Phi^\varepsilon)\) is supported, for all
sufficiently small \(\varepsilon\), in a fixed compact annulus separated
from both the wall and the center.  Thus every integration by parts below
has no boundary contribution from the support of the test.  Moreover, the
assumed local \(C^1\) convergence of the fields gives local uniform
convergence
\[
\omega^\varepsilon\to\omega_e,\qquad
j^\varepsilon\to j_e
\]
on that annulus.  Since \(-\Delta\Pi^\varepsilon=j^\varepsilon\),
the exact magnetic-potential equation is
\[
\mathbf u^\varepsilon\cdot\nabla\Pi^\varepsilon
+\varepsilon^2j^\varepsilon=2\beta\varepsilon^2.
\]
Multiplying it by \(\chi(\Phi^\varepsilon)\) and integrating gives
\begin{equation}\label{eq:appendix-current-weak}
\int_{B_1}(j^\varepsilon-2\beta)
\chi(\Phi^\varepsilon)\,\ud x=0.
\end{equation}
Indeed, the transport term is the integral of
\(\operatorname{div}(\Pi^\varepsilon
\chi(\Phi^\varepsilon)\mathbf u^\varepsilon)\), because
\(\mathbf u^\varepsilon\cdot\nabla\Phi^\varepsilon=0\), and its boundary
integral vanishes.  Passing to the limit in
\eqref{eq:appendix-current-weak} gives
\[
\int_{B_1}(j_e-2\beta)\chi(\Phi_e)\,\ud x=0.
\]
The strict radial monotonicity of \(\Phi_e\), followed by the change of
variable \(\tau=\Phi_e(r)\), now gives
\begin{equation}\label{eq:appendix-constant-current}
j_e(r)=2\beta,
\qquad 0<r<1.
\end{equation}

Next, multiply
\[
\mathbf u^\varepsilon\cdot\nabla\omega^\varepsilon
-\mathbf B^\varepsilon\cdot\nabla j^\varepsilon
-\kappa\varepsilon^2\Delta\omega^\varepsilon=0
\]
by \(\chi(\Phi^\varepsilon)\).  The magnetic-potential equation and the
identity
\[
\mathbf B^\varepsilon\cdot\nabla\Phi^\varepsilon
=\varepsilon^2(j^\varepsilon-2\beta)
\]
give, after integration by parts and division by \(\varepsilon^2\),
\begin{equation}\label{eq:appendix-vorticity-weak}
\begin{aligned}
0={}&\int_{B_1}
\big[j^\varepsilon(j^\varepsilon-2\beta)
+\kappa(\omega^\varepsilon)^2\big]
\chi'(\Phi^\varepsilon)\,\ud x\\
&-\kappa\int_{B_1}\omega^\varepsilon
|\nabla\Phi^\varepsilon|^2
\chi''(\Phi^\varepsilon)\,\ud x.
\end{aligned}
\end{equation}
Here the viscous term was integrated by parts twice; thus no convergence of
\(\Delta\omega^\varepsilon\) is used.  The local \(C^1\) convergence of the
fields gives local uniform convergence of \(\omega^\varepsilon\) and
\(j^\varepsilon\), so all terms in \eqref{eq:appendix-vorticity-weak} pass to
the limit.  By \eqref{eq:appendix-constant-current}, the result is
\[
\int_{B_1}\nabla\omega_e\cdot
\nabla\chi(\Phi_e)\,\ud x=0.
\]
In radial variables this becomes
\[
2\pi\int_0^1 r\omega_e'(r)\Phi_e'(r)
\chi'(\Phi_e(r))\,\ud r=0.
\]
Changing variables from \(r\) to \(\tau=\Phi_e(r)\) shows that
\(r\omega_e'(r)\) is constant on \((0,1)\).  Smoothness at the center makes
this constant zero.  Hence \(\omega_e\) is constant on \(0<r<1\).

The present sign convention and smoothness at the center now give
\[
\mathbf u_e=\frac{\omega_e}{2}r\,\mathbf e_\theta,
\qquad
\mathbf B_e=\frac{j_e}{2}r\,\mathbf e_\theta
=\beta r\,\mathbf e_\theta.
\]
By continuity, the current and vorticity identities extend to \(r=0\) and
\(r=1\).  Thus \(a=\omega_e/2\) and \(b=\beta\).  The single-eddy, uniform
non-Alfv\'enic, uniform boundedness, and local \(C^1\) convergence assumptions
used here are not assumed in Theorem~\ref{main-theorem}; the theorem instead
constructs a family of solutions with the selected core.
\end{proof}

\section*{Acknowledgments}

Gao's research is supported in part by the National Natural Science
Foundation of China under Grant 12494541.
Lin's research is supported in part by the National Natural Science
Foundation of China under Grant 12494544.
C. Gao thanks Professor Zhouping Xin for helpful discussions on the boundary
conditions for the magnetic field.

\end{document}